\documentclass[11pt]{amsart}

\usepackage[margin=1in]{geometry}
\usepackage{amssymb,amsmath,bbm,tikz}
\usepackage{graphicx}
\usepackage[backref=page]{hyperref}
\usepackage[width=12cm,format=plain,labelfont=sf,bf,small,textfont=sf,up,small,justification=justified,labelsep=period]{caption}
\usepackage{subcaption}
\usepackage{xstring}
\usepackage{wrapfig}
\usepackage[normalem]{ulem}
\usetikzlibrary{shapes}
\graphicspath{{}{../hamzaNotes/figures/}}

\hypersetup{
    colorlinks=true,       
    linkcolor=blue,          
    citecolor=blue,        
    filecolor=blue,      
    urlcolor=blue           
}

\makeatletter
 \renewcommand*\l@subsection{\@tocline{2}{0em}{2.8em}{6.4em}{}}
\makeatother

\newcommand{\RR}{\mathbb{R}}

\newcommand{\CC}{\mathbb{C}}
\newcommand{\NN}{\mathbb{N}}

\newcommand{\cA}{\mathcal{A}}

\newcommand{\cS}{\mathcal{S}}

\newcommand{\cH}{\mathcal{H}}
\newcommand{\cV}{\mathcal{V}}
\renewcommand{\S}{\cS}

\newcommand{\PSD}{\mathcal{S}_+}
\newcommand{\psdcone}[1]{\cS_+^{#1}}

\newcommand{\rank}{\textup{rank}\,}

\newcommand{\psd}{\succeq}

\newcommand{\conv}{\textup{conv}}

\newcommand{\Z}{\mathcal{Z}} 

\DeclareMathOperator{\cl}{cl}
\DeclareMathOperator{\linspan}{span}
\DeclareMathOperator{\interior}{int}
\DeclareMathOperator{\mindeg}{mindeg}

\DeclareMathOperator{\Aut}{Aut}

\newcommand{\tr}{\textup{tr}}

\newcommand{\epi}{\textup{epi}}

\newcommand{\dom}{\textup{dom}}
\newcommand{\ext}{\textup{ext}}
\newcommand{\intt}{\textup{int}}

\newcommand{\ff}{F}

\newcommand{\Horn}{\textsc{horn}}
\newcommand{\Honey}{\textsc{honey}}

\newtheorem{theorem}{Theorem}[section]

\newtheorem{lemma}[theorem]{Lemma}
\newtheorem{corollary}[theorem]{Corollary}
\newtheorem{proposition}[theorem]{Proposition}
\theoremstyle{definition}
\newtheorem{definition}[theorem]{Definition}
\newtheorem{example}[theorem]{Example}

\theoremstyle{remark}
\newtheorem{remark}[theorem]{Remark}

\title{Lifting for Simplicity: Concise Descriptions of Convex Sets}

\author[H. Fawzi]{Hamza Fawzi} \address{Department of Applied Mathematics and Theoretical Physics, University of Cambridge, CB3 0WA, United Kingdom} \email{h.fawzi@damtp.cam.ac.uk}

\author[J. Gouveia]{Jo{\~a}o Gouveia}
\address{CMUC, Department of Mathematics,
  University of Coimbra, 3001-454 Coimbra, Portugal}
\email{jgouveia@mat.uc.pt}

\author[P.A. Parrilo]{Pablo A. Parrilo}\address{Laboratory for
  Information and Decision Systems (LIDS), Massachusetts Institute of
  Technology, Cambridge, MA 02139, USA} \email{parrilo@mit.edu}

\author[J. Saunderson]{James Saunderson}
\address{Department of Electrical and Computer Systems Engineering, Monash University, VIC 3800, Australia}
\email{james.saunderson@monash.edu}

\author[R.R. Thomas]{Rekha R. Thomas}
\address{Department of Mathematics, University of Washington, Box
  354350, Seattle, WA 98195, USA} \email{rrthomas@uw.edu}

  \thanks{Jo{\~a}o Gouveia was partially supported by the Centre for
Mathematics of the University of Coimbra - UIDB/00324/2020, funded by the Portuguese Government through FCT/MCTES. Pablo Parrilo was partially supported by the National Science Foundation through NSF Grant {\#}CCF-1565235. James Saunderson 
was partially supported by an Australian Research Council Discovery Early Career Researcher Award (project number DE210101056). Rekha Thomas was partially suported by the National Science Foundation through NSF Grant {\#}DMS-1719538. 
}

\begin{document}

\begin{abstract}
	This paper presents a selected tour through the theory and applications of
	lifts of convex sets. A lift of a convex set is a higher-dimensional
	convex set that projects onto the original set. Many convex sets have
	lifts that are dramatically simpler to describe than the original set.
	Finding such simple lifts has significant algorithmic implications, particularly for
	optimization problems. 
        We consider both the classical
	case of polyhedral lifts, described by linear inequalities, as well as
	spectrahedral lifts, defined by linear matrix inequalities, with a
	focus on recent developments related to spectrahedral lifts.

	Given a convex set, ideally we would either like to find a (low-complexity)
	polyhedral or spectrahedral lift, or find an obstruction proving that
	no such lift is possible. To this end, we explain the connection
	between the existence of lifts of a convex set and certain structured
	factorizations of its associated slack operator. Based on this
	characterization, we describe a uniform approach, via sums of squares,
	to the construction of spectrahedral lifts of convex sets and illustrate
	the method on several families of examples. Finally, we discuss
	two flavors of obstruction to the existence of lifts:
	one related to facial structure, and the other related to
	algebraic properties of the set in question.

	Rather than being exhaustive, our aim is to illustrate the richness of the area. 
	We touch on a range of different topics related to the existence of
	lifts, and present many examples of lifts from different 
	areas of mathematics and its applications.
\end{abstract}
\maketitle 

%

\tableofcontents

\section{Introduction}
\label{sec:introduction}

\begin{wrapfigure}{R}{0.25\textwidth}
  \centering
  \includegraphics[width=0.18\textwidth]{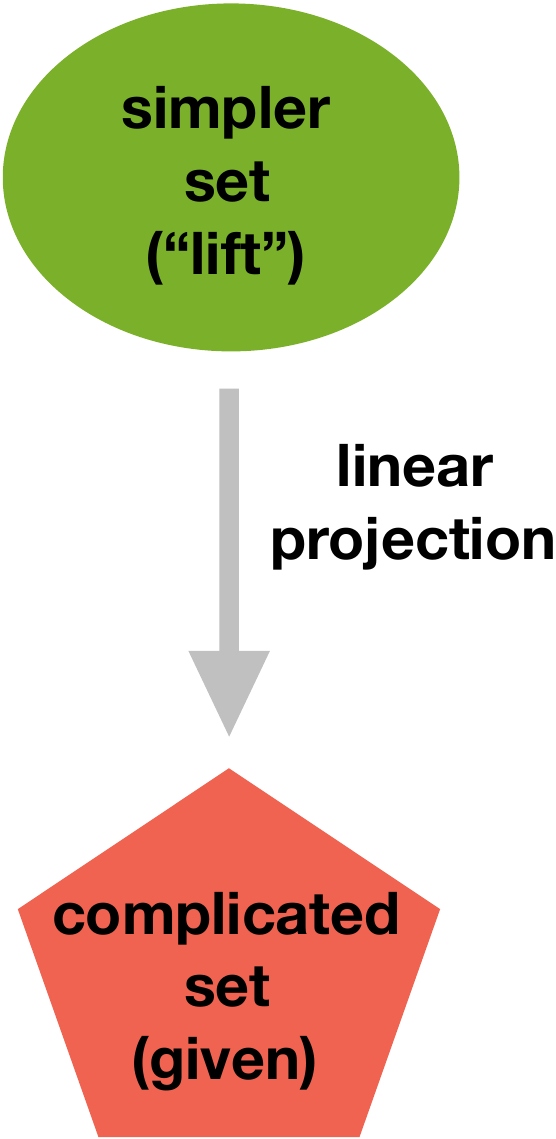}
  \caption{A complicated convex set might admit a concise ``lift'', which is a higher-dimensional convex set with a simpler description that projects to it.}
  \label{fig:cartoon}
\end{wrapfigure}

The representation of a convex set in $\RR^n$ is a crucial aspect of algorithmic primitives for convex sets, such as testing membership, computing volume, or optimization. The study of representations of \emph{convex} sets is of particular importance in the context of optimization, although it is also of interest well beyond this setting. An established idea for representing a convex set, which has received much attention recently, is to express it as the linear projection of another, ``simpler'' convex set.

This idea is very powerful --- there are by now many 
examples of convex sets that require highly complicated descriptions in their ambient space, 
but for which one can find a convex set in a higher-dimensional space, with a concise description, that projects to it. The cartoon in Figure~\ref{fig:cartoon} illustrates the idea.

Such concise ``lifts'' can lead to  
considerable algorithmic efficiencies for applications such as linear optimization over the original set. They effectively exploit the non-uniqueness of possible preimages of a projection, to find a more computationally convenient description. 
This article will take the reader on a selective tour of key ideas, examples, and applications related to representations of convex sets as projections. 

\subsection{Why lifts? Two simple examples} To motivate our topic, consider {\em linear programming} which is the problem of optimizing a linear function over a polytope. 
The number of linear inequalities needed to represent the polytope can be thought of as a measure of its complexity, as 
it directly affects the efficiency of certain algorithms for linear programming over the polytope. 
The obvious linear inequality description of a
polytope $P \subset \RR^n$ requires one inequality for every facet of $P$ (a codimension-one face). However, the number of
facets of $P$ may be very large.  In such situations,
one might ask if there is a
more compact representation of $P$.
One possible way out is to move to a higher dimension and write  $P$ as the 
linear projection of a polytope $Q \subset \RR^\ell$ where $\ell > n$. Such a $Q$ is called a {\em lift} of $P$. 
If in addition, both $\ell$ and the number of facets of $Q$ are small in size, say polynomial in $n$, then $Q$ is referred to as a 
{\em small/compact} lift of~$P$.  
Such a small lift can be very useful in optimization, since any linear optimization problem over $P$ can be solved via linear optimization over $Q$. Indeed if $P = \pi(Q)$ where $\pi$ is a linear map, then for any cost vector $c$ we have:
\[
\max_{x \in P} \; \langle c , x \rangle = \max_{y \in Q} \; \langle c , \pi(y) \rangle.
\]
The right-hand side is a linear optimization problem over $Q$, which can potentially be much easier to solve.

\begin{figure}[ht]
  \centering
  \includegraphics[width=4cm]{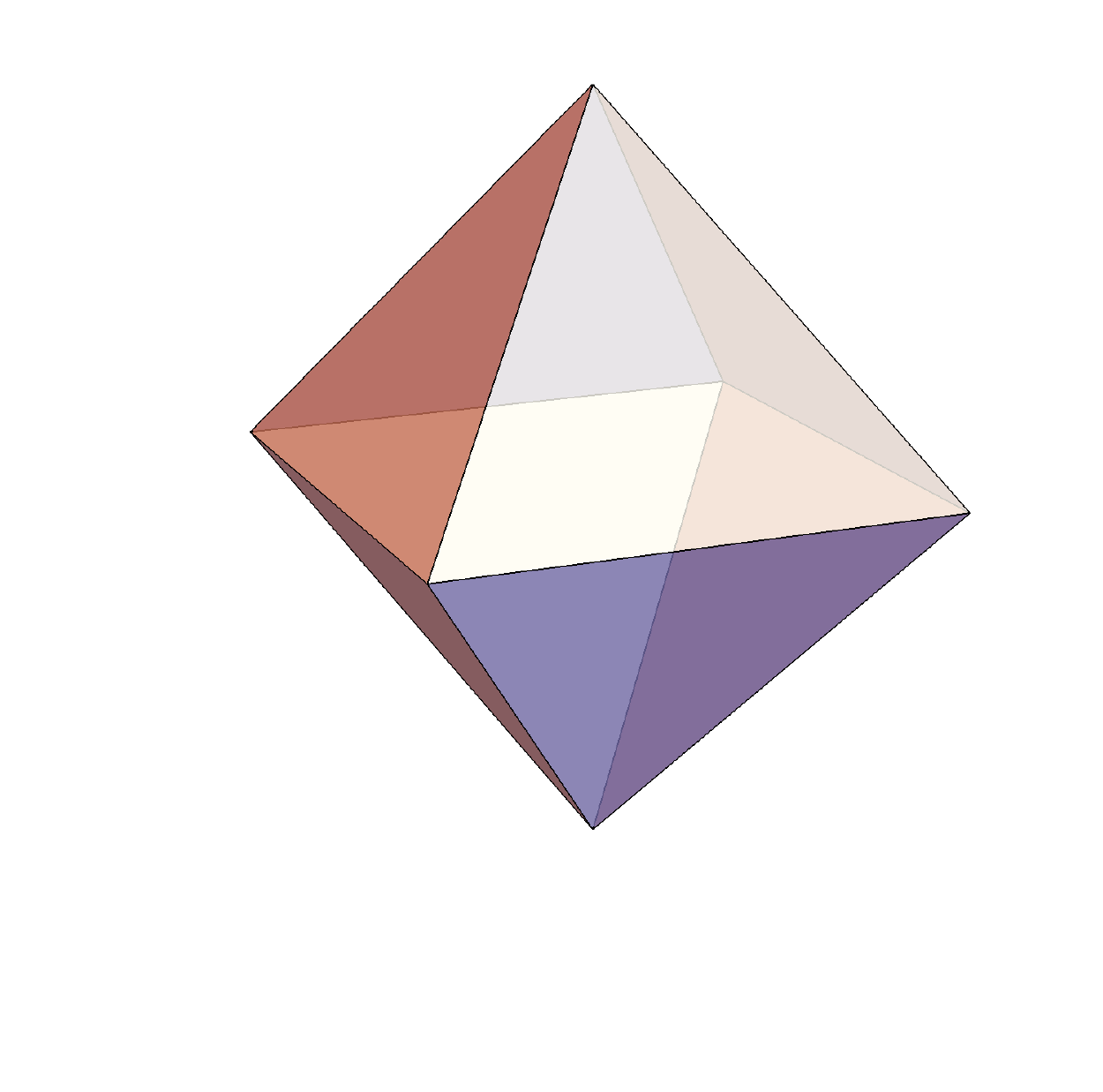}
  \caption{Cross-polytope $C_3$}
  \label{fig:octahedron}
\end{figure}

\begin{example} \label{ex:crosspolytope}
Consider the $\ell_1$-norm ball in $\RR^n$  which is the  
$n$-dimensional {\em cross-polytope} defined as the convex hull of the $2n$ vectors $\pm e_i, \,\, i=1, \ldots, n$ where $e_i$ is the 
$i$th standard basis vector in $\RR^n$.  
This polytope has $2^n$ facets, and its inequality description is 
$C_n = \{ x \in \RR^n \,:\, \sum_{i=1}^n \pm x_i \leq 1 \}$. 
The three-dimensional cross-polytope $C_3$ is shown in Figure~\ref{fig:octahedron}. 

Let  $\pi_x$ denote the projection from $\RR^{2n}$ to $\RR^n$ sending $(x,y) \mapsto x$. 
Then $C_n = \pi_x(Q_n)$ where 
$Q_n = \{ (x,y) \in \RR^{2n} \,:\,  \sum_{i=1}^n  y_i = 1, \,\,\, -y_i \leq x_i \leq y_i \,\,\forall \, i=1, \ldots, n  \}$. 
Note that $Q_n$ has only $2n$ facets, in contrast to the exponentially many of $C_n$.

\end{example}

Even though the cross-polytope has $2^n$ facets, it has only $2n$ vertices. So it is perhaps not too surprising that it has a small lift. Next we consider an example of a polytope that has both exponentially many facets and vertices. 

\begin{example} \label{ex:permutahedron}
The {\em permutahedron} $\Pi_n \subset \RR^n$ is the convex hull of the $n!$ vectors obtained by permuting the coordinates of $(1,2,\ldots,n)$. 
This is an $(n-1)$-dimensional polytope with $2^{n}-2$ facets, one corresponding to each proper subset of $[n]$. Figure~\ref{fig:perm} shows 
$\Pi_3$ and $\Pi_4$.

\begin{figure}[ht]
	\begin{center}
		\raisebox{-0.5\height}{\includegraphics[scale=1]{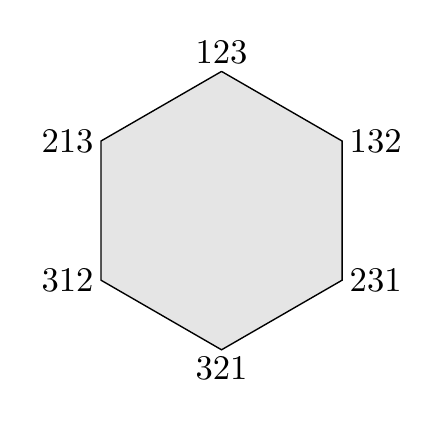}}
		\hspace{2cm}\raisebox{-0.5\height}{\includegraphics[scale=0.6]{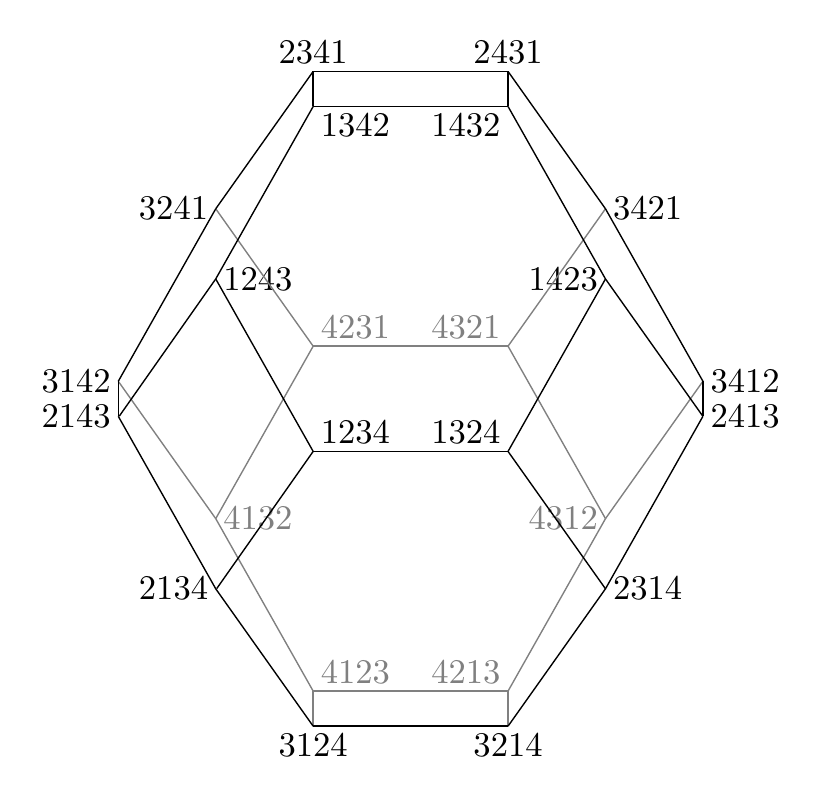}}
	\end{center}
	\caption{\label{fig:perm} On the left is the 
	permutahedron $\Pi_3$ shown in the ambient space $x_1+x_2+x_3=6$. On the right is the permutahedron
	$\Pi_4$ shown in the ambient space $x_1+x_2+x_3+x_4=10$.}
\end{figure}

The {\em Birkhoff polytope} $B_n$ which is the convex hull of all permutation matrices
 in $\RR^{n \times n}$ is a lift of the permutahedron
$\Pi_n$. The linear map $X \in \RR^{n \times n} \mapsto (1,2,\ldots,n)X \in \RR^n$ surjects $B_n$ onto~$\Pi_n$.
	The Birkhoff-von Neumann theorem~\cite{birkhoff1946three,von1953certain} 
	states that $B_n$ coincides with the set of \emph{doubly stochastic matrices}, i.e.,
$$B_n = \left\{ X \in \RR^{n \times n} \,:\, 
\begin{array}{c}
\sum_{j=1}^n X_{ij} = 1, \,\, \forall \,\, i \in [n], 
\,\,\sum_{i=1}^n X_{ij} = 1, \,\, \forall \,\, j \in [n], \\
X_{ij} \geq 0    \,\, \forall \,\,1 \leq i,j \leq n 
\end{array} 
\right\}.$$
One can see from this description that $B_n$ has $n^2$ facets given by the nonnegativity conditions $X_{ij} \geq 0$,
and lives in $\RR^\ell$ where $\ell = n^2$. Hence $B_n$ is a small lift of $\Pi_n$.
\end{example}

\label{example:octagon}
In general, lifts are far from being unique. For instance, in Figure~\ref{fig:oct} we show two different 
polyhedral lifts of a regular octagon. These are non-isomorphic, since the combinatorial structure of the polytopes is different, although both lifts have the minimum possible number of facets
(six, in this case).
\begin{figure}[h]
\begin{center}
\includegraphics[width=0.3\linewidth]{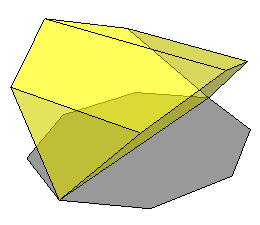} \hspace{1cm}  \includegraphics[width=0.3\linewidth]{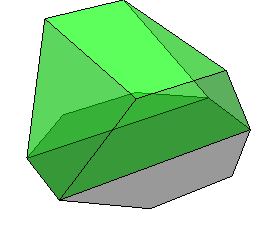}
\end{center}
\caption{Two non-isomorphic polyhedral lifts of a regular octagon, of six facets each.} 
\label{fig:oct}
\end{figure}


\subsection{Lifting to nonpolyhedral sets} The two examples of lifts we saw above were polytopes. From the computational point of view polytopes are feasible sets of linear programs. 
A significant portion of this survey will be devoted to linear programming's relative, \emph{semidefinite programming}. {\em Semidefinite programs} \cite{VaB:96} are linear optimization problems where one optimizes a linear function $\sum c_i x_i$ subject to 
constraints described by a linear matrix inequality (LMI) of the form
\begin{equation}
\label{eq:introlmi}
x_1 A_1 + \dots + x_n A_n \preceq B,
\end{equation}
where $A_1,\ldots,A_n,B$ are symmetric matrices. The notation $A \preceq B$ means that $B-A$ is \emph{positive semidefinite (psd)}. The relation $A \preceq B$ (equivalently, $B \succeq A$), is a partial order on the set of symmetric matrices of the same size. 
Linear matrix inequalities of the form \eqref{eq:introlmi} describe convex sets that are, in general, nonpolyhedral. Such sets are called \emph{spectrahedra}. We say that a convex set $C$ has a \emph{spectrahedral lift} if it can be expressed as the projection of a spectrahedron. The motivation is clear: if $C$ has a spectrahedral lift then the maximum of a linear function on $C$ can be computed using semidefinite programming.
\begin{example}
Consider the spectrahedron 
\[
\mathcal{E}_3 = \left\{ (x,y,z) \in \RR^3 : \begin{bmatrix} 1 & x & y\\ x & 1 & z\\ y & z & 1\end{bmatrix} \succeq 0\right\}.
\]
This spectrahedron is known as the \emph{elliptope} and arises in the famous semidefinite relaxation for the maximum cut problem \cite{GoemansWilliamson}. The three-dimensional elliptope is shown in Figure~\ref{fig:elliptope} along with various projections.
\end{example}

\begin{figure}[h]
\begin{center}
\includegraphics[width=0.25\linewidth]{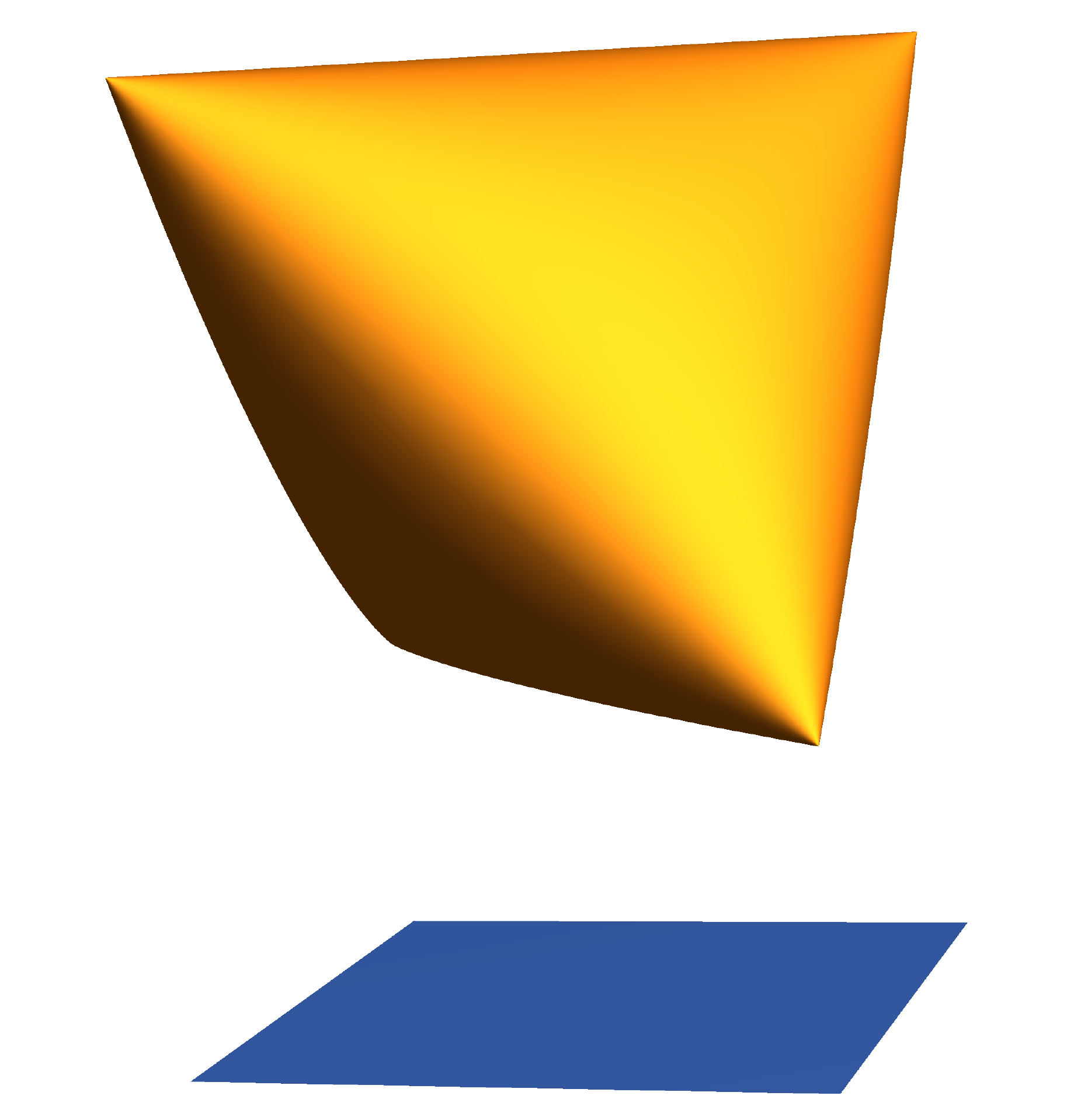} \hspace{2cm}  \includegraphics[width=0.25\linewidth]{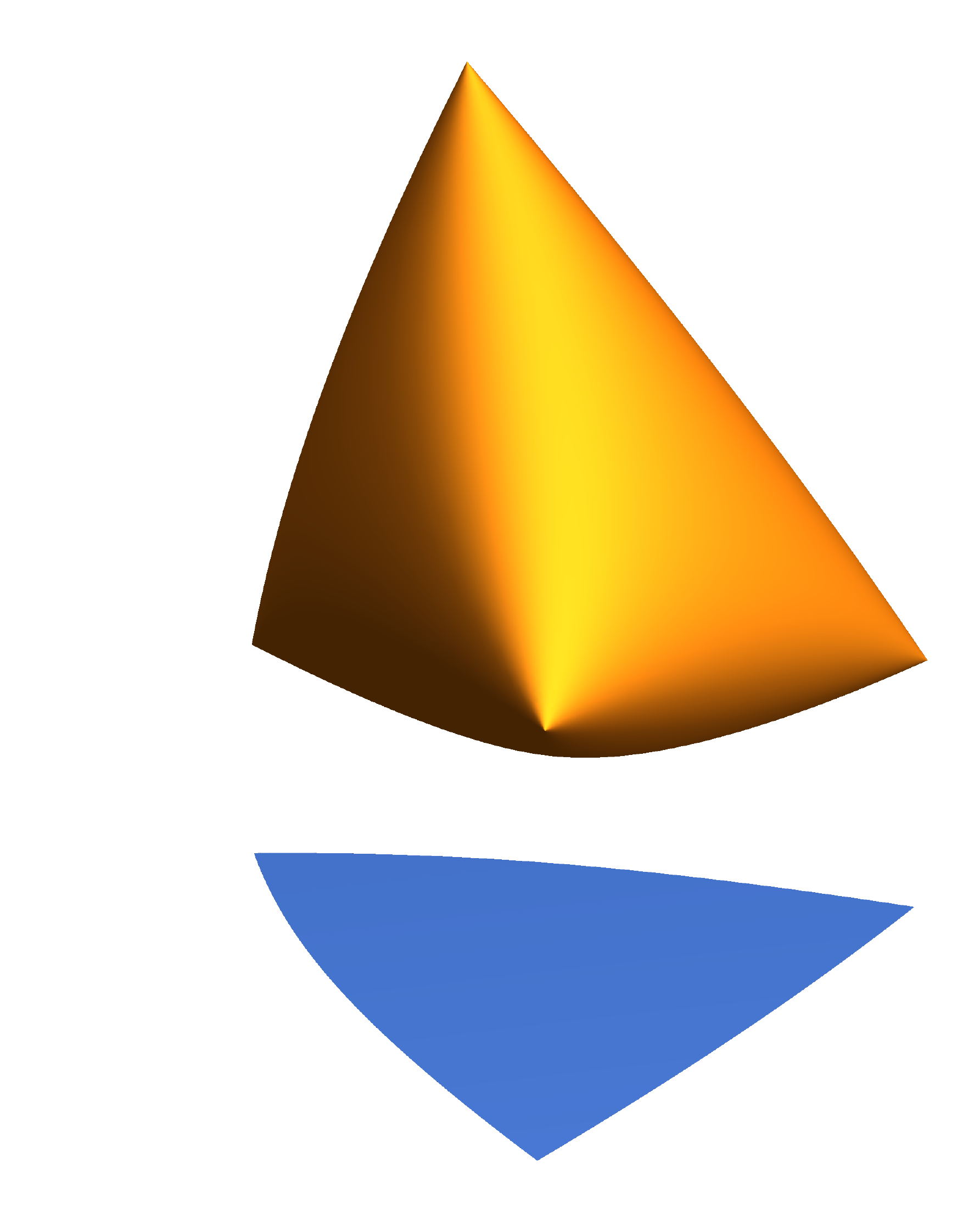}
\end{center}
\caption{Two different linear images of a spectrahedron. (Note that the linear image on the left is a polytope!)} \label{fig:elliptope}
\end{figure}

\textbf{Why focus on \emph{convex} sets?} The sets we have considered so far are all convex,  and this survey will be dedicated to lifts of convex sets. On a first encounter, focusing exclusively on lifts of convex sets may seem quite restrictive. However, essentially any linearly parameterized family of optimization problems can be rewritten as 
linear optimization over a convex set.
We illustrate this idea through a simple example. Suppose we have the family of optimization problems:
$$ \max \, \{ \alpha \sin t + \beta \cos t \,:\, t \in [0,2\pi] \}$$
where $\alpha, \beta \in \RR$ parameterize a choice of objective function.
Then we can set $x := \sin t$ and $y := \cos t$ and rewrite the family of problems as 
$\max \{ \alpha x + \beta y  \,:\, x^2 + y^2 = 1 \}$. Since the objective function is now linear, this problem is 
equivalent to maximizing the linear function $\alpha x + \beta y$ over the convex hull of the circle, 
i.e., the unit disc in $\RR^2$. This last reformulation results in a {\em convex program}, since the objective function 
is linear and the feasible region is convex. Figure~\ref{fig:3 formulations} shows the reformulations for two different objective functions.

\begin{figure}[ht]
	\begin{center}
		\includegraphics[scale=0.85]{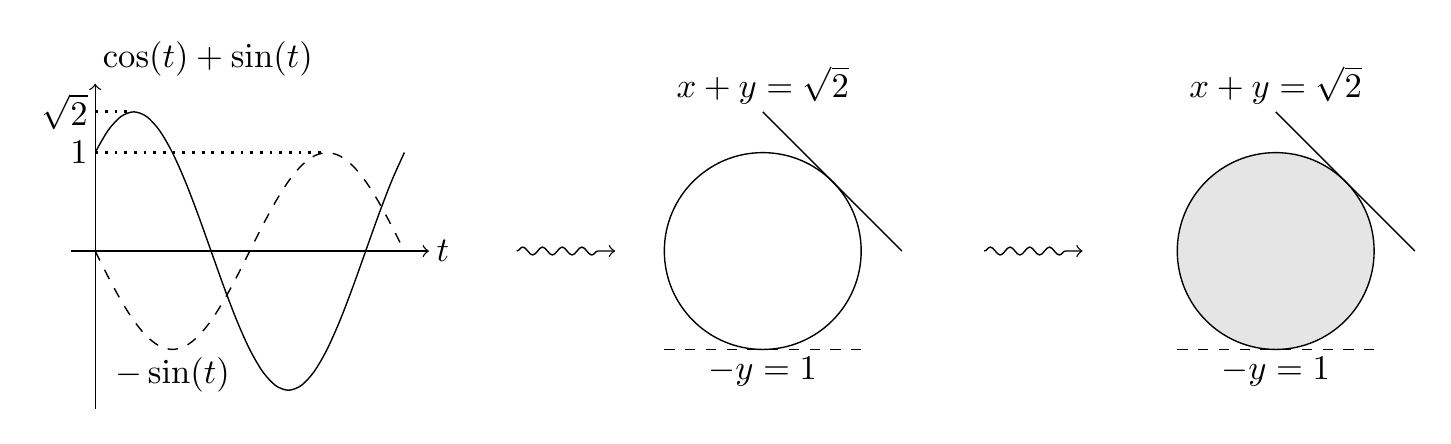}	
	\end{center}
	\caption{\label{fig:3 formulations} Reformulating the problem of maximizing $\alpha\cos(t)+\beta\sin(t)$ over $t\in [0,2\pi]$
	as a convex optimization problem for $(\alpha,\beta) = (1,1)$ (solid) and $(\alpha,\beta)=(0,-1)$ (dashed).}
\end{figure}
The above idea makes sense in much more generality. Suppose we want to reformulate a family of problems of the form 
$\max \{ \sum_{i=1}^k \alpha_i b_i(x) : x\in S \}$ where the $b_i(x)$ are fixed continuous functions from $\RR^n \rightarrow \RR$, $S\subset \RR^n$ is a compact set, and the  $\alpha_i$
 parameterize the problem instance. This can be rephrased as 
 $$\max \left\{ \sum_{i=1}^k \alpha_i z_i \,:\, z \in b(S)\right\} \quad\textup{where}\quad b(S) := \left\{(b_1(x), \ldots, b_k(x)) \,:\, x \in S \right\}.$$
Since the objective function is linear we can now optimize over the convex hull of $b(S)$ and obtain the 
equivalent convex problem $\max \left\{ \sum \alpha_i z_i \,:\, z \in \conv(b(S)) \right\}$.
The approach is most interesting when a subspace of objective functions is reformulated simultaneously.
In the special case where we consider only a single objective function $b$, this reduces to the 
(rather tautological) convex optimization problem of finding the largest point in the interval $\conv(b(S))$. 

As an example of this general convexification approach, consider applying it to a family of polynomial 
optimization problems.
For a vector $a \in \NN^n$, let $|a|$ denote its $1$-norm and let $x^a$ denote the monomial 
$x_1^{a_1} x_2^{a_2} \cdots x_n^{a_n}$. 
Suppose we wish to maximize any degree $d$ polynomial $p(x) = \sum_{|a| \leq d} p_a x^a$ over 
all $x \in S \subset \RR^n$. This is equivalent to the convex optimization problem
$$\max \left \{ \sum_{|a| \leq d}  p_a y_a \,:\, y \in \conv\{ (x^a)_{|a| \leq d}\,:\, x \in S\}  \right\}.$$
Note that the convex sets that arise from these reformulations are, in general, not polyhedra, 
and so it is not at all obvious how to computationally represent these sets. 

In combinatorial optimization we are constantly faced with linear optimization problems over finite sets in $\RR^n$, 
mostly subsets of vectors in $\{0,1\}^n$. Any such problem is equivalent to a linear program over the 
polytope that is the convex hull of the finitely many feasible solutions. We will see examples of lifts of such 
polytopes in later sections. 

\subsection{The conic point of view} The formal setting we consider in this paper is that of \emph{conic lifts} which permits an elegant treatment of both polyhedral and spectrahedral lifts in a unified way. This treatment is also motivated from optimization theory, where linear and semidefinite programming are seen as special cases of \emph{conic programming} \cite{nesterov1994interior,ben2001lectures}. To motivate the definition which follows, note that a polyhedron $P = \{x \in \RR^n : Ax \leq b\}$ where $A \in \RR^{m\times n}$ can be equivalently seen (via a linear bijection) as the intersection $\RR^m_+ \cap L$, where $\RR^m_+$ is the nonnegative orthant in 
$\RR^m$ and $L = b + \text{colspan}(A)$ is an affine space. Similarly the spectrahedron defined by \eqref{eq:introlmi} is linearly isomorphic to $\S^m_+ \cap L$, where $\S^m_+$ is the convex cone of $m\times m$ symmetric positive semidefinite matrices, and $L = B + \linspan(A_1,\ldots,A_n)$. With these examples in mind, we can state the formal definition of a conic lift that we adopt in this paper:
\begin{definition} \label{def:lift definition}
Let $C \subset \RR^n$ be a convex set and $K \subset \RR^\ell$ be a closed convex cone. We say that
$C$ has a $K${\em -lift} if $C = \pi(K \cap L)$ where $L$ is an affine space in $\RR^\ell$ and $\pi \,:\, \RR^\ell \rightarrow \RR^n$ is a
linear map. The convex set $K \cap L$ is called a $K$-lift of $C$.
\end{definition}
If $C = \pi(\RR^m_+ \cap L)$ (respectively, $C = \pi(\PSD^m \cap L)$) we say that $C$ has a polyhedral (respectively, spectrahedral) lift of {\em size} $m$.  Going back to Example \ref{ex:permutahedron}, the Birkhoff polytope $B_n$
 has $O(n^2)$ inequalities none of which are redundant, and thus is a polyhedral lift of $\Pi_n$ of size $O(n^2)$.
 The smallest $m$ for which $C = \pi(\RR^m_+ \cap L)$ is called the {\em linear extension complexity} of $C$, and the smallest $m$ 
 for which $C = \pi(\PSD^m \cap L)$ is called the {\em semidefinite extension complexity} of $C$. It was shown by Goemans 
\cite{goemans2015smallest} that the linear extension complexity of $\Pi_n$ is $O(n \log n)$. The proof involves constructing
an explicit lift of $\Pi_n$ of that size by using sorting networks.

Another convex cone of special interest in conic programming is the \emph{second-order cone} defined as
\[
\mathcal{L}_+^{3} = \{(x_0,x)\in \RR\times \RR^{2}\;:\; \|x\|_2 \leq x_0\}.
\]
If a convex set has a $(\mathcal{L}^3_+)^m$-lift (for some integer $m$) then we say it has a {\em second-order cone lift}. Examples of convex sets with second-order cone lifts include ellipsoids and $\ell_p$-norm balls for rational $p \geq 1$, which arise in many practical situations. (See, e.g.,~\cite{ben2001lectures} for these, and many more, interesting examples.) 
Figure \ref{fig:Kliftsrelation} summarizes the inclusion relationship between different types of $K$-lifts, when $K$ is one of $\RR^m_+, (\mathcal{L}^3_+)^m, \S^m_+$. A discussion and proof of these relationships, in particular the strict inclusions, will be given in Section~\ref{sec:obstructions and lower bounds}. 
\begin{figure}[ht]
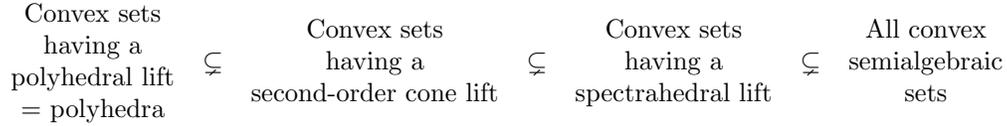

{\small
\[
\begin{array}{c}
\text{Convex sets}\\
\text{having a}\\
\text{polyhedral lift}\\
\text{= polyhedra}
\end{array}
\; \subsetneq \;
\begin{array}{c}
\text{Convex sets}\\
\text{having a}\\
\text{second-order cone lift}
\end{array}
\; \subsetneq \;
\begin{array}{c}
\text{Convex sets}\\
\text{having a}\\
\text{spectrahedral lift}
\end{array}
\; \subsetneq \;
\begin{array}{c}
\text{All convex}\\
\text{semialgebraic}\\
\text{sets}
\end{array}
\]
}
\caption{Inclusion relationship between different types of lifts. The fact that the inclusions are strict will be discussed in more detail in Section \ref{sec:obstructions and lower bounds}.
(The definition of a semialgebraic set is given in Section~\ref{sec:obs-semi-prelim}.)
	}
\label{fig:Kliftsrelation}
\end{figure}

\begin{remark}[Terminology]
We remark that there is differing terminology in the papers that study
lifts of convex sets.  Polyhedral lifts are sometimes called {\em
  extended formulations}, {\em linear lifts} or {\em
  LP-lifts}. Spectrahedral lifts are sometimes called {\em
  semidefinite lifts} or {\em psd-lifts}. A convex set that admits a
spectrahedral lift is a {\em projected spectrahedron} or {\em
  spectrahedral shadow}, and the set is said to have a {\em linear
  matrix inequality (LMI) representation} or {\em semidefinite
  representation}. {Linear extension complexity} is also known as {\em
  nonnegative rank} and {semidefinite extension complexity} as {\em
  psd rank}. In this paper we will use the terminology we have
introduced in the previous paragraphs.
\end{remark}

\subsection{Lifts in mathematics} 
The idea of going to a higher-dimensional space to seek a ``nicer'' representation of an object  is a common theme in many areas of mathematics.
A familiar example is {\em resolution of singularities} of algebraic varieties, one of the oldest and best known problems in algebraic geometry. Given a singular algebraic variety, the goal is to write it as a proper
birational image of a smooth variety.  In 1964 Hironaka~\cite{hironaka1964resolution1,hironaka1964resolution2} 
proved that such resolutions always exist for varieties over
fields of characteristic zero such as $\CC$ or $\RR$. 
In simplified terms, Hironaka's result says that the set of solutions in $k^n$, where $k$ is a field
of characteristic zero, to a system of polynomial equations can always be rationally parametrized by a smooth manifold.
(See, e.g.,~\cite{hauser2003hironaka} for a tutorial-style exposition.)
Besides the theoretical importance of resolutions, they are also of practical relevance. For instance, they allow us to use our arsenal of techniques for optimization over smooth functions to indirectly optimize over a non-smooth variety, by passing to the lifted smooth variety.

\begin{example}
Consider the deltoid curve (see Figure~\ref{fig:deltoid}), which is the set of solutions in $\RR^2$ of
$$(x^2 + y^2)^2 -8 (x^3 - 3 xy^2) +18 (x^2 + y^2) - 27 = 0.$$
As is evident from Figure~\ref{fig:deltoid}, this curve has three singularities.
\begin{figure}[h]
\begin{center}
\includegraphics[width=3cm]{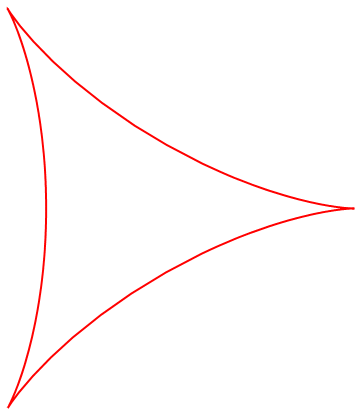} \,\,\,\,
\includegraphics[width=4cm]{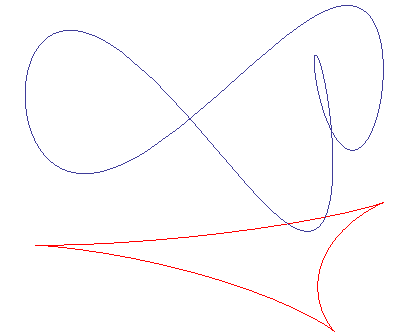}
\end{center}
\caption{On the left is the deltoid curve which has $3$ singularities.
On the right we see the deltoid as the projection of a smooth curve. \label{fig:deltoid}}
\end{figure}

The deltoid is polynomially parametrized by a circle, since it can be written as
$$\{(1+ 2 \cos(t) - 2\sin(t)^2, 2 \sin(t) - 2\sin(t)\cos(t)) \ : \ t\in[0,2\pi]\}.$$
This is already very useful, but one can actually write it as a linear projection of a smooth curve, by simply adding a $\sin(3t)$ term to the parametrization. Consider the curve
$$Y=\{(1+ 2 \cos(t) - 2\sin(t)^2, 2 \sin(t) - 2\sin(t)\cos(t), \sin(3t)) \ : \ t\in[0,2\pi]\} \subseteq \RR^3.$$
The curve $Y$ is smooth, and 
we can see it projecting onto the deltoid.

\end{example}

Here is another instance of the usefulness of ``lifting''.
Given a probability distribution $\pi_{\mathcal{X}}$ defined on some space $\mathcal{X}$, 
a fundamental problem in probability and statistics 
is to generate a sample from the distribution.  One way to do this is to
construct an ergodic Markov chain with stationary distribution given by
$\pi_{\mathcal{X}}$, and then generate sample paths for the Markov chain.  A
key measure of the performance of such a sampling method is the \emph{mixing
time} of the Markov chain, a measure of how fast the distribution of the
samples converges to the stationary distribution.

A \emph{lifting} of a Markov chain is a new Markov chain on a larger state space
$\mathcal{X}\times \mathcal{Y}$ so that the marginal on the $\mathcal{X}$ space
is the original Markov chain.  In certain situations, carefully designed
liftings of a Markov chain mix faster than the original chain, leading to
faster sampling methods~\cite{chen1999lifting}.  More generally, many methods
for sampling involve introducing auxiliary variables and a probability
distribution $\pi_{\mathcal{X}\times \mathcal{Y}}$ on a larger state space,
running a Markov chain Monte-Carlo (MCMC) method to (approximately) sample from that distribution, and
then taking the marginal on the $\mathcal{X}$ variables. Such strategies often
tend to mix faster in practice than an MCMC method working directly in the
space of the original distribution.  These are often called \emph{auxiliary
variable samplers}, and include approaches such as Hamiltonian Monte-Carlo 
(HMC)~\cite{neal2011mcmc} and slice sampling~\cite{neal2003slice}. 

On a philosophical note, recall that in {\em quantifier elimination}, one is interested in finding a set that has been specified in terms of 
quantifiers and extra variables. This is a 
projection operation that is, in general, hard to do. 
For example, projecting out variables from a system of linear inequalities can be done by 
{\em Fourier-Motzkin elimination}  \cite[Chapter 1]{ziegler2000lectures} which can take an exponential number of steps. 
The concept of a lift is inverse to the idea of projection. 
We start with a highly complicated (convex) set and ask if it can be written as the projection of a 
(if possible, simple) convex set  in some higher dimension.

 \subsection{Organization}
 There are a number of applications, coming from a wide range of areas of mathematics, that illustrate the usefulness of lifted representations of convex sets.
 In Section~\ref{sec:examples} we give four very different examples, to bring home this point. 
 
Given a convex set $C \subset \RR^n$ and a closed convex cone $K \subset \RR^\ell$, the first question one 
can ask is whether $C$ admits a $K$-lift at all. In Section~\ref{sec:slacks} we build the theory that allows us to answer this 
foundational question. The key ingredient is the {\em slack operator} of a convex set which is described in 
Section 3.3. This is a generalization of the {\em slack matrix} of a polytope which was
 introduced by Yannakakis \cite{yannakakis1991expressing} to study polyhedral lifts of polytopes. The slack matrix of a polytope 
 is defined in Section 3.1 and Yannakakis' theorem that decides the existence of a $\RR^m_+$-lift of a polytope $P \subset \RR^n$ is described in Section 3.2. The main result of Section 3 says that a convex set $C \subset \RR^n$ has a $K$-lift for some closed convex cone 
 $K \subset \RR^\ell$ if and only if the slack operator of $C$ factorizes through $K$ and its polar cone $K^\ast$. 
 This result is stated and discussed in Section 3.4 and we illustrate it on several examples. The 
 slack operator of a convex set can also be seen as a generalization of the notion of Bregman divergence of a convex function.
 This analogy is discussed in Section 3.5.

There are many creative constructions of lifts of convex sets in the literature, and in Section~\ref{sec:constructions}
we focus on construction methods for spectrahedral lifts  of convex sets. We emphasize in particular the connection with \emph{sum-of-squares} certificates of nonnegativity, and semidefinite hierarchies.
This viewpoint is illustrated on  several examples. We briefly discuss the analogous viewpoint for polyhedral lifts and establish the connection with well-known methods such as the Sherali-Adams hierarchy~\cite{sheraliadams}. 

Having addressed the existence question and seen some construction methods, a next natural step is to understand 
obstructions to the existence of a $K$-lift of $C$. In Section~\ref{sec:obstructions and lower bounds} we discuss two flavors of 
obstructions. The first is combinatorial in nature and depends on the facial structure of both $C$ and $K$. 
Such obstructions are already powerful enough to show interesting results such as the psd cone $\mathcal{S}^3_+$ 
does not admit a lift into a product of second order cones~\cite{fawzi2018representing}. Next we consider algebraic obstructions to the existence 
of lifts. We prove lower bounds on the semidefinite extension complexity of $C$ based on the degree of 
its algebraic boundary as well as $\dim(C)$.
Finally we discuss the recent result of Scheiderer~\cite{scheidererSDR} that not all convex semialgebraic sets admit a spectrahedral lift and 
provide concrete examples of such sets (see Figure \ref{fig:Kliftsrelation}).

The study of lifts of convex sets has many results and theories beyond what we have chosen to present in this 
paper. We conclude with a brief discussion of related directions and results, as well as pointers to surveys and 
expository articles on these topics.

\section{Four Examples}
\label{sec:examples}

In the rest of this paper we will study lifts of convex sets as in Definition~\ref{def:lift definition}.
We begin with four examples of lifts, each from a different area of mathematics, to illustrate the ubiquity and usefulness of the
idea even beyond the context of optimization. This section can be read independently of the rest of the paper and can be skipped if
the reader wishes to get to the mathematical theory behind lifts, which begins in Section~\ref{sec:slackoperator}.

\subsection{Lifting $0/1$ polytopes to flow polytopes via binary decision diagrams}
\label{sec:bdd}
\begin{figure}
\begin{center}
\includegraphics[scale=1]{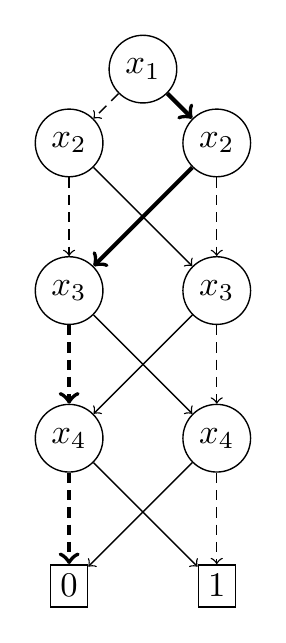}
\hspace{1cm}
\includegraphics[scale=1]{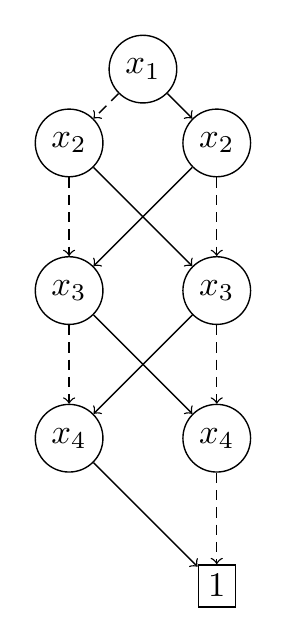}
\end{center}
\caption{\label{fig:bdd1} On the left is an (ordered) BDD corresponding to the Boolean function $f_{\textup{xor}}$ that computes
the sum of the inputs modulo $2$. Solid lines represent arcs labeled by 1, and dashed lines represent arcs labeled by 0.
 The path corresponding to the assignment $(1,1,0,0)$ to the variables is highlighted in bold. On
the right is a zero-suppressed BDD representing the set $\{x\in \{0,1\}^4\;:\; f_{\textup{xor}}(x) = 1\}$.}
\end{figure}
A \emph{binary decision diagram (BDD)}
or \emph{branching program} is a data structure that represents a Boolean function. It is a
directed acyclic graph in which exactly one node has in-degree zero (the
\emph{source node}), exactly two nodes have out-degree zero (the \emph{sink
nodes}), and every non-sink node (known as a \emph{decision node}) has
out-degree two (see Figure~\ref{fig:bdd1}).  The sink nodes are labeled $0$ and $1$.  Every decision node
is labeled with a variable $x_1,x_2,\ldots,x_n$ and has one outgoing arc
labeled $0$ and one outgoing arc labeled $1$ (which are usually depicted as
dashed and solid arcs respectively).  Any assignment of Boolean values to the
variables, i.e.,  $x_i = a_i\in \{0,1\}$ for all $i$, specifies a path from the
source to one of the sinks. The path is defined by taking, for each decision
node labeled $x_i$, the outgoing edge labeled $a_i$ (see
Figure~\ref{fig:bdd1}). Whether the path associated with $a\in \{0,1\}^n$ ends
at the $0$-sink or the $1$-sink defines a Boolean function
$f:\{0,1\}^n\rightarrow \{0,1\}$. The diagram on the left of Figure~\ref{fig:bdd1} shows a BDD
representing the function
\[ f_{\textup{xor}}(x_1,x_2,x_3,x_4) = x_1+x_2+x_3+x_4 \;\;  (\textup{mod $2$}).\]
It has the additional property that the
variables are totally ordered and, along every
source-sink path, every variable appears at most once in a way that respects the order.
Such a BDD is called an \emph{ordered BDD (OBDD)} or \emph{oblivious read-once
branching program}~\cite{bryant1986graph}. If a Boolean function has an OBDD representation with a
small number of nodes then this automatically leads to efficient algorithms for
counting the number of satisfying assignments, sampling from the uniform
distribution on satisfying assignments, and solving the integer linear program
$\max_{x\in \{0,1\}^n: f(x)=1} \sum_{i}a_ix_i$. Moreover, natural operations on
Boolean functions lead to corresponding operations on OBDDs.

Associated with a Boolean function $f:\{0,1\}^n\rightarrow \{0,1\}$ is the $0/1$ polytope
$P_f = \textup{conv}\{x\in \{0,1\}^n\;:\; f(x) = 1\}$.
It turns out that we can construct a polyhedral lift of $P_f$ from an OBDD representation of $f$.
Let $A$ denote the set of arcs of the OBDD. For any subset $S$ of arcs let $\chi^S\in \{0,1\}^A$ denote
its {\em characteristic vector} defined as $(\chi^S)_i = 1$ if $i \in S$ and $(\chi^S)_i = 0$ otherwise.
Given an OBDD $B_f$ representing $f$, consider the \emph{flow polytope}
\begin{equation*}
	F(B_f) = \textup{conv}\left\{\chi^S\in \{0,1\}^A\;:\;
	\textup{$S\subseteq A$ is a directed path from the source to the $1$-sink in $B_f$}\right\}.
\end{equation*}
For $i=1,2,\ldots,n$, let $A_i$ denote the set of arcs leaving nodes of the OBDD labeled with variable $i$.
\begin{proposition}
The polytope $F(B_f)$ is a polyhedral lift of the $0/1$ polytope $P_f$. The linear map defining the lift is
$\pi:\RR^{A}\rightarrow \RR^n$ defined by $[\pi(y)]_i = \sum_{a\in A_i} y_{a}$ for $i=1,2,\ldots,n$.
\end{proposition}

\begin{proof}
There is a one-one correspondence between paths from the source to the $1$-sink in the OBDD and
vertices of $P_f$. Moreover, every such path contains at most one arc from each $A_i$. As such, the vertices of $F(B_f)$
are mapped to those of $P_f$ by $\pi$.
\end{proof}

Since $F(B_f)$ is an example of a network flow polytope, it can be described using one linear inequality for each
arc in the network. Indeed, it is a fundamental result (see, e.g.,~\cite{schrijver2003combinatorial}) that
\[ F(B_f) = \{y\in \RR^{A}\;:\; y \geq 0,\; By = b\}\]
where $b\in \RR^V$ and $B:\RR^{A}\rightarrow \RR^V$ are defined by
\[ b_v = \begin{cases} 1 & \textup{if $v$ is the $1$-sink}\\ -1 & \textup{if $v$ is the source}\\0 & \textup{otherwise,}\end{cases}
\quad\textup{and}\quad [By]_v = \sum_{a\in v_{\textup{in}}}y_a  - \sum_{a\in v_{\textup{out}}}y_a\quad\textup{for all $v\in V$},\]
and where $v_{\textup{in}}\subseteq A$ and $v_{\textup{out}}\subseteq A$ denote the set of incoming arcs and outgoing arcs
from vertex $v\in V$. To summarize, any OBDD representation of a Boolean function leads to a polyhedral
lift of the corresponding $0/1$ polytope of the same size as the OBDD.
\begin{corollary}
\label{cor:bdd-lift}
If $f:\{0,1\}^n\rightarrow \{0,1\}$ is a Boolean function with an OBDD representation having $|A|$ arcs, then
the polytope $P_f = \{x\in \{0,1\}^n\;:\; f(x) =1\}$ has a polyhedral lift of size $|A|$.
\end{corollary}
\begin{example}[Odd parity polytope]
The odd parity polytope is the convex hull of binary vectors with an odd number
	of ones. From Figure~\ref{fig:bdd1} it is clear that the Boolean
	function $f_{\textup{xor}}$ has an OBDD with $4(n-1)+2 = 4n-2$ arcs. It
	follows that the odd parity polytope has a polyhedral lift of size
	$4n-2$.
\end{example}
There is a nontrivial upper bound on the size of an OBDD of a general Boolean function $f$.
This gives a corresponding upper bound on the size of a polyhedral lift of an arbitrary $0/1$ polytope.

\begin{theorem}
Any $0/1$ polytope in $\RR^n$ has a polyhedral lift of size at most $\frac{6.25}{n}\cdot 2^n$.
\end{theorem}
\begin{proof}
Any $n$-variable Boolean function has an OBDD with at most
$\frac{3.125}{n}\cdot 2^n$ nodes~\cite[Theorem 1]{liaw1992obdd}, and hence at
most $\frac{6.25}{n}\cdot 2^n$ arcs. The result then follows from
Corollary~\ref{cor:bdd-lift}.
\end{proof}
There are many variations and generalizations of OBDDs, a number of which
naturally give rise to flow-polytope lifts of $0/1$ polytopes.
A similar approach can be used to construct polyhedral lifts based
on free BDDs (in which every variable appears at most once on every source-sink path, but not
necessarily in the same order~\cite{masek1976fast,wegener2000branching}). The
lifted representations of $0/1$ polytopes described in this section can be
simplified by using zero-supressed OBDDs (ZBDDs)~\cite{minato1993zero}.  These are OBDDs
where any arcs leading only to the $0$-sink are removed, and the corresponding
diagram is reduced by removing any redundancy. A ZBDD is shown on the right in Figure~\ref{fig:bdd1}.  Clearly these arcs do not change
any path from the source to the $1$-sink, so the polyhedral lifts described
above can (possibly) be reduced in size by using a ZBDD representation instead.

\subsection{The chain polytope of a poset and the Lov\'asz theta body of a graph}
\label{sec:chainTH}

\subsubsection{Chain polytopes}
Let $\mathcal{P}$ be a finite partially ordered set (poset), where the partial order is denoted by $\prec$\footnote{We also use this notation for the specific partial order on symmetric matrices induced by the positive semidefinite cone. Occasionally both will be used in close proximity, but the meaning should be clear from the context.}, and let $\RR^{\mathcal{P}}$ denote the Euclidean space whose coordinates are indexed by the elements of $\mathcal{P}$. We start by quickly recalling
a series of basic concepts from posets, that will be used throughout the example. The {\em characteristic vector} of an element $a \in \mathcal{P}$ is the vector in $\RR^{\mathcal P}$ with $1$ in the
coordinate indexed by $a$ and $0$ everywhere else.  An {\em antichain} in $\mathcal{P}$ is a collection of elements of $\mathcal{P}$ such that no two in the collection are comparable in the partial order.  A {\em chain} in $\mathcal{P}$ on the other hand is a collection of elements that are totally ordered. The characteristic vector of a collection of elements in $\mathcal{P}$, for instance of a chain or antichain, is the sum of the characteristic vectors of all elements in that collection. A closely related concept to antichain is that of a \emph{filter}, i.e. a set $I$ such that $a_1 \in I$ and $a_2 \succeq a_1$ implies $a_2 \in I$. There is a natural bijection between filters $I$ and antichains $A$, given by
$$ I = \{ y \in \mathcal{P} \,:\, y \succeq x \textup{ for some } x \in A \} \,\,\,\,\leftrightarrow \,\,\,\, A = \{ \textup{minimal elements of } I \}.$$
Finally, we say that $a_2$ \emph{covers}  $a_1$ if  $a_1 \prec a_2$ and there is no $x \in \mathcal{P}$ with  $a_1 \prec x \prec a_2$. We call this a \emph{cover relation}.

\begin{example}
Consider $\mathcal P$ to be the poset of the faces of the triangular prism illustrated in Figure \ref{fig:poset} ordered by inclusion. We will usually represent such a poset by its \emph{Hasse diagram}. This is simply a graph where the nodes are the elements of the poset, drawn in such a way that if $a_1 \preceq a_2$ then $a_1$ is below $a_2$, and the edges represent the \emph{cover relations}. In Figure \ref{fig:poset} we can also see the Hasse diagram of the poset of faces of the triangular prism.

\begin{figure}[ht]
\begin{center}
\raisebox{0.2\height}{\includegraphics[width=0.3\textwidth]{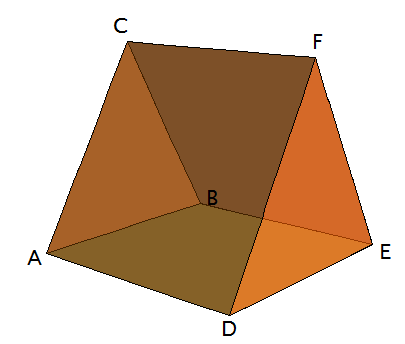}}
\hspace{1cm}
\begin{tiny}
\begin{tikzpicture}[scale=0.9]
  \node (abcdef) at (0,6) {$ABCDEF$};
  \node (abc) at  (-3,4.5) {$ABC$};
  \node (abde) at (-1.5,4.5) {$ABDE$};
  \node (bcef) at (0,4.5) {$BCEF$};
  \node (acdf) at (1.5,4.5) {$ACDF$};
  \node (def) at  (3,4.5) {$DEF$};
  \node (ab) at (-4,3) {$AB$};
  \node (bc) at (-3,3) {$BC$};
  \node (ac) at (-2,3) {$AC$};
  \node (df) at (4,3) {$DF$};
  \node (ef) at (3,3) {$EF$};
  \node (de) at (2,3) {$DE$};
  \node (ad) at (-1,3) {$AD$};
  \node (be) at (0,3) {$BE$};
  \node (cf) at (1,3) {$CF$};
  \node (a) at (-3,1.5) {$A$};
  \node (b) at (-2,1.5) {$B$};
  \node (c) at (-1,1.5) {$C$};
  \node (d) at (1,1.5) {$D$};
  \node (e) at (2,1.5) {$E$};
  \node (f) at (3,1.5) {$F$};
  \node (empty) at (0,0) {$\emptyset$};
  \draw (empty) -- (a) -- (ab) -- (abc) -- (abcdef) -- (def) -- (df) -- (f) -- (empty);
  \draw (empty) -- (b) -- (bc) -- (bcef) -- (abcdef) -- (bcef) -- (ef) -- (e) -- (empty);
  \draw (empty) -- (c) -- (ac) -- (abc) -- (abcdef) -- (def) -- (de) -- (d) -- (empty);
  \draw  (a) -- (ac) -- (acdf) -- (cf) -- (c) -- (bc) -- (abc);
  \draw  (a) -- (ad) -- (abde) -- (abcdef) -- (acdf)--(df) -- (d) -- (ad)-- (acdf);
  \draw  (b) -- (ab) -- (abde) -- (de) -- (e)--(be) -- (bcef) -- (cf)-- (f)--(ef) -- (def);
  \draw  (b) -- (be) -- (abde);
\end{tikzpicture}
\end{tiny}
\end{center}
\caption{\label{fig:poset} Triangular prism and the Hasse diagram of its face poset}
\end{figure}
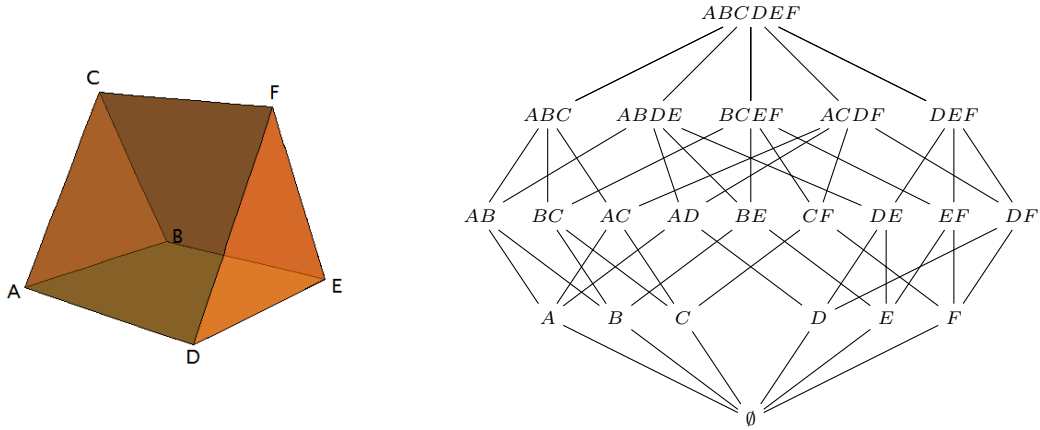

This poset has a single maximal element, the entire prism, and minimal element, the empty face. An example of a chain is $(\emptyset, A, AB, ABDE, ABCDEF)$ which, in this case, is maximal. An example of a maximal antichain is $\{ABC,AD,BE,CF,DEF\}$. The filter mapping to that antichain is the collection of all the elements that are greater than or equal to any of the elements, namely 
	\[\{ABC,AD,BE,CF, DEF,\\ADBE,ACDF,BCEF,ABCDEF\}.\]
\end{example}

The \emph{chain polytope} $\textsc{chain}(\mathcal{P})$ of $\mathcal{P}$, introduced in \cite{stanley1986two}, is the convex hull in $\RR^{\mathcal{P}}$ of the characteristic vectors of all antichains of $\mathcal{P}$. The name comes from its
inequality description:
\begin{equation}
\textsc{chain}(\mathcal{P})
= \left\{ x \in \RR^{\mathcal{P}}_+ \;:\; 0 \leq x_{a_1}+x_{a_2}+\cdots + x_{a_k} \leq 1,
	\,\,\forall \,\, \textup{ chains } a_1 \prec a_2 \prec \cdots \prec a_k
\textup{ in } \mathcal{P} \right\}.\label{eq:chain_ineq}
\end{equation}
The correctness of this inequality description will become clear at the end of this section.
The facets of $\textsc{chain}(\mathcal{P})$ are given by the nonnegativity constraints $0 \leq x_a$ for all $a \in \mathcal{P}$ and the inequalities
$x_{a_1}+x_{a_2}+\cdots + x_{a_k} \leq 1$ for all maximal chains $a_1 \prec a_2 \prec \cdots \prec a_k$ in $\mathcal{P}$.

\begin{example} There are five non isomorphic posets with three elements. Those will have three dimensional chain polytopes. We can see all of them in Figure~\ref{fig:3dchainpoly}. Note that two have the same chain polytope up to coordinate labeling.

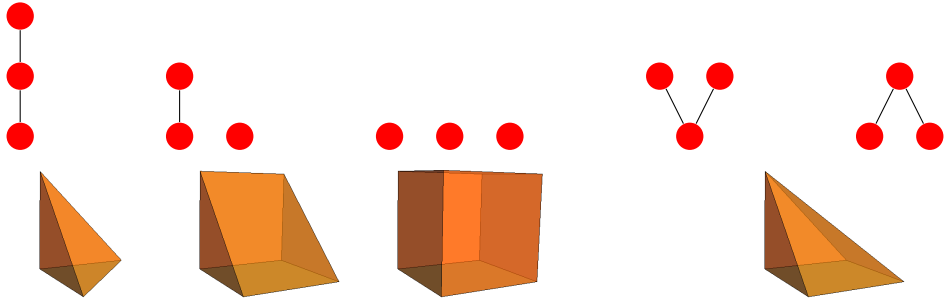
\begin{figure}[ht]

\begin{tikzpicture}[scale=0.8]
  \node[circle,fill=red] (1) at (0,0) {};
  \node[circle,fill=red] (2) at  (0,1) {};
  \node[circle,fill=red] (3) at (0,2) {};
  \draw  (1) -- (2) -- (3);
\end{tikzpicture} \hspace{1.5cm}
\begin{tikzpicture}[scale=0.8]
  \node[circle,fill=red] (1) at (-0.5,0) {};
  \node[circle,fill=red] (2) at  (-0.5,1) {};
  \node[circle,fill=red] (3) at (0.5,0) {};
  \draw  (1) -- (2);
\end{tikzpicture}\hspace{1.5cm}
\begin{tikzpicture}[scale=0.8]
  \node[circle,fill=red] (1) at (-1,0) {};
  \node[circle,fill=red] (2) at  (0,0) {};
  \node[circle,fill=red] (3) at (1,0) {};
\end{tikzpicture}\hspace{1.5cm}
\begin{tikzpicture}[scale=0.8]
  \node[circle,fill=red] (1) at (0,0) {};
  \node[circle,fill=red] (2) at  (-0.5,1) {};
  \node[circle,fill=red] (3) at (0.5,1) {};
  \draw  (2) -- (1) -- (3);
\end{tikzpicture}\hspace{1.5cm}
\begin{tikzpicture}[scale=0.8]
  \node[circle,fill=red] (1) at (0,1) {};
  \node[circle,fill=red] (2) at  (-0.5,0) {};
  \node[circle,fill=red] (3) at (0.5,0) {};
  \draw  (2) -- (1) -- (3);
\end{tikzpicture}

\includegraphics[width=2cm]{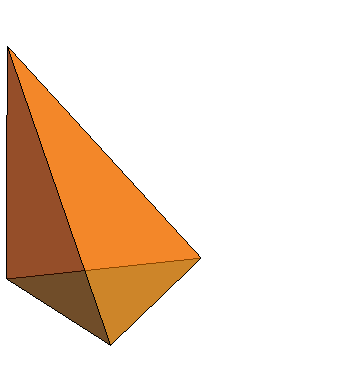}\hspace{0cm}
\includegraphics[width=2cm]{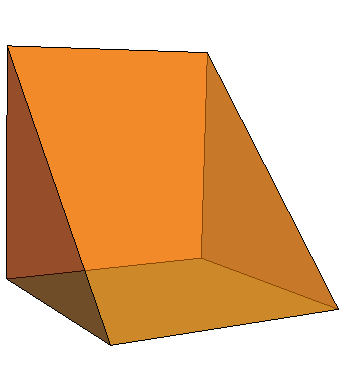}\hspace{0.5cm}
\includegraphics[width=2cm]{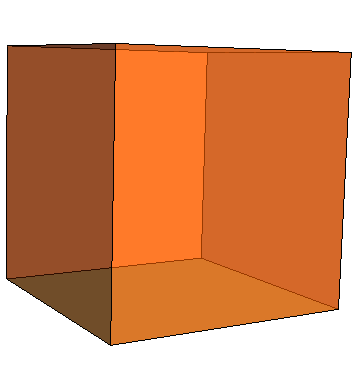}\hspace{2.75cm}
\includegraphics[width=2cm]{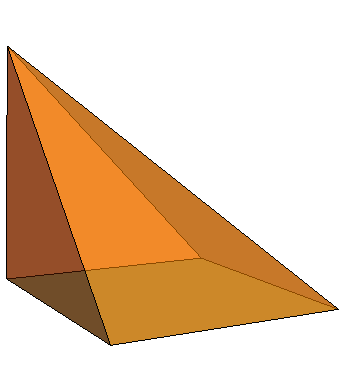}

\caption{\label{fig:3dchainpoly} All the three-element posets and their corresponding chain polytopes}

\end{figure}

\end{example}

The number of vertices and facets of $\textsc{chain}(\mathcal{P})$ can be superpolynomial in $|\mathcal{P}|$ as the following example shows.

\begin{example} Let $\mathcal{B}(n)$ denote the Boolean lattice of the subsets of $[n]$. The number of antichains in $\mathcal{B}(n)$
is the $n$-th Dedekind number which grows at least as fast as $2^{\frac{N}{\sqrt{\log(N)}}}$, where $N=2^n$ is the number of elements in $\mathcal{B}(n)$. As antichains correspond to vertices of the chain polytope $\textsc{chain}(\mathcal{B}(n))$, there are a superpolynomial number of them.
On the other hand, since all  maximal chains define facets of $\textsc{chain}(\mathcal{B}(n))$, and there are $n!$ maximal chains in $\mathcal{B}(n)$, the number of facets of $\textsc{chain}(\mathcal{B}(n))$ is of the order $N^{\log(\log(N))}$ which is also superpolynomial in $n$.
\end{example}

The \emph{order polytope} of the poset $\mathcal{P}$ was also introduced in \cite{stanley1986two}.  It is given by
\begin{equation}\label{eq:order polytope}
\textsc{order}(\mathcal{P}) = \{ x \in \RR^{\mathcal{P}} \,:\, 0 \leq x_{a_1} \leq x_{a_2} \leq 1, \,\,\forall \,\, a_2 \succeq a_1 \textup{ in } \mathcal{P} \}.
\end{equation}
It is in fact enough to consider the inequalities corresponding to cover relations, together with $x_a \geq 0$ for all minimal elements $a \in \mathcal{P}$ and
$x_a \leq 1$ for all maximal elements $a \in \mathcal{P}$.
The vertices of $\textsc{order}(\mathcal{P})$ are the characteristic vectors of filters, which we have seen to have a natural bijection to antichains,
which leads to the following result.


\begin{proposition}[Theorem 3.2 \cite{stanley1986two}]
The map $f:\RR^{\mathcal{P}} \rightarrow \RR^{\mathcal{P}}$ defined by 
\[
f(x)_{a_i}=
\begin{cases}
x_{a_i} & \text{if } a_i \text{ is a minimal element of } \mathcal{P} \\
\min\{x_{a_i}-x_{a_j} \ | \ a_i \textrm{ covers } a_j \} & \text{otherwise}
\end{cases}
\]
is piecewise linear and maps $\textsc{order}(\mathcal{P})$ bijectively to $\textsc{chain}(\mathcal{P})$ preserving the bijection between vertices.
\end{proposition}
This is not really a lift of the chain polytope, as the map involved is not linear. However, as is current practice in linear optimization, one can introduce slack variables
to linearize the minimum function and we get the following result.

\begin{corollary} \label{cor:chain polytope has a small lift}
The chain polytope $\textsc{chain}(\mathcal{P})$ of a poset $\mathcal{P}$ is the projection onto the $z$ coordinates of $(z,x)\in\RR^{\mathcal{P}} \times \RR^{\mathcal{P}}$ such that $x \in \textsc{order}(\mathcal{P})$, $0 \leq z_{a_i} \leq x_{a_i}-x_{a_j}$ for all cover relations $a_i \succ a_j$ and $z_{a_i}=x_{a_i}$ for all minimal elements $a_i$.
\end{corollary}

Corollary \ref{cor:chain polytope has a small lift} gives an explicit polyhedral lift of $\textsc{chain}(\mathcal{P})$ with $O(|\mathcal{P}|^2)$ facets, showing that one can use linear programming to efficiently optimize over the chain polytope.
Since posets can be thought of as acyclic digraphs, this allows us to efficiently optimize over sets of mutually unreachable vertices in an acyclic digraph. Further discussions on the algorithmic implications of the chain polytope can be found on Section~14.5 of \cite{schrijver2003combinatorial}.

\subsubsection{A spectrahedral lift of the chain polytope}

Next we will see that $\textsc{chain}(\mathcal{P})$ has a spectrahedral lift of size $O(|P|)$. This lift is a special case of a
general construction of spectrahedral lifts of {\em stable set polytopes} of {\em perfect graphs}. We describe the general construction below.

The {\em comparability graph} of the poset $\mathcal{P}$ is the graph whose vertices are the elements of $\mathcal{P}$ and whose edges are $(a_1,a_2)$ for all $a_1 \prec a_2$ (not just cover relations).
A {\em stable set} $S$ of a graph $G=(V(G),E(G))$ is a subset of $V(G)$ such that no pair in $S$ lies in $E(G)$. The characteristic vector of a stable set $S$ in $G$ is the vector in $\{0,1\}^{V(G)}$ with coordinate $1$ in positions indexed by $v \in S$ and $0$ everywhere else. The {\em stable set polytope} of $G$, denoted as $\textsc{stab}(G)$,
 is the convex hull of the characteristic vectors of all stable sets in $G$. Note that the antichains in a poset $\mathcal{P}$ are precisely the stable sets in its comparability graph. Hence $\textsc{chain}(\mathcal{P})$  is the stable set polytope of the
 comparability graph of $\mathcal{P}$.

 \begin{figure}[ht]
\begin{center}
\includegraphics[width=5cm]{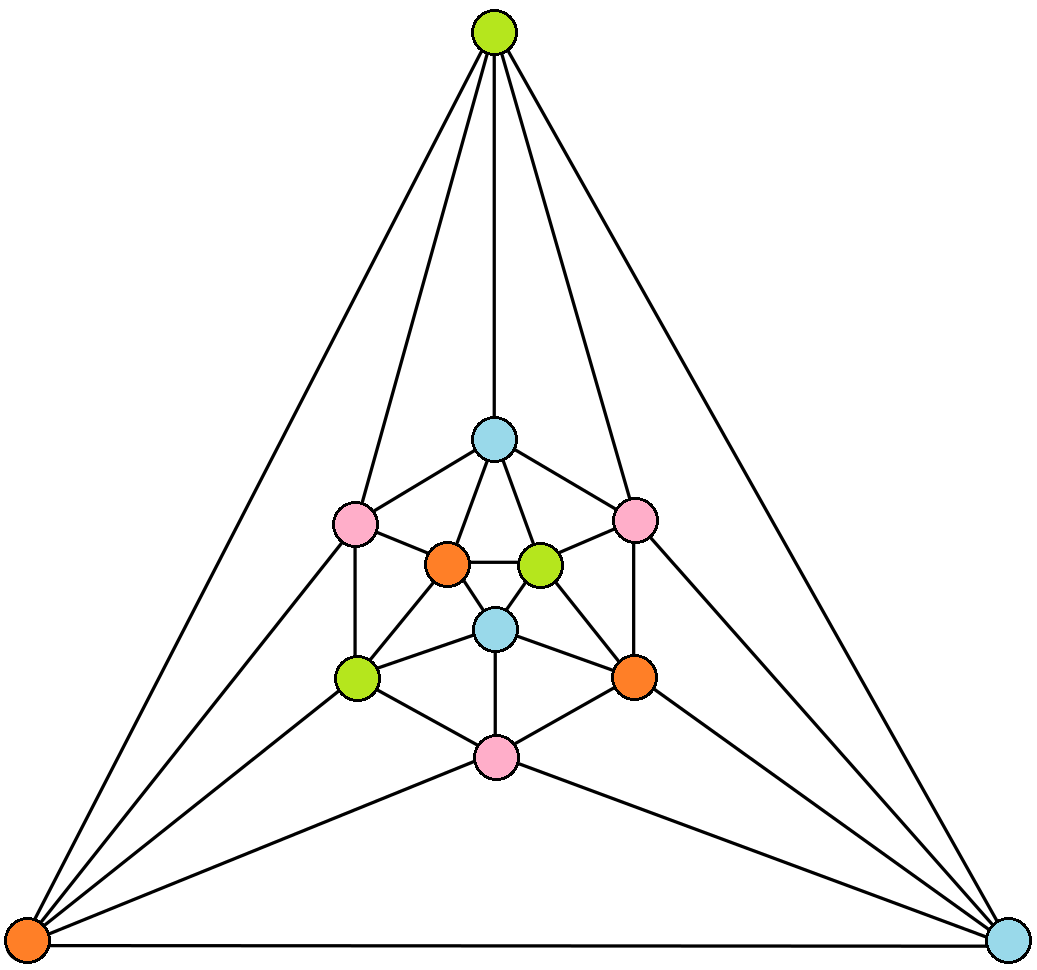}
\end{center}
\caption{\label{fig:stable sets} Icosahedral graph with an optimal $4$-coloring. Each maximal monochromatic set of nodes is actually a maximal stable set.}
\end{figure}

Recall that a {\em perfect graph} is a graph for which the {\em chromatic number} equals the {\em clique number} for every induced subgraph. In Figure \ref{fig:stable sets} we can see that the icosahedral graph is not perfect, as its chromatic number is $4$ while its maximal clique is
a triangle and hence its clique number is $3$.
In general, optimizing a linear function over the stable set polytope of a graph is NP-hard. However if the graph $G$ is perfect then $\textsc{stab}(G)$ admits a small spectrahedral lift of size $n+1$ where $n = |V(G)|$.
 In particular, we get that the largest stable set in a perfect graph can be found in polynomial time via semidefinite programming \cite{GrotschelLovaszSchrijver1984, grotschel1986relaxations}. This
construction marks the beginning of the use of semidefinite programming in discrete optimization. Comparability graphs of posets are perfect graphs (see for instance \cite[Section 5.7]{Gol04}), and hence the chain polytope $\textsc{chain}(\mathcal{P})$ has a spectrahedral lift of size $|\mathcal{P}|+1$. We now
describe the general construction from which this lift arises.

Let $G = (V(G),E(G))$ be a graph where $V(G) = [n]$. For a stable set $S \subset [n]$, let $\chi^S$ denote its characteristic vector in
$\{0,1\}^n$. Setting $x := \chi^S$, notice that
$$X := \begin{bmatrix} 1 \\ x \end{bmatrix} \begin{bmatrix} 1 &  x^\top \end{bmatrix}$$
 is a matrix
in $\PSD^{n+1}$ with the following properties:
\begin{align*}
	X_{00} = 1, \,\,\, X_{0i} = X_{i0} &= x_i \,\,\forall i \in [n], \,\,\, X_{ii} = x_i \,\,\forall i \in [n],\\
	\,\,\, X_{ij} &= 0 \,\,\forall ij \in E(G), \,\,
\rank(X) = 1.
\end{align*}
Consider the spectrahedron defined by all the above conditions except for the rank condition:
\begin{align*}
L(G) := \left\{  X \in \PSD^{n+1} \,:\, \begin{array}{l}
X_{00} = 1, \\
X_{0i} = X_{i0} = x_i \,\,\forall i \in [n], \\
X_{ii} = x_i \,\,\forall i \in [n], \\
X_{ij} = 0 \,\,\forall  ij \in E(G)
\end{array}
\right\}.
\end{align*}

Now consider the linear map $\pi$ that maps $X \in \mathcal{S}^{n+1}$ to $(X_{01}, X_{02}, \ldots, X_{0n}) \in \RR^n$.

\begin{definition} The projection
$\textsc{th}(G) := \pi(L(G)) \subset \RR^n$ is called the {\em theta body} of the graph $G$.
\end{definition}

Theta bodies were introduced in \cite{grotschel1986relaxations}. For the formulation shown above, see \cite[Chapter 9]{GLSBook}.
Since every vertex of $\textsc{stab}(G)$ is a $\chi^S$ and the matrix $X$ obtained by lifting $\chi^S$ as above lies in $L(G)$,
it follows from convexity that $\textsc{stab}(G) \subseteq  \textsc{th}(G)$.

\begin{theorem} \cite[Corollary 3.11]{grotschel1986relaxations}
\label{thm:THperfect}
$\textsc{stab}(G) =   \textsc{th}(G)$ if and only if $G$ is perfect.
\end{theorem}

Since the comparability graph of a poset $\mathcal{P}$ is perfect,  and $\textsc{chain}(\mathcal{P})$ is its stable set polytope, we
get the following corollary.

\begin{corollary}
The chain polytope $\textsc{chain}(\mathcal{P})$ admits a spectrahedral lift of size $|\mathcal{P}|+1$.
More precisely,
	\begin{multline} \textsc{chain}(\mathcal{P}) = \left\{ x \in \RR^{\mathcal{P}} \,:\, \exists \,\, Y\in \RR^{\mathcal{P} \times \mathcal{P}} \textup{ s.t. }\phantom{\begin{bmatrix} 1 & x^T \\ x & Y \end{bmatrix}}\right.\\
		\left.\begin{bmatrix} 1 & x^T \\ x & Y \end{bmatrix} \succeq 0, \,\, \textup{diag}(Y)=x, \,\,y_{ab}=0 \textup{ if }a \succ b \textup{ or } a \prec b \right\}.\end{multline}
\end{corollary}

Other examples of perfect graphs are chordal graphs and bipartite graphs. In all of these cases,
$\textsc{stab}(G)$ admits a spectrahedral lift of size $|V(G)| + 1$.

If $G=([n],E)$ is perfect, then $\textsc{stab}(G)$ also
has an explicit (although large) polyhedral description, namely,
$$\textsc{stab}(G) = \left\{ x \in \RR^n \,: \, x \geq 0, \,\sum_{i \in K} x_i \leq 1 \;\;\forall \, \textup{ cliques } K \textup{ in } G \right\}.$$
This is why the facets of the chain polytope $\textsc{chain}(P)$ are given by the maximal chain inequalities in the comparability graph of $\mathcal{P}$
and nonnegativities of variables.

We have seen that the chain polytope 
admits the polyhedral lift $\textsc{order}(\mathcal{P})$ of size $O(|\mathcal{P}|^2)$ and the spectrahedral lift
$L(G)$ of size
$|\mathcal{P}|+1$. We will see later that no spectrahedral lift of $\textsc{chain}(\mathcal{P})$ can have size smaller than
$|\mathcal{P}|+1$.
More generally, any spectrahedral lift of a polytope $P$ has size at least $\dim(P)+1$.


\subsection{Convex polynomials and their epigraphs} \label{sec:epi}
In the area of nonlinear programming, it is traditional to express optimization problems in the form
\begin{equation}
	\label{eq:nlp}
	 \min_x\; f(x)\;\;\textup{subject to}\;\; g_i(x) \leq 0\;\;\textup{for $i=1,2,\ldots,m$}
\end{equation}
for functions $f:\RR^n\rightarrow \RR$ and $g_1,\ldots,g_m:\RR^n\rightarrow \RR$.
If these are all convex functions, then such an optimization problem
is called a convex optimization problem. Since the advent of interior point methods in the 1990s, it has become
common to reformulate such convex optimization problems as
\[ \min_{x,t}\;t\;\;\textup{subject to}\;\; \begin{cases} (x,t)\in \textup{epi}(f)\\
 (x,0)\in \textup{epi}(g_i)\;\; \textup{for $i=1,2,\ldots,m$}\end{cases}\]
where $\textup{epi}(f) = \{(x,t)\in \RR^n\times \RR\;:\; f(x)\leq t\}$ is the \emph{epigraph} of $f$. Here we
have introduced an additional variable and expressed the problem as the minimization of a linear functional over a convex set.
This motivates the study of lifted representations of the epigraphs of convex functions.

\begin{figure}
\begin{center}
\includegraphics[scale=0.4, trim=80 120 80 100, clip]{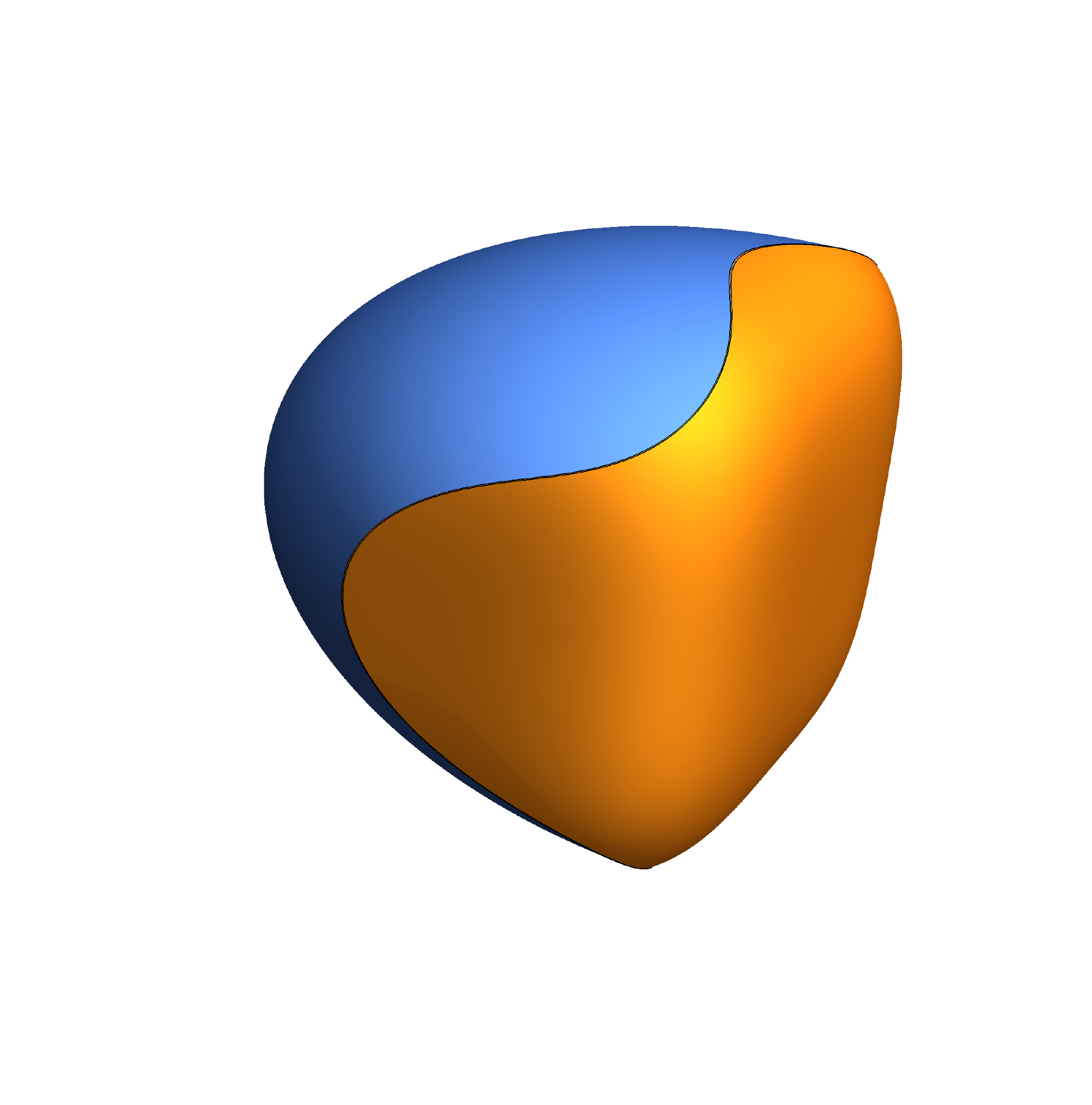}
\end{center}
\caption{\label{fig:ellipsoid-four} The intersection of the zero-sublevel sets of the (sos-)convex polynomials
$g_1(x,y,z) = x^4+y^4+z^4-3$ and $g_2(x,y,z) = (x-2)^2 + 2y^2+3z^2-5$.}
\end{figure}

If we restrict to differentiable convex functions, one approach to constructing lifted representations of their epigraphs
is to make use of the first-order characterization of convexity, i.e.,
\begin{equation}
	\label{eq:fo-conv} f(x) - [f(y) + \langle \nabla f(y),x-y\rangle] \geq 0\;\;\textup{for all $(x,y)\in \RR^{n}\times \RR^{n}$}.
\end{equation}
This inequality just tells us that the graph of $f$ is above the graph of its linear approximation at $y$,
for any $y$. The expression on the left hand side of~\eqref{eq:fo-conv} is the \emph{Bregman divergence}
associated with $f$.

\subsubsection{Convex and sos-convex polynomials}

We now restrict, further, to convex \emph{polynomials} of degree at most $2d$. 
(With  minor modifications,
the discussion extends to convex rational functions $f(x)/g(x)$ with $g(x) > 0$ for all $x$.)
The Bregman divergence of a convex polynomial $f$ of degree at most $2d$ in $n$ variables,
is a non-negative polynomial of degree $2d$
in the $2n$ variables $(x,y)$. Since the Bregman divergence is linear in $f$,
the first order characterization of convexity~\eqref{eq:fo-conv}
tells us that the set of convex polynomials of degree at most $2d$ in
$n$ variables is a slice of $\textup{Pol}_+^{2n,2d}$, the cone of
nonnegative polynomials of degree at most $2d$ in $2n$ variables. 
Denote by $\textup{Pol}^{l,2d}$, the set of all polynomials in $l$ variables of degree at most $2d$.

Whenever $n=1$ or $2d=2$ or $(n,2d) = (2,4)$, it can be shown~\cite{ahmadi2013complete} that 
the Bregman divergence
of a convex polynomial of degree at most $2d$ in $n$ variables can be written as a \emph{sum of squares}, i.e.,
there exist $f_1,f_2,\ldots \in \textup{Pol}^{2n,d}$ such that 
\[ f(x) - [f(y) + \langle \nabla f(y),x-y\rangle] = \sum_{k} f_k(x,y)^2\;\;
\textup{for all $(x,y)\in \RR^{n}\times \RR^{n}$}.\]
Otherwise, the polynomials for which the Bregman divergence is a sum of squares (called \emph{sos-convex polynomials}~\cite{helton2010semidefinite})
are a strict subset of all convex polynomials with a given degree and number of variables.

Let $\Sigma_{+}^{\ell,2d}$ denote the cone of polynomials of degree at most $2d$ in $\ell$ variables
that are sums of squares. This cone has a spectrahedral lift given by
\[ \Sigma_+^{\ell,2d} = \{f\in \textup{Pol}^{\ell,2d}\;:\; \exists Q \psd 0,\;\;\textup{such that}\;\;f(x) = v_{d}(x)^{\top}Qv_{d}(x)\;\;\forall x\in \RR^{\ell}\}.\]
Here $v_{d}(x)$ is the length $\binom{\ell+d}{d}$ vector consisting of all monomials of degree at most $d$ in $\ell$ variables. Using this, we obtain a spectrahedral lift of the cone of sos-convex polynomials.
\begin{proposition}
	\label{prop:sos-cvx}
The cone of sos-convex polynomials has the following spectrahedral lift
	\begin{multline} \left\{f\in \textup{Pol}^{n,2d}\;:\; \exists Q \psd 0\;\;\textup{s.t.}\;\;\right.\\
		\left.
		f(x) - [f(y) + \langle \nabla f(y),x-y\rangle] = v_{d}(x,y)^{\top}Qv_{d}(x,y)\;\;\forall (x,y) \in \RR^{n}\times \RR^{n}\right\}.
	\end{multline}
\end{proposition}
In the cases $n=1$, $2d=2$, and $(n,2d) = (2,4)$, for which all convex polynomials are sos-convex,
Proposition~\ref{prop:sos-cvx}, in fact, gives a spectrahedral lift of the cone of convex polynomials.

\subsubsection{Epigraphs of convex and sos-convex polynomials}

We can also use the characterization~\eqref{eq:fo-conv} to express the epigraph of any fixed convex polynomial as
an affine slice of $\textup{Pol}_{+}^{n,2d}$, and give a spectrahedral lift of
the epigraph of any sos-convex polynomial.
\begin{proposition}
\label{prop:epi-nn}
If $f$ is a convex polynomial, then
\begin{equation}
\label{eq:epi-nn}
	 \textup{epi}(f) = \{(x,t)\in \RR^{n}\times \RR\;:\; t-[f(y) + \langle \nabla f(y),x-y\rangle] \in \textup{Pol}_+^{n,2d}\}.
\end{equation}
If $f$ is a sos-convex polynomial, then
\begin{multline}
\label{eq:epi-sos}
	 \textup{epi}(f) = \left\{(x,t)\in \RR^{n}\times \RR\;:\;\exists Q \psd 0\;\;\textup{s.t.}\;\;\right.\\
	\left.t-[f(y) + \langle \nabla f(y),x-y\rangle] = v_d(y)^{\top}Qv_d(y)\;\;\forall y\in \RR^n\right\}.
\end{multline}
\end{proposition}
\begin{proof}
If $f$ is convex (respectively, sos-convex) and $(x,t)\in \textup{epi}(f)$, then $f(x) \leq t$ and so
\[ t-[f(y) + \langle \nabla f(y),x-y\rangle]  = (t-f(x)) + f(x) - [f(y) + \langle \nabla f(y),x-y\rangle]\]
is nonnegative (respectively, a sum of squares).
On the other hand, if $t-[f(y) + \langle \nabla f(y),x-y\rangle] \geq 0$ for all $y\in \RR^n$, then
putting $y=x$ we have that $t\geq f(x)$, and so $(x,t)\in \textup{epi}(f)$.
\end{proof}
Since the epigraphs of sos-convex polynomials have spectrahedral lifts,
we can reformulate convex optimization problems of the form~\eqref{eq:nlp}, in which the objective function and the constraints
are sos-convex polynomials, as semidefinite programming problems. Figure~\ref{fig:ellipsoid-four} shows the feasible region of
such a problem for two sos-convex polynomials.

\begin{example}[Epigraph of a convex quadratic and convex quadratic programming]
	If $f(x) = x^{\top}Ax + 2b^{\top}x + c$ and $A$ is positive semidefinite, then $f$ is a convex quadratic.
Optimization problems of the form~\eqref{eq:nlp} in which $f$ and the $g_i$ are all convex quadratics are
known as (convex) quadratic programs. Since any convex quadratic is sos-convex,
\begin{align*}
 \textup{epi}(f) & = \left\{(x,t)\in \RR^n\times \RR\;:\; \exists Q \psd 0\;\;\textup{s.t.}\;\;\right.\\
	&\qquad \left.
	y^{\top}Ay - 2(Ax)^{\top}y + (t-c-2b^{\top}x)  = \begin{bmatrix} 1 & y^{\top}\end{bmatrix} Q \begin{bmatrix} 1\\y\end{bmatrix}\;\;\forall y\in \RR^n\right\}\\
		& = \left\{(x,t)\in \RR^n\times \RR\;:\; \begin{bmatrix} t-c-2b^{\top}x & -(Ax)^{\top}\\-Ax & A\end{bmatrix} \psd 0\right\}.
\end{align*}
	If $f(x) = x^{\top}x$ and we consider the $t=1$ slice of $\textup{epi}(f)$,  we obtain
a well-known spectrahedral representation of the sphere:
	\[ \{x\in \RR^n\;:\; x^{\top}x\leq 1\} = \left\{x\in \RR^n\;:\; \begin{bmatrix} 1 & -x^{\top}\\-x & I\end{bmatrix} \psd 0\right\} = \left\{x\in \RR^n\;:\; \begin{bmatrix} 1 & x^{\top}\\x & I\end{bmatrix} \psd 0\right\}\]
		where the last description is obtained via a congruence by $\left[\begin{smallmatrix} 1 & 0\\0 & -I\end{smallmatrix}\right]$, which preserves positive semidefiniteness.
\end{example}

\subsection{The honeycomb lift of the Horn cone}

Given the eigenvalues of two $n \times n$ Hermitian matrices $A$ and $B$, a classical question from linear algebra is to characterize the eigenvalues of $A+B$. Equivalently, the problem is to characterize the triplets $(\lambda, \mu, \nu) \in (\RR^n)^3$
where $\lambda$ is the vector of eigenvalues of an $n \times n$ Hermitian matrix $A$ in weakly-decreasing order,
$\mu$ that of $B$, and $\nu$ that of $C$, such that $A+B+C=0$. Call such a triplet $(\lambda, \mu, \nu) \in (\RR^n)^3$ {\em valid}.
The question of characterizing valid triplets was posed by Hermann Weyl in 1912. He gave a
number of necessary conditions, starting with the obvious trace condition that $\sum \lambda_i + \sum \mu_i + \sum \nu_i = 0$.

In 1962 Horn conjectured that the set of valid triplets forms a convex polyhedral cone and proposed a complete set of linear inequalities that describes it
via a recursive recipe  \cite{Horn}. Let $\Horn_n \subset \RR^{3n}$ be the set of valid triplets $(\lambda,\mu,\nu)$.
The inequalities conjectured by Horn have the form
$$ \lambda_{i_1} + \ldots + \lambda_{i_r} + \mu_{j_1} + \ldots + \mu_{j_r} + \nu_{k_1} + \ldots + \nu_{k_r} \geq 0$$
where $1 \leq r < n$ and all triplets of indices $1 \leq i_1 < \ldots < i_r \leq n, 1 \leq j_1 < \ldots < j_r \leq n, 1 \leq k_1 < \ldots < k_r \leq n$
are in a set $T_{n,r}$. Horn's conjecture described the set $T_{n,r}$ recursively in terms of sets $T_{r,s}$ for $1 \leq s < r$. This recipe
leads to an exponential number of inequalities for $\Horn_n$.

The Horn conjecture was proved by a combination of the works of Klyachko \cite{Klyachko} and Knutson-Tao \cite{KnutsonTaoJAMS,KnutsonTao}. Remarkably the work of Knutson-Tao also shows that the convex set $\Horn_n$ admits a small polyhedral lift of size $O(n^2)$.



\begin{theorem} The Horn cone for $n \times n$ Hermitian matrices admits a polyhedral lift called $\textsc{honey}_n$ with $O(n^2)$ facets.
\end{theorem}

In this short section we describe the lift of Knutson-Tao but we do not attempt to explain the proof, which uses sophisticated arguments from algebraic geometry and representation theory (we refer the interested reader to \cite{KnutsonTao} for more details).

Consider a honeycomb structure as shown in Figure \ref{fig:hcomb}. The honeycomb is formed of internal hexagons, and has $3n$ outgoing edges, each labeled by $\lambda_i$, $\mu_i$ or $\nu_i$ in the way shown in the Figure.

\newcommand{\honeycomb}[2]{\begin{tikzpicture}[hexa/.style= {shape=regular polygon,regular polygon sides=6,draw,minimum size=1cm,rotate=30}]

\def\Nn{#1}
\pgfmathtruncatemacro{\Nhex}{\Nn-2}

\foreach \j in {1,...,\Nhex}{%
  \foreach \i in {1,...,\j}{%
    \node[hexa] (h\j;\i) at ({(\i-\j/2)*sin(60)},{-\j*0.75}) {};} }

\IfEqCase{#2}{%
  {1}{\draw[ultra thick,black] (h1;1)--++(0,-1)--++(0.866/2,-0.5/2)--++(0,-0.5);}
  {0}{}}

\foreach \j in {1,...,\Nn}{%
  \coordinate(l\j) at ({(1-(\Nn-\j+1)/2)*sin(60)},{0.75-(\Nn-\j+1)*0.75+cos(60)/2});
  \draw (l\j) -- + ({-sin(60)/2},{cos(60)/2}) node[at end,left] {$\lambda_{\j}$};
}

\foreach \j in {1,...,\Nn}{%
  \coordinate(m\j) at ({(\j/2)*sin(60)},{0.75-\j*0.75+cos(60)/2});
  \draw (m\j) -- + ({sin(60)/2},{cos(60)/2}) node[at end,right] {$\mu_{\j}$};
  }

\foreach \j in {1,...,\Nn}{%
  \coordinate (n\j) at ({(-1+(\Nn-\j+1)-\Nhex/2)*sin(60)},{-\Nhex*0.75-0.5});
  \draw (n\j) -- + (0,-1/2) node[at end,below] {$\nu_{\j}$};
}

\draw (h1;1)--+(0,1);
\draw (h\Nhex;1)--+(-0.866,-0.5);
\draw (h\Nhex;\Nhex)--+(0.866,-0.5);

\end{tikzpicture}}

\begin{figure}[ht]
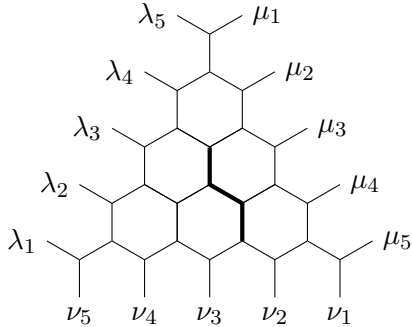

  \centering

\honeycomb{5}{1}

\caption{The honeycomb structure used to describe the lift of $\Horn_n$ for the case $n=5$}
\label{fig:hcomb}

\end{figure}

Let $E_n$ be the set of edges in the $n$-honeycomb. It is not hard to see that $|E_n| = O(n^2)$. Let $\Gamma_n \subset E_n \times E_n$ be the pairs of edges $(a,b) \in E_n \times E_n$ that are in one of the configurations shown in Figure \ref{fig:configshoneycomb}. An example of such a configuration is highlighted in Figure \ref{fig:hcomb}.

\def\NEconstraint{\tikz[baseline=-5ex]{
\draw[very thick] (0,0)--node[below,pos=0.3] { $a$}(-30:0.75) --++ (30:0.75) -- node[above,pos=0.7] { $b$}  +(-30:0.75);}}

\def\NWconstraint{\tikz[baseline=.1ex]{
\draw[very thick] (0,0)--node[right,midway] { $a$}(0,0.75) --++ (150:0.75) -- node[left,midway] { $b$}  +(0,0.75);}}

\def\Sconstraint{\tikz[baseline=.1ex]{
\draw[very thick] (0,0)--node[below,midway] { $b$}(30:0.75)--++(0,0.75)-- node[above,midway] { $a$} +(30:0.75);}}

\begin{figure}[ht]
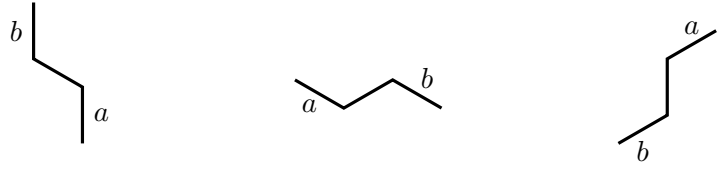

\centering
\NWconstraint \qquad\qquad \qquad   \NEconstraint \qquad \qquad \qquad\Sconstraint
\caption{Configurations of pairs of edges $(a,b)$ in a honeycomb that appear in the lift of $\Horn_n$. Note that these three patterns are $120^{\circ}$ rotations of each other.}
\label{fig:configshoneycomb}
\end{figure}

We are now ready to describe the lift. Consider the polyhedral cone $\Honey_n$ in $\RR^{E_n}$ defined by:
\begin{multline}
\label{eq:lifthoney}
\Honey_n = \left\{e \in \RR^{E_n} \; : \; e_{a} - e_{b} \geq 0  \;\; \forall (a,b) \in \Gamma_n\right.
\\
\left.e_{a} + e_{b} + e_{c} = 0 \;\; \forall \text{edges $a,b,c$ meeting at a common vertex} \right\}.
\end{multline}
Consider the map $\partial$ that projects a vector $e \in \RR^{E_n}$ onto the $3n$ outgoing edges  of the honeycomb. The result of Knutson and Tao shows that 
\begin{multline}
\Horn_n = \left\{ (\lambda,\mu,\nu) \in \RR^{3n} \; : \; \lambda_1 \geq \dots \geq \lambda_n, \;\; \mu_1 \geq \dots \geq \mu_n, \;\; \nu_1 \geq \dots \geq \nu_n\right. \\
 \left.  \text{ and } \exists e \in \Honey_n \text{ s.t. } (\lambda,\mu,\nu) = \partial(e) \right\}.
\end{multline}
\begin{example}

\begin{figure}[ht]
\centering
\includegraphics[scale=1]{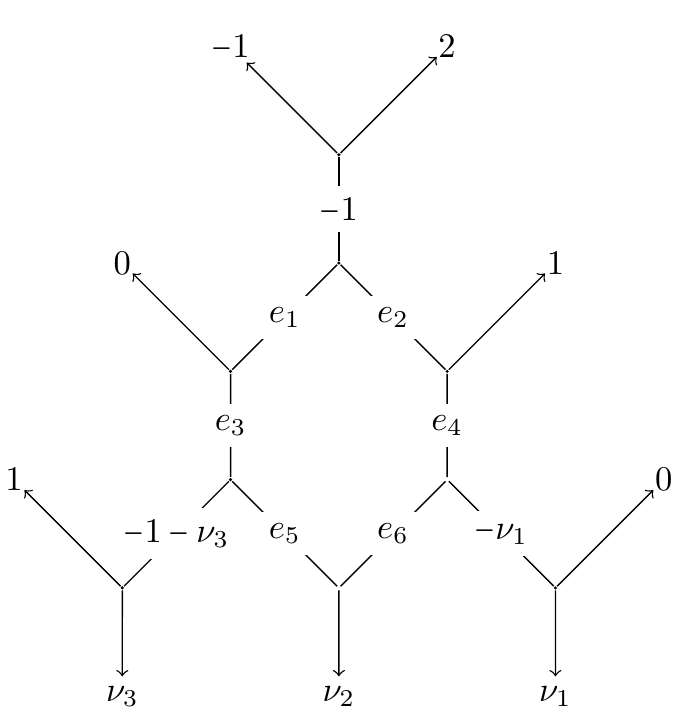}
\caption{A $3$-honeycomb. \label{fig:3honeycomb}}
\end{figure}

Take $n=3$ and suppose $\lambda = (1,0,-1)$, $\mu = (2,1,0)$ and $\nu=(\nu_1,\nu_2, \nu_3)$
are the vectors of eigenvalues of $A, B, C$ such that $C = -A-B$ and all matrices are $3 \times 3$ and Hermitian.
We would like to use honeycombs to decide the possible values of $\nu$.

A $3$-honeycomb with boundary values $\lambda, \mu, \nu$ is shown in Figure~\ref{fig:3honeycomb}.
There are $9$ internal edges in this honeycomb, of which three are determined by the boundary values. The remaining edges are labeled $e_1, \ldots, e_6$ as shown. Writing the second condition in \eqref{eq:lifthoney} for all vertices of the honeycomb, yields the following system of equations:
\begin{equation}
\label{eq:3honeycombeqs}
\begin{aligned}
&e_1 + e_2 = 1 \,\, \Rightarrow \,\, e_2 = 1 - e_1,\\
	&e_1 + e_3 = 0 \,\, \Rightarrow \,\, e_3 = -e_1, \\
	&e_3 + e_5 = 1 + \nu_3 \,\, \Rightarrow \,\, e_5 = 1 + \nu_3 + e_1,\\
	&e_5 + e_6 = -\nu_2 \,\, \Rightarrow \,\, e_6 = -1 - \nu_2 - \nu_3 - e_1, \\
	&e_4 + e_6 = \nu_1 \,\, \Rightarrow \,\, e_4 = 1 + (\nu_1 + \nu_2 + \nu_3) + e_1,\\ 
	&e_2 + e_4 = -1 \,\, \Rightarrow \,\, e_2 = -2 - (\nu_1 + \nu_2 + \nu_3) - e_1.
\end{aligned}
\end{equation}
The first and last equations imply the trace equation $\nu_1 + \nu_2 + \nu_3 = -3$. We have eliminated all the $e$ variables
except $e_1$.

Now we impose the nonnegativity constraints coming from the first condition in \eqref{eq:lifthoney}. For example we can identify a configuration of the leftmost type in Figure \ref{fig:configshoneycomb} that yields $-1-e_3 \geq 0$. Using the fact that $e_3 = -e_1$ derived in \eqref{eq:3honeycombeqs} gives $e_1 \geq 1$.
Proceeding like this will all configurations of edges shown in Figure \ref{fig:configshoneycomb} we arrive at the following inequalities:
\begin{align*}
&1 \leq e_1 \leq 2, \,\,\,\, e_1 \geq 2 + \nu_1 + \nu_2, \,\,\,\, e_1 \geq 1 + \nu_1, \,\,\,\, e_1 \geq -\nu_2, \,\,\,\, e_1 \geq 2 + \nu_2, \\
& e_1 \leq 3 + \nu_1 + \nu_2, \,\,\,\, e_1 \leq 2 + \nu_1.
\end{align*}
In addition, we also have the chamber inequalities: $\nu_1 \geq \nu_2 \geq -3-\nu_1 - \nu_2$.
The possible values of $\nu$ are those that allow a value of $e_1$ to exist given the above inequalities.

For a concrete example consider the following triple of symmetric matrices
$$ A = \begin{bmatrix} 0 & 0 & 0 \\ 0 & 0 & 1 \\ 0 & 1 & 0\end{bmatrix}, \,\,\,
B = \begin{bmatrix} 2 & 0 & 0 \\ 0 & 1 & 0 \\ 0 & 0 & 0\end{bmatrix} \textup{ and }
C = -A-B = \begin{bmatrix} -2 & 0 & 0 \\ 0 & -1 & -1 \\ 0 & -1 & 0
\end{bmatrix}$$
which have eigenvalues $\lambda = (1,0,-1)$, $\mu = (2,1,0)$ and $\nu = (\frac{-1 + \sqrt{5}}{2},  \frac{-1-\sqrt{5}}{2}, -2)$ respectively.
Plugging in  $\nu_1 = \frac{-1 + \sqrt{5}}{2}$ and $\nu_2 = \frac{-1-\sqrt{5}}{2}$ into the inequalities we should arrive at a feasible value for $e_1$
and indeed, we get $\frac{1 + \sqrt{5}}{2} \leq e_1 \leq 2$ which is satisfiable.

In this example, $\textsc{honey}_3$ can be thought of as a cone in $\RR^{6+9}$ with variables indexed by $e_1, \ldots, e_6$ and
the $9$ variables in $\lambda, \mu, \nu$.
It could be in $\RR^{9+9}$ but recall that the variables associated to three of the bounded edges can be eliminated right away
and expressed in terms of $\lambda, \mu, \nu$. The Horn cone is in $\RR^9$. When we fix $\lambda$ and $\mu$ as we have done in this
example, we are looking at a slice of $\textsc{honey}_3$. The method above derives the inequality description of this slice.


\end{example}

\section{Slack Operators}
\label{sec:slackoperator}
\label{sec:slacks}

A fundamental question concerning lifts is to characterize when they exist. More formally,
given a convex set $C$ and a convex cone $K$, how can we decide whether $C$ has a $K$-lift?
In this section we focus on answering this basic existence question.
In order to do this, we introduce a central tool for studying lifts of convex sets,
namely the {\em slack operator} of the set. We will see that the existence of a
$K$-lift of a convex set $C$ is equivalent to the
existence of a {\em factorization} of the slack operator of $C$ through the cone $K$,
and its dual cone $K^\ast$. This connection provides us with a useful algebraic handle on the
geometric picture of cone lifts of convex sets.


\subsection{Slack matrices}

In the case of polyhedral lifts of polytopes, their existence was characterized in~\cite{yannakakis1991expressing}. A key component turns out to be the \emph{slack matrix} of a polytope, which we introduce in this section.

There are two traditional ways of representing a polytope $P$.  We can represent $P$ as a $\cV$-polytope, i.e., as the convex hull of a minimal finite set of points, its vertices.  We can also represent $P$ as a $\cH$-polytope, i.e.,  as a bounded intersection of a finite number of closed half-spaces.  It is a fundamental property of polytopes that these two representations are equivalent.  Thus we would commonly either list all the vertices of $P$, or list a minimal set of affine inequalities that cut out $P$, more precisely, one inequality per facet of $P$.


\begin{example} \label{ex:d+4}
Consider the polytope $P$ in $\RR^3$, shown on the left in Figure~\ref{fig:d+4example}, with the seven vertices $(\pm 1,\pm 1,-1/2)$, $(\pm 1, -1, 1/2)$ and $(0,-1,3/2)$. This polytope can also be described by the seven affine inequalities that cut it out:
$$1 - x \geq 0,\quad 1 + x \geq 0, \quad 1 + 2z \geq 0, \quad 1 + y \geq 0, \quad 1 -2y - 2z \geq 0,$$
$$1 - x - \frac{1}{2}y - z \geq 0, \quad \quad 1 + x -\frac{1}{2} y - z \geq 0.$$

\begin{figure}[h]
\begin{center}
\includegraphics[width=5cm]{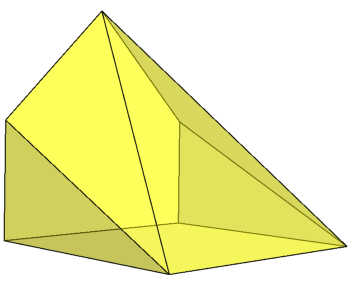} \hspace{2cm} \includegraphics[width=5cm]{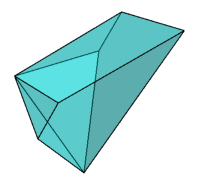}
\end{center}
\caption{Polytope $P$ with seven vertices and seven facets,  and its polar $P^\circ$.}
\label{fig:d+4example}
\end{figure}

\end{example}

The slack matrix of $P$ combines the $\mathcal{V}$ and $\mathcal{H}$ descriptions of $P$ into a single, less parsimonious, description. To do this we compute the {\em slack},  $\beta - h^\top v$, of each vertex $v$ in each facet defining inequality $h^\top x \leq \beta$, obtaining a nonnegative matrix, with rows indexed by the facets and columns indexed by the vertices of the given polytope $P$. For instance, the slack matrix of the polytope
$P$ from Example~\ref{ex:d+4} is the following $7\times 7$ matrix:
$$S_P=\left[
\begin{array}{ccccccc}
 2 & 2 & 2 & 0 & 0 & 0 & 1 \\
 0 & 0 & 0 & 2 & 2 & 2 & 1 \\
 0 & 0 & 2 & 0 & 0 & 2 & 4 \\
 0 & 2 & 0 & 0 & 2 & 0 & 0 \\
 4 & 0 & 2 & 4 & 0 & 2 & 0 \\
 3 & 2 & 2 & 1 & 0 & 0 & 0 \\
 1 & 0 & 0 & 3 & 2 & 2 & 0 \\
\end{array}
\right].$$
The rows of $S_P$ are indexed by the facets of $P$ in the order they are listed above. The columns of $S_P$ are
indexed by the vertices of $P$ in the following order:
$$(-1,-1,-\frac{1}{2}), (-1,1,-\frac{1}{2}), (-1,-1,\frac{1}{2}), (1,-1,-\frac{1}{2}), (1,1,-\frac{1}{2}), (1,-1,\frac{1}{2}), (0,-1,\frac{3}{2}).$$
Since the inequality defining a facet of $P$ is only defined up to scaling by a positive number, the slack matrix $S_P$
is only defined up to positive scaling of rows.

There is another interesting interpretation of the slack matrix, that brings the duality theory of polytopes to the forefront.
Recall that given a polytope $P \subset \RR^n$ with the origin in its interior, one can define its {\em polar} as the set
$$P^{\circ}=\{y \in \RR^n \;:\;  \langle x, y \rangle \leq 1, \textrm{ for all } x \in P\}$$
which is again a polytope. A fundamental result in the duality of polytopes is that there is a
one-to-one correspondence between the vertices of $P^\circ$ and the facets of $P$,
that sends vertex $y$ of $P^{\circ}$ to the facet of $P$ cut out by the inequality $1-\langle x, y \rangle \geq 0$.
In our example, this means that $P^\circ$ is the polytope with vertices $(\pm 1,0,0)$, $(0,0,-2)$,
$(0,-1,0)$, $(0,2,2)$, $(1,1/2,1)$ and $(-1,1/2,1)$,  shown on the right in Figure~\ref{fig:d+4example}.
Note that if the facet inequalities of $P$ are scaled so that they are of the form
$1 - \langle x, a_i \rangle \geq 0$ (which is possible since the origin lies in the interior of $P$), then the transposed matrix $S_P^\top$ is a slack matrix
of $P^\circ$. With the language of duality, we can think of the slack matrix $S_P$ as the map  $1-\langle x,y\rangle$
evaluated at the vertices of $P^{\circ} \times P$.

\subsection{Nonnegative factorizations and Yannakakis' theorem}

In order to characterize the existence of polyhedral lifts for a polytope, we
need to introduce the notion of \emph{nonnegative rank} of a matrix.  A {\em
nonnegative factorization} of a nonnegative matrix $M \in \RR_+^{n \times k}$
is a pair of nonnegative matrices $A \in \RR_+^{m \times n}$ and $B \in
\RR_+^{m \times k}$ such that $M=A^\top B $. We say $m$, the inner dimension of
the product, is the \emph{size} of the factorization, and define the
\emph{nonnegative rank} of $M$ to be the smallest $m$ for which a nonnegative
factorization of $M$ exists.  Computing nonnegative ranks of matrices is a hard
problem. Nonnegative factorizations and nonnegative rank have
many interpretations and applications in diverse fields. An early
survey can be found in \cite{cohen1993nonnegative}, but there is a growing body
of literature on this topic. Somewhat surprisingly, it turns out that this hard
algebraic problem encodes the geometric problem of interest---the existence of polyhedral lifts.


\begin{theorem}[Yannakakis \cite{yannakakis1991expressing}] The linear extension complexity of a polytope equals the nonnegative rank of its slack matrix. \label{thm:yannakakis}
\end{theorem}
\begin{proof}
Suppose a polytope $P$ is cut out by the facet inequalities $h_i^\top x \leq \beta_i$ for $i=1,\dots,f$,
	and that $S_P=A^\top B$ is a nonnegative factorization of its slack matrix, with inner dimension $m$.
	If we let $a_i$, for $i=1,2,\ldots,f$, denote the columns of $A$, then we will show that 
\begin{equation} \label{eq:polylift}
	P=\left\{ x \in \RR^n \;:\; \exists y\in \RR_+^m\; \textup{s.t.}\; a_i^\top y = \beta_i - h_i^\top x,\;\textup{for all $i=1,\dots,f$} \right\}.
\end{equation}
This equality rewrites $P$ as the projection of a polytope with $m$ facets, or, equivalently, as
	the 
	projection of a slice of $\RR_+^m$. Hence $P$ has a size $m$ polyhedral lift.

	To see why~\eqref{eq:polylift} holds, first note that if $p$ is in the right hand side we must have
$h_i^\top p = \beta_i - a_i^\top y$ for some nonnegative $y$. Then, since $a_i^\top y$ is nonnegative, it
	follows that $h_i^\top p \leq \beta_i$ for all $i$, and so $p\in P$. On
	the other hand, if $v_j$ is a vertex of $P$, let $y=b_j$ be the
	corresponding column of $B$ which is nonnegative. Then $a_i^\top y = a_i^\top b_j = [S_P]_{ij} =
	\beta_i - h_i^\top v_j$. Consequently
	$v_j$, and indeed all vertices of $P$, are in the right hand side set. Since
	the right hand side set is convex, all of $P$ must be contained in it,
	and we have equality between the two sets.

Suppose now that $P$ can be written as the projection of a polytope $Q$ with $m$ facets. Take the slack matrix ${S}_Q$ and keep only columns that correspond to vertices that project to vertices of $P$, and call this reduced $m \times v$ matrix $\widetilde{S_Q}$. Now any facet inequality on $P$ can be pulled back by the projection to a valid inequality on $Q$, so by a version of Farkas Lemma, it can be written as a nonnegative combination of facet inequalities of $Q$. Record the coefficients as column vectors, and form a matrix $A \in \RR_+^{m \times f}$ with those columns. Then $S_P= A^\top \widetilde{S_Q}$ is a nonnegative factorization of $S_P$ with inner dimension $m$.
\end{proof}

\begin{example} Revisiting Example~\ref{ex:d+4}, one can now ask for the linear extension complexity of the polytope $P$. 
	We will prove that it is at least $6$ in Example~\ref{eg:oct-lb}, using basic considerations about how lifts interact with 
	facial structure that will be introduced in Section~\ref{sec:obstructions and lower bounds}. 
	To show that there is a lift with $6$ facets,
	Theorem~\ref{thm:yannakakis} tells us that it is enough to find a nonnegative factorization of
	the slack matrix $S_P$, of size $6$. One such factorization is
$$\left[
\begin{array}{ccccccc}
 2 & 2 & 2 & 0 & 0 & 0 & 1 \\
 0 & 0 & 0 & 2 & 2 & 2 & 1 \\
 0 & 0 & 2 & 0 & 0 & 2 & 4 \\
 0 & 2 & 0 & 0 & 2 & 0 & 0 \\
 4 & 0 & 2 & 4 & 0 & 2 & 0 \\
 3 & 2 & 2 & 1 & 0 & 0 & 0 \\
 1 & 0 & 0 & 3 & 2 & 2 & 0
\end{array}
\right]=
\left[
\begin{array}{cccccc}
 2 & 0 & 0 & 0 & 0 & 1  \\
 0 & 0 & 0 & 0 & 2 & 1  \\
 0 & 0 & 0 & 2 & 0 & 4  \\
 0 & 0 & 2 & 0 & 0 & 0 \\
 0 & 4 & 0 & 2 & 0 & 0  \\
 2 & 1 & 0 & 0 & 0 & 0  \\
 0 & 1 & 0 & 0 & 2 & 0
\end{array}
\right]
\left[
\begin{array}{ccccccc}
 1 & 1 & 1 & 0 & 0 & 0 & 0 \\
 1 & 0 & 0 & 1 & 0 & 0 & 0 \\
 0 & 1 & 0 & 0 & 1 & 0 & 0 \\
 0 & 0 & 1 & 0 & 0 & 1 & 0 \\
 0 & 0 & 0 & 1 & 1 & 1 & 0 \\
 0 & 0 & 0 & 0 & 0 & 0 & 1
\end{array}
\right].$$
It is not hard to recover the lift $Q$ from the factorization. We just start with an inequality description of $P$ and use the factorization to write the lift as explicitly described in \eqref{eq:polylift}. We then simply use the equality constraints to eliminate some the variables $y$.
\qed
\end{example}


\begin{example}

Recall that in Corollary \ref{cor:chain polytope has a small lift} we have shown that the chain polytope $ \textsc{chain}(\mathcal{P})$
of a poset $\mathcal{P}$ has a small polyhedral lift.
By Yannakakis' Theorem, the slack matrix of $ \textsc{chain}(\mathcal{P})$ must have small nonnegative
	factorization, and we will now show that this is indeed the case.

The vertices of $ \textsc{chain}(\mathcal{P})$
	are indexed by antichains of $\mathcal{P}$. The facets,
	apart from the nonnegativities of the variables, are indexed by maximal chains in $\mathcal{P}$.
	The facet inequality corresponding to a maximal chain $C$ is of the form $\sum_{a \in C} x_a \leq 1$.
	As such, given an antichain $A$ and a maximal chain $C$,
	the $(C,A)$-entry of the slack matrix is $0$ if $C$ and $A$ intersect, and $1$ otherwise.

In Corollary \ref{cor:chain polytope has a small lift} we saw that $ \textsc{chain}(\mathcal{P})$ has a lift of the form $(z,x) \in \mathbb{R}^{\mathcal{P} \times \mathcal{P}}$ given by the inequalities
$$ \left\{ \begin{array} {ll}
0 \leq z_a, &  \textrm{ for all } a \in \mathcal{P},\\
0 \leq x_{a_i}-x_{a_j}-z_{a_i}, & \textrm{ for all cover relations } a_i\succ a_j, \\
0 \leq 1-x_a, &  \textrm{ for all } a \in \mathcal{P}_{max},\\
\end{array}
\right.  $$
together with the equalities  $z_a=x_a$ for $ a \in \mathcal{P}_{\min}$.
Here we denote by $\mathcal{P}_{\min}$ and $\mathcal{P}_{\max}$ the set of minimal, respectively maximal, elements of $\mathcal{P}$. We will also use $E$ to be the set of all cover relations.
The proof of Theorem \ref{thm:yannakakis} then gives us a way of extracting a factorization.

For every antichain $A$, we will assign a vector that is simply the preimage of its indicator vector $\chi_A$ through the lift, evaluated at each of the facet inequalities.
The preimage of $\chi_A$ is $(\chi_A,\chi_F)$ where $F$ is the filter over $A$, i.e.,
	the set of all elements  greater than or equal to
 some element of the antichain.
	Evaluating this at each of the defining inequalities above gives a vector $(\chi_A,y,w) \in \{0,1\}^{\mathcal{P} \times \mathcal{P}_{\min} \times E \times \mathcal{P}_{\max}}$ with $y_{a\succ b}=1$ if and only if $a$ is in $F \setminus A$ and $b\notin F$, and $w_a=1$ if and only if $a \in\mathcal{P}_{\max}\setminus F$.

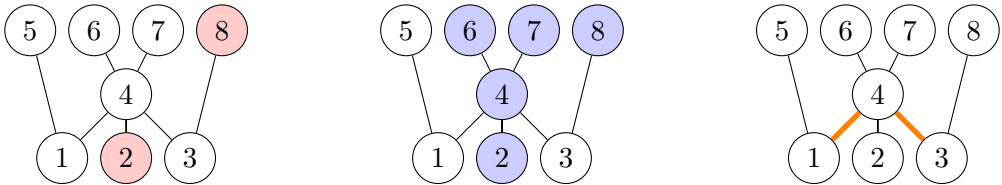
\begin{figure}[ht]
\begin{tikzpicture}[scale=0.85]
  \node[circle,draw] (1) at (-1,0) {1};
  \node[circle,draw,fill=red!20] (2) at (0,0) {2};
  \node[circle,draw] (3) at (1,0) {3};
  \node[circle,draw] (4) at (0,1) {4};
  \node[circle,draw] (5) at (-1.5,2) {5};
  \node[circle,draw] (6) at (-0.5,2) {6};
    \node[circle,draw] (7) at (0.5,2) {7};
  \node[circle,draw,fill=red!20] (8) at (1.5,2) {8};
  \draw  (5) -- (1) -- (4)--(2)--(4)--(3)--(8);
  \draw  (6)--(4)--(7);
\end{tikzpicture} \hspace{1.5cm}
\begin{tikzpicture}[scale=0.85]
  \node[circle,draw] (1) at (-1,0) {1};
  \node[circle,draw,fill=blue!20] (2) at (0,0) {2};
  \node[circle,draw] (3) at (1,0) {3};
  \node[circle,draw,fill=blue!20] (4) at (0,1) {4};
  \node[circle,draw] (5) at (-1.5,2) {5};
  \node[circle,draw,fill=blue!20] (6) at (-0.5,2) {6};
    \node[circle,draw,fill=blue!20] (7) at (0.5,2) {7};
  \node[circle,draw,fill=blue!20] (8) at (1.5,2) {8};
  \draw  (5) -- (1) -- (4)--(2)--(4)--(3)--(8);
  \draw  (6)--(4)--(7);
\end{tikzpicture} \hspace{1.5cm}
\begin{tikzpicture}[scale=0.85]
  \node[circle,draw] (1) at (-1,0) {1};
  \node[circle,draw] (2) at (0,0) {2};
  \node[circle,draw] (3) at (1,0) {3};
  \node[circle,draw] (4) at (0,1) {4};
  \node[circle,draw] (5) at (-1.5,2) {5};
  \node[circle,draw] (6) at (-0.5,2) {6};
  \node[circle,draw] (7) at (0.5,2) {7};
  \node[circle,draw] (8) at (1.5,2) {8};
  \draw  (5) -- (1) -- (4)--(2)--(4)--(3)--(8);
  \draw(6)--(4)--(7);
  \draw[orange,line width=2pt]  (1)--(4)--(3);
\end{tikzpicture}

	\caption{ \label{fig:antichainfilter} Example of an antichain (left), the corresponding filter (center), and the vector $y$ of the factorization (right).}
\end{figure}

For example, in Figure \ref{fig:antichainfilter} we can see an antichain $A$ and its corresponding filter $F$ in red and blue respectively. We would, in this case, have $\chi_A=(0,1,0,0,0,0,0,1)$, since the antichain contains only the nodes $4$ and $8$, while $\chi_F=(0,1,0,1,0,1,1,1)$. Setting $z=\chi_A$ and $x=\chi_F$ and looking at the inequalities above, the only interesting ones are those coming from the cover relations that will give us the vector $y$ indicated in the figure. It is the set of edges in the Hasse diagram that go from a node not in the filter to a node that is in the filter but not in the antichain.

As for the facets, the nonnegativities of the variables in $\textsc{chain}(\mathcal{P})$ lift to nonnegativities of the variables
$z_a$ in the lift, so we get the vector $(e_a,0,0)\in \{0,1\}^{\mathcal{P}  \times E \times \mathcal{P}_{\max}}$
where $e_a$ is $1$ only in coordinate $a$. The inequality for a maximal chain $C$ of the form
$a_0 \prec a_1 \prec \cdots \prec a_N$, lifts to
$1 - \sum_{i=0}^N {z_{a_i}} \geq 0$, which is the sum of the inequalities corresponding to the
 maximal element in $C$ and also all the cover relations. So we get a vector
$(0,\chi_C,e_{a_N})\in \{0,1\}^{\mathcal{P} \times E \times \mathcal{P}_{\max}}$, where $\chi_C$ is simply the indicator vector of which cover relations are in the chain.

Since we are following the proof of Theorem \ref{thm:yannakakis}, these must necessarily form a factorization. It is not hard to see this independently. For the nonnegativity inequalities this is trivial, so we just have to see that given an antichain $A$ and a chain $C$ the inner product $\langle (\chi_A,y,z), (0,\chi_C,e_{a_N}) \rangle = \langle y, \chi_C \rangle + z_{a_N}$ is $0$ if the chain intersects the antichain and $1$ otherwise. But $\langle y, \chi_C \rangle$ is $1$ if and only if the chain enters the filter outside of an antichain element, while $z_{a_N}$ is $1$ if and only if the chain never enters the filter at all, so they can never be $1$ at the same time. If they are both $0$ it means that the chain did enter the filter at an element of the antichain, and so we have indeed a factorization.
This is a slightly different factorization from the classical one originally given by Yannakakis \cite{yannakakis1991expressing}, but of roughly the same size.


\end{example}

\subsection{The slack operator}

Since we are interested in lifts of convex sets in general and not just polytopes, we will next need to generalize the notion of slack matrix to a more continuous object, the slack operator of a convex set. To simplify the exposition we will always consider compact convex sets with the origin in the interior, and will refer to these as \emph{convex bodies}.
Given a convex body $C$, its polar is $$C^{\circ}=\{y \in \RR^n \;:\;  \langle x, y \rangle \leq 1, \textrm{ for all } x \in C\},$$
which is again a convex body. Recall that an \emph{extreme point} of a convex set $C$ is a point in $C$ that cannot be written as a convex combination of two distinct points of $C$, i.e., it is not in the relative interior of any line segment contained in $C$. We will denote the set of all extreme points of $C$ as $ \ext(C)$. Generalizing the notion of extreme point, a \emph{face} of a convex set $C$ is a convex subset $F \subseteq C$ such that if any point in $F$ is in the relative interior of a line segment contained in $C$, then the entire line segment is contained in $F$.

Consider again the function $1-\langle y, x \rangle$ but now with domain $C^\circ \times C$. In this domain, it is a nonnegative function, and imduces a one-to-one correspondence between points $\tilde{y} \in C^{\circ}$ and linear inequalities valid over $C$, simply given by $1-\langle\tilde{y},x\rangle  \geq 0$. Since $1-\langle y, x \rangle$ is affine
in $x$ and $y$ (separately), and $C$ and $C^\circ$ are compact,
its value at any point is completely determined by its value on $\textup{ext}(C^\circ)\times \textup{ext}(C)$.
We call the resulting map the $\emph{slack operator}$ of $C$.

\begin{definition}
	Given a convex body $C$, its slack operator is the map
$s_C: \ext (C^\circ) \times \ext(C) $ defined as $s_C(y,x)=1-\langle x, y \rangle$.
\end{definition}

When $C$ is a polytope, both $C$ and $C^\circ$ have a finite number of extreme points, and we recover the slack matrix.

\begin{example} \label{ex:slackoperator}
Let $C  = \{ (x,y) \in \RR^2 \,:\, (1+x)^2(x-1) + y^2 \leq 0, \; x \geq -1 \}$. This is a convex set bounded by a cubic curve. Its boundary is parametrized by $$\left(1-v^2,2v-v^3\right)$$
with $v \in [-\sqrt{2},\sqrt{2}]$, and every boundary point is extreme. The polar of $C$ is the convex hull of the cardioid defined by
$4 x^4 + 32 y^4 + 13 x^2 y^2 + 18 xy^2 - 4 x^3 - 27 y^2 = 0$, as seen in Figure~\ref{fig:polar}.
\begin{figure}[h]
\begin{center}
\includegraphics[width=4cm]{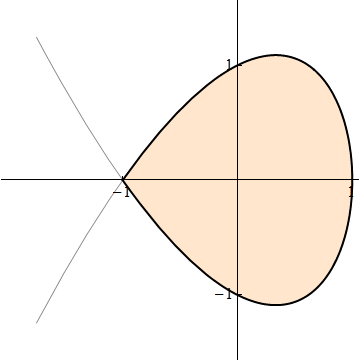} \hspace{2cm} \includegraphics[width=4cm]{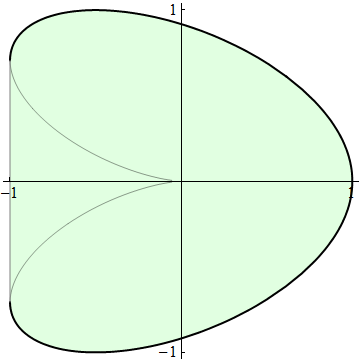}
\end{center}
	\caption{Convex set $C$ (left) and its polar $C^{\circ}$ (right) with extreme points highlighted.} \label{fig:polar}
\end{figure}
Not every boundary point of $C^\circ$ is extreme. The extreme boundary points are parameterized by
$$\left(\frac{2-3u^2}{2-u^2+u^4},\frac{2u}{2-u^2+u^4}\right)$$
with $u \in [-\sqrt{2},\sqrt{2}]$. The slack operator can be parametrized as $s_C:\left[-\sqrt{2},\sqrt{2}\right]^2 \rightarrow \RR_+$ where
$$s_C(u,v)=1-\frac{(2-3u^2)(1-v^2)+2u(2v-v^3)}{2-u^2+u^4}
= \frac{(2-v^2)(v-u)^2 + (u^2-v^2)^2}{u^4 + (2-u^2)}.$$
It is obvious from the last expression that $s_C$ is nonnegative on its domain.
\qed
\end{example}

\subsection{Cone factorizations of slack operators}

We define these general slack operators in the hope that there might be a geometric interpretation to
their factorizations. 
However, to go from polyhedral lifts to more general conic lifts, we still need to generalize the idea of
nonnegative factorizations.
Before that, we need to quickly recall the definition of the dual of a cone. Given a cone $K \subseteq \RR^\ell$ we define its dual to be $K^\ast = \{y \in \RR^\ell\;:\;  \langle x, y \rangle \geq 0 \textup{ for all } x \in K \}$. Here we may think of any inner product but we will always consider the usual Euclidean one. When working with cones of symmetric matrices this amounts to $\langle X,Y \rangle = \tr(XY)$, where $\tr$ is the usual matrix trace, also called the trace inner product.

\begin{definition}
Let $K$ be a closed convex cone and $K^*$ its dual. Furthermore, let $C$ be a convex body
and $s_C$ its slack operator. A factorization of $s_C$ through $K$, or $K$-{\em factorization of }$C$, is a pair of maps $A:\ext(C)\rightarrow K$ and $B : \ext(C^{\circ})\rightarrow K^*$ such that $s_C(y,x)=\langle A(x),B(y) \rangle$ for all $x \in \ext(C)$ and $y \in \ext(C^\circ)$.
\end{definition}

This allows us to factorize the slack operator of a non-polyhedral sets $C$ through a cone $K$.

\begin{example} \label{ex:slackoperatorfact}
Recall from Example~\ref{ex:slackoperator} that the slack operator of the set
$$C=\{(x,y) \in \RR^2 \;:\; (1+x)^2(x-1) + y^2 \leq 0, \; x \geq -1\}$$
is given by the function $s_C(u,v)$ with domain $[-\sqrt{2},\sqrt{2}]^2$ and equaling
$$1-\frac{(2-3u^2)(1-v^2)+2u(2v-v^3)}{2-u^2+u^4}=\frac{2 u^2 + u^4 - 4 u v + 2 v^2 - 3 u^2 v^2 + 2 u v^3}{2-u^2+u^4}.$$
Here we are identifying the extreme points of both $C$ and $C^\circ$ by the interval $[-\sqrt{2},\sqrt{2}]$ since we have an explicit parametrization that we can think of as a labeling of the points. We claim that we can factor this operator through the cone $\psdcone{3}$. To see this, define
$$A(v)=\begin{bmatrix} 1 & 0 & 1-v^2 \\ 0 & 2-v^2 & v(2-v^2) \\ 1-v^2 & v(2-v^2) & 1  \end{bmatrix} =
	\begin{bmatrix} 1 \\ 0 \\ 1-v^2 \end{bmatrix}
        \begin{bmatrix} 1 \\ 0 \\ 1-v^2 \end{bmatrix}^\top 
+ (2-v^2)
\begin{bmatrix} 0 \\ 1 \\ v \end{bmatrix}
\begin{bmatrix} 0 \\ 1 \\ v \end{bmatrix}^\top
$$
and
{\small
$$
B(u)= \frac{1}{2-u^2+u^4}\begin{bmatrix}
 \left(u^2-1\right)^2 & -u \left(u^2-1\right) & u^2-1 \\
 -u \left(u^2-1\right) & u^2 & -u \\
 u^2-1 & -u & 1
\end{bmatrix} =
\frac{1}{2-u^2+u^4}
\begin{bmatrix} u^2-1 \\ -u \\ 1 \end{bmatrix}
\begin{bmatrix} u^2-1 \\ -u \\ 1 \end{bmatrix}^\top.$$}
Since these matrices are positive semidefinite for $u,v \in [-\sqrt{2},\sqrt{2}]$,
and $s_C(u,v)= \langle A(v),B(u)\rangle$, we have the claimed factorization.
\end{example}

As in the Yannakakis setting of polytopes, we can show that the existence of lifts and the existence of factorizations are equivalent even in this  generalized setting of arbitrary convex bodies and closed convex cones. Recall that a convex set $C$ has a $K$-lift, into a closed convex cone $K$, if we can write it as $C= \pi(K \cap L)$, where $\pi$ is a linear map and $L$ an affine subspace.
We will additionally say that $C$ has a \emph{proper} $K$-lift if $L$ intersects the relative interior of $K$. Note that if $C$ has a $K$-lift it always has a proper lift to some face of $K$, so the assumption of properness is not a very strong one.

\begin{theorem}{\cite{GPTlifts}} \label{thm:GPTliftstheorem}
If a convex body $C$ has a proper $K$-lift then its slack operator, $s_C$, has a $K$-factorization. Reciprocally, if $s_C$ has a $K$-factorization then $C$ has a $K$-lift.
\end{theorem}
\begin{proof}
The proof is just a reworking of the proof of Theorem~\ref{thm:yannakakis}.

Let $A$ and $B$ form a
$K$-factorization of the slack operator $s_C$. Then one can check that
$$C=\left\{ x \in \RR^n \;:\; \exists X \in K\;\textup{s.t.}\; \langle X, B(y) \rangle = s_C(y,x),\; \textrm{for all}\; y \in \ext(C^{\circ}) \right\},$$
since the equalities force $s_C(y,x) \geq 0$ for all $y\in \ext(C^{\circ})$ guaranteeing that the right hand side is contained in $C$, while plugging in $X=A(x)$ shows that $\ext(C)$, hence $C$, is contained in the right hand side. Although there are
potentially infinitely many linear equalities, they necessarily cut out a finite dimensional space, and we can use them to eliminate the variable $x$ and explicitly get $C$ as a projection of a slice of $K$.

Suppose now $C$ has a $K$-lift, i.e., $C=\pi(K \cap L)$, for some linear map $\pi$ and affine space $L$. For $x$ in $\ext(C)$, define
$A(x)$ to be any element in the preimage $\pi^{-1}(x) \cap (K \cap L)$. For $y \in \ext(C^{\circ})$, one can pull back the inequality $1-\langle x,y\rangle$ by $\pi$ to a valid inequality on $K \cap L$. The hypothesis that $L$ intersects the interior of $K$ (Slater's condition) can be shown to guarantee that any valid inequality on $K \cap L$, when restricted to $L$ equals $\langle Y,-\rangle$ for some $Y \in K^*$. Pick such a $Y$ as $B(y)$. Then one can check that $\langle A(x), B(y) \rangle = s_C(y,x)$ as intended.
\end{proof}

The cones $K$ that we will mostly be interested in are nonnegative orthants and positive semidefinite cones. In those cases,  their special structure allows us to omit the properness condition from the statement of Theorem~\ref{thm:GPTliftstheorem} (see \cite[Corollary 1]{GPTlifts}).

\begin{example}
\label{eg:cardiod-lift}
Let us consider again the convex set $C$ given in Example~\ref{ex:slackoperatorfact}. Since the slack operator is factorizable through $\psdcone{3}$, Theorem~\ref{thm:GPTliftstheorem} implies that $C$ must have a $\psdcone{3}$-lift. In fact
$$C=\left\{ (x,y) \;:\;  \begin{bmatrix} 1 & 0 & x \\ 0 & 1+x & y \\ x & y & 1  \end{bmatrix} \succeq 0 \right\}.$$
Indeed, check that the nonnegativity of the principal minors of the above matrix gives the polynomial inequalities defining $C$.
Moreover, the symmetry of the existence of factorizations and the self-duality of the psd cone implies that the polar $C^\circ$ must also have a  $\psdcone{3}$-lift.
One can show that
	$$C^\circ=\left\{ (z,w) \;:\; \exists z_1,z_2,z_3 \in \RR\;\;\textup{s.t.}\;\; \begin{bmatrix} 1-z_3-z_1 & z_2 & -(z+z_3)/2 \\ z_2 & z_3 & -w/2 \\ -(z+z_3)/2 & -w/2 & z_1  \end{bmatrix}\succeq 0 \right\}.$$
\end{example}

An important basic property of slack operators is that for any convex body $C$ we have $s_C(y,x)=s_{C^{\circ}}(x,y)$, in other words the slack operator of $C^\circ$ is the transpose of the slack operator of $C$. This follows directly from the definition, since $(C^\circ)^\circ=C$ for convex bodies. Furthermore it also follows directly from the definition that $s_C$ has a $K$-factorization if and only if its transpose has a $K^*$-factorization.
An immediate consequence of Theorem~\ref{thm:GPTliftstheorem} is that $C$ has a $K$-lift if and only if $C^\circ$ has a $K^*$-lift.

Moreover, note that if a convex body $C$ can be written as a linear image of some convex set $Q$, i.e., $C = \pi(Q)$, then $C^\circ \cong Q^\circ \cap \ker(\pi)^{\perp}$, which means that we can write $C^\circ$ as a linear
slice of the polar of its lift. In the case of polytopes, this has a nice interpretation. We saw that if $P$ has a $\RR_+^m$-lift then $P^\circ$ also has a $\RR_+^m$-lift. This means that if $P$ is the image of a
polytope with $m$ facets, then $P^\circ$ is a slice of a polytope with $m$ vertices.

\begin{example}
We have seen in, Figure \ref{fig:oct}, that the regular octagon is the projection of a polytope with $6$ facets. Since the regular octagon is self polar, we then get that it must also be a section of a polytope with $6$ vertices. Figure~\ref{fig:oct2} shows the regular octagon as a slice of an octahedron.

\begin{figure}[h]
\begin{center}
\includegraphics[width=0.3\linewidth]{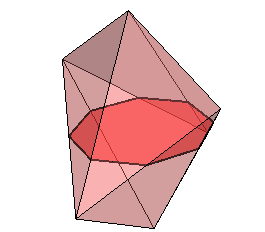}
\end{center}
\caption{The regular octagon as a minimal linear slice of an octahedron.} \label{fig:oct2}
\end{figure}
\end{example}

One should also point out that there is a simple extension of Theorem \ref{thm:GPTliftstheorem} so that it covers approximations of sets, instead of only exact representations. To state it we need to introduce the \emph{generalized slack operator}. Given two convex bodies $C \subseteq D$ we define the
generalized slack operator of the pair $(C,D)$ to be $s_{C,D}: \ext(D^{\circ}) \times \ext(C) \rightarrow \RR_+$ given by $s_{C,D}(y,x)=1-\langle x,y \rangle$. A look at the proof of Theorem \ref{thm:GPTliftstheorem} shows that it immediately generalizes to the following result.

\begin{proposition}
\label{prop:nestedconvexsets}
Let $C$ and $D$ be a pair of convex bodies with $C \subseteq D$. If there is a convex set
	$B$ such that $B$ has a proper $K$-lift and $C \subseteq B \subseteq D$ then
 $s_{C,D}$ has a $K$-factorization. Reciprocally, if $s_{C,D}$ has a $K$-factorization
	then there exists a convex set $B$ such that $B$ has a $K$-lift and $C \subseteq B \subseteq D$.
\end{proposition}

By setting $D=(1+\varepsilon)C$ we obtain a powerful tool to study approximability of a given set by sets with small lifts.


\subsection{Bregman divergence as a slack operator}
An important notion that is closely related to that of slack operator is the \emph{Bregman divergence} of a convex function, which was briefly mentioned in Section \ref{sec:epi} and is widely used in applications.
The Bregman divergence can, under some mild assumptions, be seen as the slack operator of the epigraph of the
corresponding convex function, while slack operators can be seen as special cases of a nonsmooth version of the Bregman divergence, tying the two notions together.

Following \cite{bauschke1997legendre}, let $f:\RR^n \rightarrow \overline{\RR}$ be a closed convex proper function
whose domain has nonempty interior. If $f$ is differentiable in the interior of its domain we
define its Bregman divergence as the function $D_f:\RR^n \times \intt (\dom (f)) \rightarrow [0,+\infty)$
given by
$$D_f(x,y)=f(x)-f(y)-\langle \nabla f(y), x-y \rangle.$$
In other words, $D_f(x,y)$ is the difference between $f(x)$ and the value at $x$ of the linear approximation to $f$ around $y$.
The convexity of $f$ implies that $D_f(x,y) \geq 0$.

\begin{example}
If we take $f(x)=x^\top Qx$ to be a convex quadratic function, then
$$D_f(x,y)=x^\top Qx-y^\top Qy - 2y^\top Q(x-y) = (x-y)^\top Q(x-y)$$
since $Q$ is a symmetric matrix.
In particular, if $f(x)=\|x\|^2$ we have $D_f(x,y)=\|x-y\|^2$.
\end{example}

It turns out that the Bregman divergence is a slack operator of sorts for convex functions, where Fenchel duality replaces polarity of convex sets.
Recall that the Fenchel conjugate of  $f$ is the function $f^*:\RR^n \rightarrow \overline{\RR}$ defined by
$$f^*(y^*)=\sup \{ \langle x,y^*\rangle -f(x) \ | \ x \in \RR^n \}.$$
The function $f^*$ is closed and convex and {\em Fenchel's inequality} says that $f(x) + f^*(y^*) \geq \langle x,y^* \rangle$.
We will refer to $F(x,y^*) := f(x) + f^*(y^*) -  \langle x,y^* \rangle$ as the {\em slack} in the Fenchel inequality.
Using simple properties of Fenchel conjugation one gets that $D_f(x,y)= F(x,\nabla f(y))$. As such, the Bregman divergence is essentially a slack, but of duality of convex functions rather than convex sets.

In fact, if $f$ is well-behaved, namely if $f$ is a {\em convex function of Legendre type}, (smooth and strictly convex in its domain, and the limits of $\|\nabla f(x)\|$ as $x$ approaches the boundary of the domain through its interior is $+\infty$), then the map $y\rightarrow y^* :=\nabla f(y)$ is an isomorphism from the interior of the domain of $f$ to the interior of the domain of $f^*$. In that case, the Bregman divergence and the slack of the Fenchel inequality are just reparametrizations of each other. For more details, see \cite[Section 26]{rockafellar}.

\begin{example}
The function $f(x)=e^x$ is a strictly convex smooth function with domain $\RR$. Its conjugate $f^*$ has domain in the nonnegative reals, and is given by $f^*(0)=0$ and $f^*(z)= z \ln(z) -z$ for $z>0$.
The Fenchel inequality becomes $F(x,z)=e^x+ z \ln(z) -z -xz \geq 0$. The Bregman divergence is
$$D_f(x,y)=F(x,f'(y))=F(x,e^y)=e^x-e^y+(y-x)e^y.$$
As above, we can think of the function $F$ as the map $x_0+y_0-\langle x, y \rangle$ with domain the product of the
graph of $f$ and the graph of $f^*$. When extended to the epigraphs, this operator creates a one-to-one
correspondence between points in the epigraph of a function and affine under-estimators of the other function,
as illustrated in Figure~\ref{fig:epi}.

\begin{figure}[ht]
\begin{center}
\includegraphics[width=0.45\columnwidth]{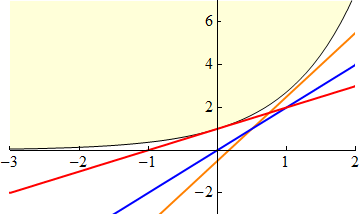}
\hspace{1cm}
\includegraphics[width=0.45\columnwidth]{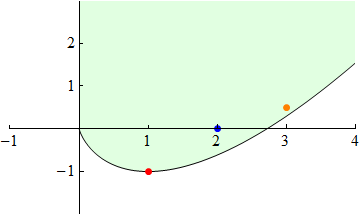}
\end{center}
\caption{The epigraphs of $f(x)=e^x$ and $f^*(z)= z \ln(z) -z$, with some linear underestimators of $f$ and corresponding points in the epigraph of $f^*$ marked.} \label{fig:epi}
\end{figure}

For instance, the point $(2,0)$ in $\epi(f^*)$ gives rise to the inequality $x_0-2x \geq 0$,
while the point $(1,-1)$, that is an extreme point of $\epi(f^\ast)$ gives rise to the inequality
$x_0-1-x\geq 0$, and so on.
\qed
\end{example}

This illustrates that for a smooth convex function of Legendre type, its Bregman divergence is, in a very strong sense, the same as the slack operator of its epigraph. The only difference is that  $\epi(f)$ and $\epi(f^*) $ are not actually polar to each other, having instead a different duality relation.
\medskip

Interestingly, not only can we think of the Bregman divergence as a slack operator, we can conversely think of slack operators as straightforward generalizations of the Bregman divergence.
For any closed convex set $C$, consider its \emph{indicator function} $\iota_C(x)$ given by $\iota_C(x)=0$ if $x\in C$ and $+\infty$ otherwise. This is a proper closed convex function, and its Fenchel conjugate is the \emph{support function} of $C$, given by $h_C(y)=\sup_{x \in C} \langle x,y\rangle$. Since $\iota_C(x)$ is not a smooth convex function, its usual Bregman divergence is not defined. However, there are many generalizations of the Bregman divergence to the nonsmooth case. A simple way of doing this is to think of Bregman divergence as
the affine operator $x_0+y_0-\langle x, y \rangle$ applied to $\epi(\iota_C) \times \epi(h_C)$.
We still have a correspondence between points in the epigraph of $h_C$ and valid inequalities on the epigraph of $\iota_C$.
Restricting to the extreme points of $\epi(\iota_C)$, which correspond to the extreme points of $C$, and a representative of every extreme ray of the cone $\epi(h_C)$, recovers the slack operator of $C$ (up to scaling). An advantage of this viewpoint is that we can remove the technical condition of requiring $0$ to be in the interior of $C$.

\section{Constructions of Lifts}
\label{sec:constructions}

Is there a \emph{systematic} way to construct lifts for convex bodies? In this section we will try to answer this question, informed by the factorization theorem from the previous section.

\newcommand{\XX}{\mathcal X}

\subsection{Spectrahedral lifts and sums of squares}

 We start by restating the factorization theorem (Theorem \ref{thm:GPTliftstheorem}) in the special case of spectrahedral lifts, which will be our main focus in this section.
\begin{theorem}
\label{thm:GPTliftstheorem22}
Let $C = \conv(\XX)$ be a convex body. Then $C$ has a spectrahedral lift of size $m$ if, and only if, there exists a map $A:\XX \rightarrow \S^m_+$ such that the following holds: for any $\ell \in C^{\circ}$,
there exists a $B \in \S^m_+$ such that
\begin{equation}
\label{eq:1-ellAB}
1 - \ell(x) = \tr ( A(x) B ) \quad \forall x \in \XX.
\end{equation}
\end{theorem}
\begin{remark}
Throughout this section we use the notation $\ell$ for elements of $C^{\circ}$, to emphasize that these are linear forms defined on the ambient space where $C$ lives. We also write $\ell(x)$ (as opposed to the bracket notation $\langle \ell , x \rangle$) for the evaluation of $\ell$ at some element $x$.
\end{remark}
Recall that the slack operator of $C$ is indexed precisely by pairs $(\ell,x) \in \ext(C^{\circ}) \times \ext(C)$ and the associated entry in the matrix is $1-\ell(x) \geq 0$.  Equation \eqref{eq:1-ellAB} is nothing but a $\S^m_+$ factorization of this slack operator.

It is useful to think of Equation \eqref{eq:1-ellAB} as a \emph{certificate of nonnegativity} of the affine function $1-\ell(x)$ on the set $\XX$. Indeed the right-hand side of \eqref{eq:1-ellAB} is ``obviously'' nonnegative as it is the inner product of two positive semidefinite matrices. That \eqref{eq:1-ellAB} is a certificate of nonnegativity becomes clearer if we factorize $A(x)$ as $A(x) = F(x)F(x)^\top$ and $B = DD^\top$. (Such a factorization is possible since $A(x),B \psd 0$.) Then we have, for any $x \in \XX$
\begin{equation}
\label{eq:1-ellSOS}
1 - \ell(x) = \tr ( A(x) B ) = \tr \left ( (F(x)^\top D)^\top F(x)^\top D \right) = \sum_{1\leq i,j \leq m} M_{ij}(x)^2
\end{equation}
where $M(x) = F(x)^\top D$. Equation \eqref{eq:1-ellSOS} shows that $1-\ell(x)$ is a \emph{sum of squares} of the entries of $M(x)$. In turn, the entries of $M(x)$ are linear combinations of the entries of $F(x)$. We have thus shown the following: if $C=\conv(\XX)$ has a spectrahedral lift of size $m$, then there is a \emph{subspace} $V$ of functions on $\XX$ with $\dim(V) \leq m^2$ such that the following is true:
\begin{equation}
\label{eq:1-ellSOSV}
\tag{*-SDP}
\begin{array}{cc}
\text{For any valid linear inequality $\ell(x) \leq 1$ on $\XX$,}\\
\text{there exist functions $h_k \in V$ s.t. $1 - \ell(x) = \sum_{k} h_k(x)^2$ for all $x \in \XX$.}
\end{array}
\end{equation}
The subspace $V$ here is precisely $V = \linspan \left\{ F_{ij} : 1 \leq i,j \leq m \right\}$, where we think of each $F_{ij}$ as a function on $\XX$. (We will look at the nature of these functions later in this section.)

One can also prove a converse of the previous statement. This is the object of the next proposition, showing how to construct a lift from a subspace satisfying \eqref{eq:1-ellSOSV}. Note that the statement below is not an \emph{exact} converse, as it gives a spectrahedral lift of size $\dim(V)$ instead of $\sqrt{\dim(V)}$. The reason has to do with the rank of the map $A(x)$ which is rank-one in the construction \eqref{eq:Axrankone} below, whereas it can be of higher rank in Theorem \ref{thm:GPTliftstheorem22}.

\begin{proposition}
\label{prop:soslift}
Let $C = \conv(\XX)$ be a convex body. Assume there exists a subspace $V$ of functions on $\XX$ such that \eqref{eq:1-ellSOSV} is true. Then $C$ has a spectrahedral lift of size $\dim(V)$.
\end{proposition}
\begin{proof}
We construct a positive semidefinite factorization of the slack operator of $C$. Let $m = \dim(V)$, $f_1,\ldots,f_m$ be a basis of $V$, and $\ff(x) = [f_1(x),\ldots,f_m(x)]^\top$. Define 
\begin{equation}
\label{eq:Axrankone}
A(x) = \ff(x)\ff(x)^\top \in \S^m_+.
\end{equation}

Let $\ell \in C^{\circ}$ so that $1-\ell(x) \geq 0$ for all $x \in \XX$. We want to show that there exists $B \in \S^m_+$ such that $1- \ell(x) = \tr(A(x) B)$ for all $x \in \XX$. We know by \eqref{eq:1-ellSOSV} that there exist functions $h_k \in V$ such that $1-\ell(x) = \sum_{k} h_k(x)^2$ for $x \in \XX$. Since $f_1,\ldots,f_m$ form a basis of $V$, we know there exist vectors $b_k \in \RR^m$ such that $h_k(x) = b_k^\top \ff(x)$. Then
\[
1 - \ell(x) = \sum_{k} h_k(x)^2 = \sum_{k} (b_k^\top \ff(x))^2 = \tr\left(\ff(x) \ff(x)^\top B\right)
\]
where $B = \sum_{k} b_k b_k^\top \in \S^m_+$. This completes the proof.
\end{proof}

Proposition \ref{prop:soslift} reduces the problem of constructing a spectrahedral lift of $C$ to that of finding a subspace $V$ of functions such that \eqref{eq:1-ellSOSV} holds. We now list some examples of lifts constructed in this way.

\begin{example}[Elliptope]
Consider the convex body $C = [-1,1]^2$ with extreme points $\XX = \{ -1, +1 \}^2$. Define the subspace $V = \linspan(1,x,y)$ of affine functions, and consider the valid linear inequality $x \leq 1$ on $C$. Note that we can write
\[
1 - x = \frac{1}{2} (1-x)^2 \quad \forall(x,y) \in \XX
\]
where we used the fact that $x^2 = 1$. Similarly one can check that the other facet inequalities of $C$ (namely $1 + x \geq 0$ and $1\pm y \geq 0$) are sums of squares from functions in $V$. It thus follows that $C$ has a spectrahedral lift of size $\dim(V) = 3$. Indeed one can verify that
\begin{equation}
\label{eq:square_lift_elliptope}
C = \left\{ (x,y) \in \RR^2 : \exists z \in \RR, \; \begin{bmatrix} 1 & x & y\\ x & 1 & z\\ y & z & 1\end{bmatrix} \psd 0 \right\}.
\end{equation}
This is precisely the example of the three-dimensional \emph{elliptope} seen in the introduction (see Figure \ref{fig:elliptope}(left)) which projects onto the two-dimensional square.
\end{example}
%

\begin{example}[Convex hull of rational curves]
Consider a plane curve $\XX \subset \RR^2$ having a rational parametrization,  $\XX = \left\{(x(t),y(t)) : t \in \RR\right\}$ where $x(t),y(t)$ are rational functions, i.e., $x(t)=a(t)/q(t)$, $y(t) = b(t)/q(t)$ where we assume that $q(t) > 0$ for all $t \in \RR$ and $\max(\deg a, \deg b) \leq \deg q$.
We can show that $\cl\conv(\XX)$ (where $\cl$ denotes the closure) has a spectrahedral lift of size at most $\deg(q)/2+1$ \cite{parrilo2006exact,henrionrational}; we will show this by exhibiting a subspace of functions $V$ such that property \eqref{eq:1-ellSOSV} holds for $\XX$.

We first define a function $t:\XX\rightarrow \RR$ that gives the ``time'' coordinate of any point in $\XX$ in the parametrization $(x(t),y(t))$; namely for any $(x_0,y_0) \in \XX$ let $t(x_0,y_0)=t_0 \in \RR$ be such that $(x(t_0),y(t_0)) = (x_0,y_0)$ (if multiple such $t_0$ exist, an arbitrary choice is made so that $t$ is single-valued). Consider now the subspace $V$ of bivariate functions of the form
\begin{equation}
\label{eq:funcxyrational}
h(x,y) = \sum_{i=0}^{\deg(q)/2} c_i \frac{t(x,y)^i}{\sqrt{q(t(x,y))}}
\end{equation}
where the $c_i$ are arbitrary coefficients. Note that $\dim(V) = 1+\deg(q)/2$.

We claim that any valid linear inequality on $\cl \conv(\XX)$ is a sum of squares from $V$. Indeed, assume we have a linear form $\ell$ such that $1 - \ell(x)$ is nonnegative on $\XX$, i.e., $1-\ell(x(t),y(t)) \geq 0$ for all $t \in \RR$. Using the fact that $x(t) = a(t)/q(t)$ and $y(t) = b(t)/q(t)$ this gives:
\[
q(t) - \ell(a(t),b(t)) \geq 0 \quad \forall t \in \RR.
\]
Note that $q(t) - \ell(a(t),b(t))$ is a (univariate) polynomial of degree at most $\deg q$. It is a well known simple fact that univariate polynomials are nonnegative in the real line if and only if they are sums of squares (see Chapter 3 of \cite{BPTSIAMBook} for an overview on sums of squares), therefore $q(t)-\ell(a(t),b(t))$ must be a sum of squares of some polynomials $h_k(t) \in \RR[t]$ of degree at most $(\deg q)/2$, i.e.,
\[
q(t) - \ell(a(t),b(t)) = \sum_{k} h_k(t)^2.
\]
Evaluating at $t=t(x,y)$ and dividing by $q$ we get
\[
1 - \ell(x,y) = \sum_{k} \left(\frac{h_k(t(x,y))}{\sqrt{q(t(x,y))}}\right)^2 \qquad \forall (x,y) \in \XX.
\]
Since the functions $(x,y) \mapsto h_k(x,y) / \sqrt{q(t(x,y))}$ belong to $V$ we have shown what we wanted.
\end{example}

\subsection{Hierarchies} Consider the case where $C$ is the convex hull of a set defined using polynomial equations.
A natural choice of subspace $V$ in this setting is the space $V=V_k$ of \emph{polynomials} of degree at most $k$, where $k$ is a certain fixed integer.
If property \eqref{eq:1-ellSOSV} holds with such $V_k$ then we get a spectrahedral lift of $\conv(\XX)$ of size $\dim(V_k)$. This spectrahedral lift is precisely the lift produced by the \emph{Lasserre/theta-body method}, see \cite{lasserre2009convex,thetabodies}.

We note that it is possible for a convex body $C$ to have a spectrahedral lift, and yet not verify property \eqref{eq:1-ellSOSV} for any  subspace of polynomials $V$ (of arbitrary degree). For example, consider the curve defined by the equation $y^2 - x^3 + x^4 = 0$ and shown in Figure \ref{fig:piriform}; this curve is known as the \emph{Piriform curve}. Because of the singularity at the origin, one can show that the Lasserre/theta body hierarchy for this curve is not exact at any finite level $k$, see \cite[Theorem 9]{gouveiathomaschapter}. However the convex hull of this curve has a SDP representation since the curve has a rational parametrization, namely it can be shown that it is parametrized by
\[
\begin{cases}
x(\theta) = \frac{1+\sin(\theta)}{2}\\
y(\theta) = \frac{\cos(\theta)(1+\sin(\theta))}{4},
\end{cases}
\]
which can be turned into a rational parametrization using $\cos(\theta) = \frac{1-t^2}{1+t^2}$ and $\sin(\theta) = \frac{2t}{1+t^2}$.

\begin{figure}
\centering
\includegraphics[scale=0.7]{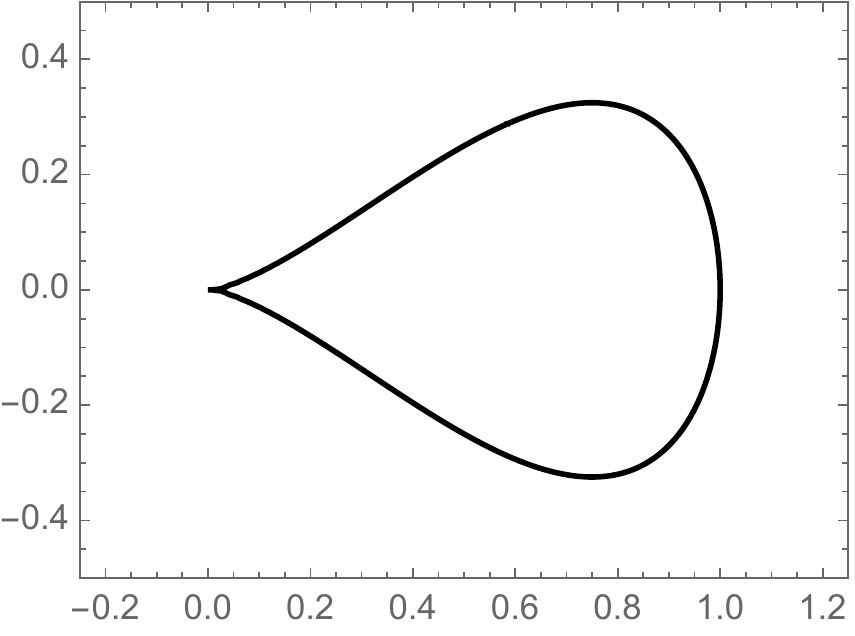}
\caption{The Piriform curve defined by $y^2 - x^3 + x^4 = 0$. This is a rational curve and so the convex hull of this curve has a spectrahedral lift. However property \eqref{eq:1-ellSOSV} is not true for any subspace $V$ consisting of polynomials. (Note that functions of the form \eqref{eq:funcxyrational} are not polynomials.) This example was considered in \cite[Example 4.4]{gouveia2011positive}.}
\label{fig:piriform}
\end{figure}

\begin{example}[$k$-level polytopes]
\label{sec:k-level} A polytope $P$ is called \emph{$k$-level} if every facet-defining linear function of $P$ takes at most $k$ different values on the vertices of the polytope. Said differently, for every affine hyperplane $H$ defining a facet of $P$, all the vertices of $P$ lie in at most $k$ different translates of $H$. See Figure \ref{fig:23level} for examples of 2 and 3-level polytopes in $\RR^3$. One important example of 2-level polytope is the stable set polytope of perfect graphs (see Section \ref{sec:chainTH}, Theorem \ref{thm:THperfect}).

\begin{figure}[ht]
  \centering
  \includegraphics[height=2.5cm]{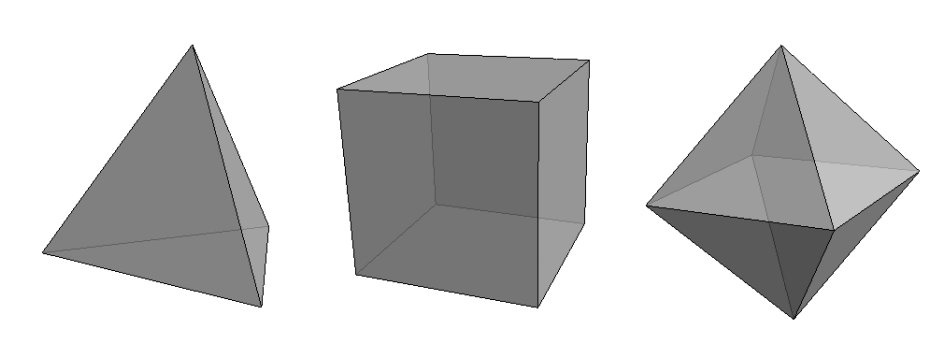}
  \qquad\qquad\qquad
  \includegraphics[height=2.5cm]{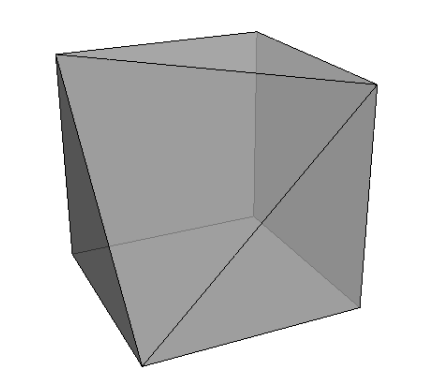}
  \caption{Left: examples of 2-level polytopes. Right: a polytope that is 3-level but not 2-level. Figure from \cite{gouveiathomaschapter}.}
  \label{fig:23level}
\end{figure}

One can prove that if $P$ is a $k$-level polytope, then $P$ admits a spectrahedral lift of size at most $\binom{n+k-1}{k-1} = \dim(V_{k-1})$ where $V_{k-1}$ is the space of polynomials of degree at most $k-1$ in $n$ variables. To do this we will prove that \eqref{eq:1-ellSOSV} holds with $V=V_{k-1}$.

Indeed, let $\XX$ be the set of vertices of the polytope, and assume that $1-\ell(x) \geq 0$ for all $x\in \XX$. We need to show that there exist $h_i \in V$ such that $1-\ell(x) = \sum_{i} h_i(x)^2$ for all $x \in \XX$. Note that by Farkas' Lemma, $1-\ell(x)$ is a nonnegative combination of facet defining inequalities, so it is enough to show that any facet inequality can be written as a sum of squares. We can therefore assume that $1-\ell(x)\geq 0$ defines a facet which,
since $P$ is $k$-level, implies that $\left\{1-\ell(x) : x \in \XX\right\} = \{a_1,a_2,\ldots,a_k\}$ for some nonnegative numbers $a_i \in \RR$. Consider a univariate polynomial $r \in \RR[t]$ of degree $k-1$ such that $r(a_i) = \sqrt{a_i}$ (e.g., constructed using the Lagrange interpolation formula), and let $h(x) = r(1-\ell(x))$ be the evaluation of $r$ at $1-\ell(x)$. Since $\deg r \leq k-1$ we have $h \in V_{k-1}$. Note that for any $x \in \XX$ we have $1-\ell(x) = r(1-\ell(x))^2 = h(x)^2$ where we used the fact that $1-\ell(x) \in \{a_1,\ldots,a_k\}$ and that $r(a_i)^2 = a_i$ for all $i=1,\ldots,k$. We have thus shown that this choice of subspace works. An illustration of this construction is shown in Figure \ref{fig:klevel_hexagon} in the case of the regular hexagon.
\end{example}

\begin{figure}[ht]
  \centering
  \includegraphics[scale=.75]{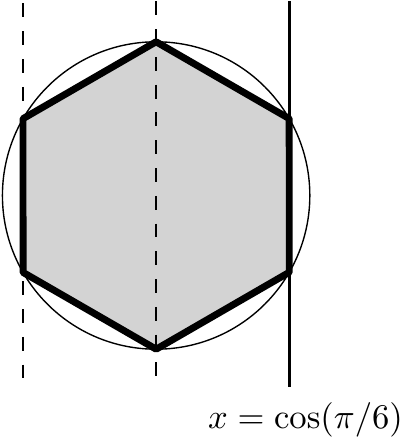}
  \qquad\qquad\qquad
  \includegraphics[width=6cm]{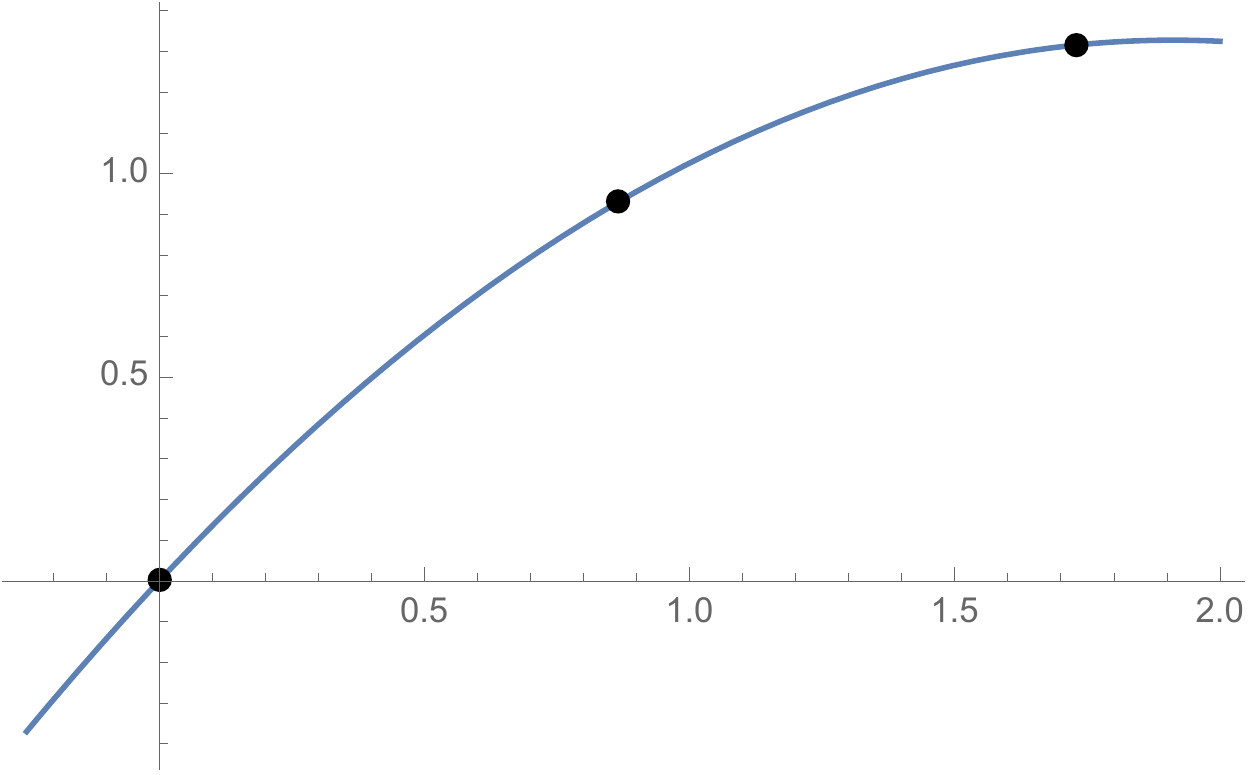}
  \caption{Left: The convex hull of the regular hexagon is a 3-level polytope. The levels of the facet inequality $\cos(\pi/6) - x \geq 0$ are $\{0,a_1,a_2\} = \{0,\cos(\pi/6),2\cos(\pi/6)\}$. Right: the polynomial $r(t)$ of degree 2 such that $r(0) = 0$, $r(a_1) = \sqrt{a_1}$, $r(a_2) = \sqrt{a_2}$. Any facet inequality of the regular hexagon is a sum of squares of quadratic polynomials. One can show that this gives a spectrahedral lift of the regular hexagon of size 5.}
  \label{fig:klevel_hexagon}
\end{figure}

\subsection{Symmetric convex bodies} In many cases, the convex body $C \subset \RR^n$ we are interested in has some symmetries. This means that there are certain linear invertible transformations of $\RR^n$ that leave $C$ invariant. The symmetry group (or \emph{automorphism group}) of $C$, is precisely the set of such transformations:
\[
\Aut(C) = \left\{ T \in GL(\RR^n) : T(C) = C \right\}.
\]
For example the automorphism group of the regular $n$-gon in the plane is the dihedral group and consists of $2n$ elements. When constructing a lift of the convex body $C$, one may require that the lift ``respects'' the symmetry of $C$. Such lifts are called symmetric or \emph{equivariant} lifts and have been considered since the work of Yannakakis \cite{yannakakis1991expressing}. Having an equivariant spectrahedral lift of a convex body $C$ has a very natural interpretation in terms of the subspace $V$ of functions in \eqref{eq:1-ellSOSV}. Finding an equivariant lift of $C = \conv(\XX)$ amounts to finding a subspace $V$ of functions on $\XX$ that is \emph{invariant} under the action of $\Aut(C)$. Here the action of $\Aut(C)$ on functions $\XX\rightarrow \RR$ is the natural one defined by $(g \cdot f)(x) = f(g^{-1} x)$ for $g \in \Aut(C)$ and $x \in \XX$. A question the reader may ask is whether there is benefit in \emph{breaking symmetry} when constructing lifts of convex bodies. More precisely, given a symmetric convex body, can non-equivariant lifts be much smaller than their equivariant counterparts? It turns out that the answer is yes, and the gap can be very large. Kaibel et al. \cite{kaibel-pashkovich-theis} showed the existence of polytopes $(P_n)$ that have a nonequivariant polyhedral lift of polynomial size in $n$, and yet any equivariant lift has size at least $n^{\Omega(\log n)}$. Also it was shown in \cite{equivariantliftssos} that the even parity polytope has no polynomial-size equivariant spectrahedral lift, despite having a linear size non-equivariant one (cf. Section \ref{sec:bdd}).

\subsection{Linear programming lifts} We have so far only discussed the case of spectrahedral lifts. In this concluding paragraph, we briefly discuss polyhedral lifts. It can be shown via the slack matrix factorization theorem, that, for a finite set $\XX$, a polytope $P = \conv(\XX)$ has a polyhedral lift of size $m$ if, and only if, there are $m$ nonnegative functions $a_1,\ldots,a_m:\XX \rightarrow \RR_+$ such that the following is true:
\begin{equation}
\tag{*-LP}
\label{eq:1-ellLP}
\begin{array}{cc}
\text{For any valid linear inequality $\ell(x) \leq 1$ on $\XX$,}\\
\text{there exist $b_1,\ldots,b_m \geq 0$ s.t. $1 - \ell(x) = \sum_{i=1}^m b_i a_i(x)$ for all $x \in \XX$.}
\end{array}
\end{equation}
As such, whereas the task of constructing a spectrahedral lift corresponds to finding a subspace $V$ such that \eqref{eq:1-ellSOSV} holds, constructing an LP lift requires finding a finite set of functions $a_1,\ldots,a_m:\XX \rightarrow \RR_+$ such that \eqref{eq:1-ellLP} is true. There are some settings where a natural set of functions can be considered. Assume that $\XX \subset \{0,1\}^n$ is a subset of the Boolean hypercube and is described using polynomial inequalities, i.e.,
\begin{equation}
\label{eq:Xbinary}
\XX = \left\{ x \in \{0,1\}^n : g_1(x) \geq 0, \ldots, g_k(x) \geq 0\right\},
\end{equation}
where $g_1,\ldots,g_k$ are polynomials in $x$. Such a setting covers many combinatorial optimization problems. In this case one can consider the following family of functions
\begin{equation}
\label{eq:aTIg}
a_{T,I,\gamma}(x) = \prod_{i \in T} x_i \prod_{i \in I \setminus T} (1-x_i) \prod_{i=1}^k g_i(x)^{\gamma_i},
\end{equation}
where $T \subseteq I \subset \{1,\ldots,n\}$ and $(\gamma_1,\ldots,\gamma_k) \in \NN^k$. The functions \eqref{eq:aTIg} are clearly nonnegative on $\XX$. Many of the early \emph{lift-and-project methods} in combinatorial optimization (Sherali-Adams \cite{sheraliadams} or Lov\'asz-Schrijver \cite{lovaszschrijver}) can be understood in this framework and correspond to choosing functions $a_i$ of the form \eqref{eq:aTIg}.

\section{Obstructions and lower bounds}
\label{sec:obstructions and lower bounds}
%
%
%
Given a convex body $C$ and a closed convex cone $K$, how can we show that $C$
does not have a $K$-lift?  In other words, what properties of $C$ are
obstructions to $C$ having a $K$-lift?  When $K$ comes from
a cone family, such as $\RR_+^m$ or $\mathcal{S}_+^m$, we typically phrase
things in more quantatitive language, aiming to find \emph{lower bounds on the
size} of the corresponding lift. 

For polytopes, the study of lower bounds on the size of polyhedral and
spectrahedral lifts has received considerable attention in
recent years. (See Section~\ref{sec:discussion} for more information and
references.) In this section, we focus on cases where $C$ is not a polytope. In this
more general setting, even deciding whether $C$ has a $K$-lift for some $K$ in a
given cone family (such as positive semidefinite cones of any finite size) 
can be very challenging. 

The characterization of the existence of lifts in terms of factorizations of the slack
operator, discussed in Section~\ref{sec:slackoperator}, 
is a key tool to finding obstructions. This connection
allows us to conclude the non-existence of $K$-lifts by showing that it is impossible to 
find a $K$-factorization of the slack operator. Crucially, we can often find such
obstructions by considering only very coarse features of the slack operator. For instance, 
we may consider only the pattern of zeros and non-zeros of the slack operator, 
capturing only which extreme points lie on which supporting hyperplanes. 

As a concrete example, consider the following result of Goemans about the size of polyhedral 
lifts of polytopes. 
We use this as a starting point when discussing obstructions to lifts of convex bodies that are not polytopes
throughout this section. 
		\begin{proposition}[{\cite{goemans2015smallest}}]
\label{prop:goemans}
Assume $C$ is a polytope with $v$ extreme points. Then any polyhedral lift of $C$ has size at least $\lceil \log_2 v\rceil$.
\end{proposition}
The bound of Goemans can be tight (up to a constant multiplicative factor). One family of examples
showing this are the regular $N$-gons in the plane. These have $N$ extreme points and are know to have
polyhedral lifts of size either $2 \lceil \log_2 N\rceil$ or $2\lceil \log_2 N\rceil-1$ 
depending on the value of $N$~\cite{vandaele2017linear}.  These polyhedral lifts are 
improved versions of a lift of the regular $2^n$-gon of size $2n+4$ due to 
Ben-Tal and Nemirovski~\cite{ben2001polyhedral}.

Another family of examples showing that Goemans' bound is tight, comes from the
permutahedra $\Pi_n$. Recall from Example~\ref{ex:permutahedron}, that for each
$n$, $\Pi_n$ is the convex hull of all permutations of $(1,2,\ldots,n)$. As
such $\Pi_n$ has $n{!}$ extreme points. Goemans' lower bound therefore tells us
that any polyhedral lift of $\Pi_n$ has size at least $n\log_2(n/e)$. By
comparison, the Birkhoff polytope lift of the permutahedron from
Example~\ref{ex:permutahedron} has size $O(n^2)$. Goemans constructed a polyhedral 
lift of the permutahedron of size $O(n\log(n))$, and used Proposition~\ref{prop:goemans}
to show that this lift cannot be substantially reduced in size~\cite{goemans2015smallest}. 

Goemans' lower bound also implies the intuitively obvious fact 
that a convex body with infinitely many extreme points
cannot have a polyhedral lift. In other words, having infinitely many 
extreme points is an obstruction to the existence of a polyhedral lift.
To see why, we can consider the contrapositive of Proposition~\ref{prop:goemans}. This tells us that  
if a convex body has a polyhedral lift with at most $m$ facets, then it has at most $2^m$ extreme points. 

In the rest of this section, we consider generalizations of
Proposition~\ref{prop:goemans} in two directions.  In
Section~\ref{sec:obs-facial}, we view Goemans' bound as being about obstructions to lifts arising
from the complexity of the facial structure (in this case, the number of vertices) of
$C$. In particular, we
discuss other obstructions based on facial structure that are applicable even
when $C$ is not a polytope.  In terms of the slack operator, focusing on facial
structure essentially corresponds to considering only its pattern of zeros and
non-zeros.  In Section~\ref{sec:obs-alg}, we view Goemans' bound
as being about obstructions to lifts arising from the algebraic
complexity of the boundary of $C^\circ$. In the case of Goemans' bound the
number of vertices of $C$ is the number of facets of $C^\circ$, which is also
the smallest degree of a polynomial vanishing on the boundary of $C^\circ$.  In
particular, in Section~\ref{sec:obs-alg} we discuss other obstructions of an
algebraic nature.

\subsection{Obstructions based on facial structure}
\label{sec:obs-facial}
The faces (see Section~\ref{sec:slackoperator} for the definition) of a convex set are partially
ordered by inclusion. Throughout this section we let $\mathcal{F}_C$ denote the
poset of faces of $C$. 
If $C = \pi(Q)$ is the projection of a convex set $Q$, then the preimage under $\pi$ 
of any face of $C$ is a face of $Q$. 
This observation can be used to find lower bounds on the size of lifts.
\begin{example}
	\label{eg:oct-lb}
	Consider again the regular octagon for which we see two polyhedral lifts with six facets 
	in Figure~\ref{fig:oct}. 
	In each case, the preimage (under the projection) of each vertex of the octagon is a face
of the lift. Moreover, a vertex of the lift can belong to at most one such preimage face. 
This implies that any polyhedral lift of an octagon must have at least $8$ vertices. 
Since any polytope with $5$ or fewer facets has at most six vertices, it follows that 
any polyhedral lift of an octagon has size at least $6$.
\end{example}
In general, if $Q$ is a lift of $C$ then there is an \emph{order embedding}
$\phi: \mathcal{F}_C\rightarrow \mathcal{F}_Q$, 
i.e., a map that satisfies $F\subseteq F'$ if and only if $\phi(F) \subseteq \phi(F')$ for any $F,F'\in \mathcal{F}_C$.
Furthermore, if $Q = K \cap L$ for some closed convex cone $K$ and affine space $L$, then there is a 
natural order embedding $\psi:\mathcal{F}_Q\rightarrow \mathcal{F}_K$ which sends a face $F$ of $Q$ to the minimal face of
$K$ containing $F$. Overall, then, the composition $\psi\circ \phi:\mathcal{F}_C\rightarrow \mathcal{F}_K$ 
gives an order embedding from the face poset of $C$ to the face poset of $K$. 
For the lift on the right of Figure~\ref{fig:oct}, the embeddings $\phi$ and $\psi$, and their composition, are all
illustrated in Figure~\ref{fig:poset-embed}.
\begin{figure}
	\begin{center}
		\includegraphics[scale=0.5]{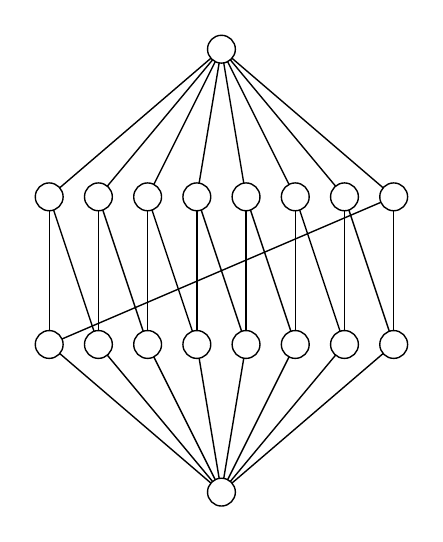}
		\includegraphics[scale=0.5]{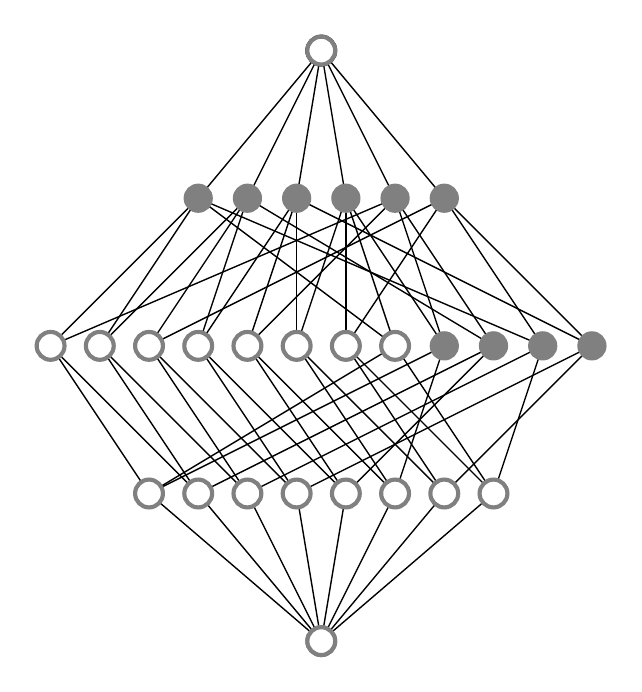}
		\includegraphics[scale=0.5]{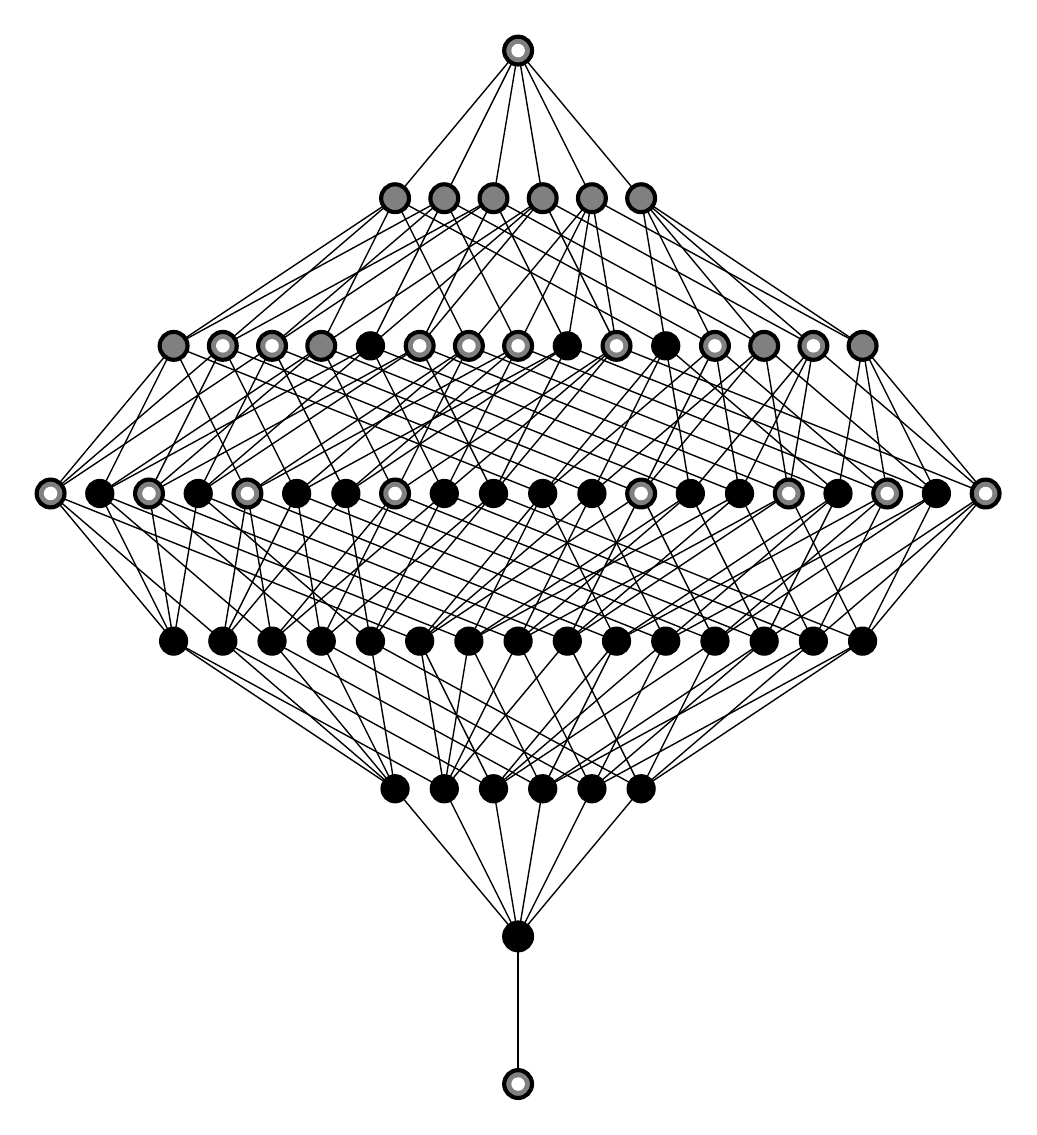}
	\end{center}
	\caption{\label{fig:poset-embed} Poset embeddings implied by the lift $P = \pi(Q) = \pi(\RR_+^6\cap L)$ 
	of the regular octagon $P$ shown on the right in Figure~\ref{fig:oct}. Left: Hasse diagram of $\mathcal{F}_P$.
	Center: Hasse diagram of $\mathcal{F}_Q$ with the white nodes showing the embedding of $\mathcal{F}_P$.
	Right: Hasse diagram of $\mathcal{F}_{\RR_+^6}$ with grey nodes showing the embedding of $\mathcal{F}_Q$ and 
	white nodes showing the embedding of $\mathcal{F}_{P}$.}
\end{figure}

Since lifts induce embeddings of face posets, we can find obstructions to the existence of $K$-lifts 
by finding obstructions to embedding the face poset of $C$ in the face poset of $K$. 

\subsubsection{Obstructions from counting faces}
If the face poset of the cone $K$ is finite, then we can use the order embedding together with a counting 
argument to obtain obstructions to the existence of $K$-lifts. This approach makes sense in the 
case of polyhedral lifts, where $K=\RR_+^m$. The face lattice of $\RR^m_+$ is the \emph{Boolean lattice of order $m$}.
A counting argument gives a slight generalization of the bound of Goemans.
 \begin{lemma}
 \label{lem:logbnd}
 If $C$ has an 
	 $\RR_+^m$-lift, then $m\geq \lceil\log_2(|\mathcal{F}_C|)\rceil$.
 \end{lemma}
 \begin{proof}
	 If $C$ has a $\RR_+^m$-lift then $|\mathcal{F}_C|= |(\psi\circ\phi)(\mathcal{F}_C)|\leq 2^m$ since 
	 $\psi\circ \phi$ is injective and the cardinality of the Boolean lattice of order $m$ is $2^m$. 
 \end{proof}

\subsubsection{Obstructions from chains of faces}

Any order embedding maps a chain of a certain length to a chain of the same length, giving the following 
basic obstruction to the existence of lifts. 
\begin{lemma}
\label{lem:chain-length}
	If $C$ is a convex body (of dimension at least one) and $C$ has a $K$-lift for a closed convex cone $K$, 
	then the longest chain of faces in $C$ is strictly smaller than the 
longest chain of faces in $K$.
\end{lemma}
\begin{proof}
	Since $C$ is a convex body, it is compact. Since $C$ has dimension at least one and is compact, it is not a cone. 
	As such, if $C = \pi(K \cap L)$ then $0\notin L$. Therefore, any 
	chain of faces in $C$ is mapped, by the order embedding $\phi$, to a chain of faces of the 
	same length in $K$ that does not include the face $\{0\}$. In particular a longest chain of 
	faces in $C$ is mapped to a chain of faces in $K$ that is not of maximum length. 
\end{proof}
Since the empty set is always a face of a convex set, we focus the discussion on chains of non-empty faces. 
\begin{example}[Convex hull of a cardiod]
Consider the convex body 
	\[C = \textup{conv}\{(x,y)\in \RR^2\;:\; 4x^4 + 32y^4 + 13x^2y^2+18xy^2-4x^3-27y^2=0\},\]
	the convex 
hull of the cardiod shown on the right in Figure~\ref{fig:polar}. 
In Example~\ref{eg:cardiod-lift} we saw that this convex body has a $\mathcal{S}_+^3$-lift. 

The convex hull of a cardiod has a length three chain of non-empty faces given by 
\[ \{(-1,\textstyle{\frac{1}{\sqrt{2}}})\} \subseteq [ (-1,-\textstyle{\frac{1}{\sqrt{2}}}), (-1,\textstyle{\frac{1}{\sqrt{2}}})] \subseteq C.\]
From Lemma~\ref{lem:chain-length}, we know that if the convex hull of a cardiod is to 
	have a $K$-lift, then $K$ must have a chain of non-empty faces of length 
at least four. Since $\mathcal{S}_+^2$ does not have a chain of non-empty faces of length four, we can conclude that 
the convex hull of a cardiod does not have an $\mathcal{S}_+^2$-lift. In fact, for any smooth convex cone, 
the longest chain of non-empty faces has length three (since the only faces are $\{0\}$, the extreme rays, and 
the cone itself). As such, we can conclude that the convex hull of the cardiod does not have a $K$-lift where $K$ is a smooth 
convex cone. 
\end{example}

 Obstructions based on the length of chains of faces 
can be used to give simple lower bounds on the size of spectrahedral lifts of polytopes.
 \begin{corollary}[{\cite{GPTlifts}}]
 \label{cor:lbchain-poly}
 Any spectrahedral lift of a polytope $C$ has size at least $\textup{dim}(C)+1$.
 \end{corollary}
 \begin{proof}
For a polytope $C$, the dimension strictly increases along chains of faces, and  
any maximal chain of non-empty faces has length $1+\textup{dim}(C)$.  
 On the other hand, the longest chain of non-empty faces in $\mathcal{S}_+^m$ has length at most $m+1$.
 This is because the rank (which is constant on the relative interior of faces of $\mathcal{S}_+^m$)
	 strictly increases along chains of faces, and the rank is zero on the face $\{0\}$.
\end{proof}

The polytopes $C$ that have spectrahedral lifts of size $\textup{dim}(C)+1$ are 
known as \emph{psd-minimial polytopes}~\cite{gouveia2013polytopes}. 
These are the polytopes for which the lower bound of Corollary~\ref{cor:lbchain-poly} is tight. 
Any $2$-level polytope (see Section~\ref{sec:k-level}) is psd-minimal. In particular, 
the stable set polytopes of perfect graphs from Section~\ref{sec:examples}, are examples of psd-minimal polytopes. There is a complete classification of psd-minimal polytopes in dimensions up to four~\cite{gouveia2017four},
obtained by a careful study of certain algebraic properties of their slack matrices. 

\subsubsection{Obstructions based on neighborliness}

We now consider more elaborate features of the facial structure of convex bodies that, if present, can 
rule out much larger classes of lifts. These features build on the idea of neighborly 
polytopes.
\begin{definition}
	A convex polytope $C$ is \emph{$k$-neighborly} if, for every subset of $k$ extreme points of $C$, 
	the convex hull of that set forms a face of $C$.
\end{definition}
Perhaps the most celebrated examples of neighborly polytopes are the cyclic polytopes~\cite{gale1963neighborly}. 
These are obtained by taking the convex hull of $v$ points on the moment curve $(t,t^2,\ldots,t^n)$.
Neighborly polytopes are extremal in the sense that among all $n$-dimensional polytopes with $v$ vertices, 
the neighborly polytopes have the largest number of faces in each dimension.
This property makes them central objects in polyhedral combinatorics. They also provide a 
geometric interpretation of key results in other contexts, such as compressive sensing~\cite{donoho2005sparse}.

Any face of a polytope is exposed by a linear functional. As such, a convex polytope $C$ 
is $k$-neighborly if, and only if, for every $I\subseteq \textup{ext}(C)$ with $|I|=k$, there is a 
linear functional $\ell_I\in C^\circ$ that exposes exactly the extreme points in $I$ 
(i.e., $\ell_I(x) = 1$ for $x\in I$ and $\ell_I(x)<1$ for $x\in \textup{ext}(C)\setminus I$).  
We can generalize this definition to non-polyhedral convex bodies,
and to strict subsets of the extreme points. This property turns out to provide 
an obstruction to certain lifts of non-polyhedral convex bodies.
\begin{definition}
	\label{def:knV}
 	A convex body $C$ is \emph{$k$-neighborly with respect to $V\subseteq \textup{ext}(C)$} if, for every $I\subset V$ with $|I|=k$
	there exists $\ell_I\in C^\circ$ such that $\ell_I(x) = 1$ for $x\in I$ and $\ell_I(x) < 1$ for all $x\in V\setminus I$. 
\end{definition}
One can check that a $k$-neighborly polytope is $k$-neighborly with respect to its entire set of extreme points. 
However, this definition is most interesting beyond the polyhedral setting. 
\begin{example}
	\label{eg:sdptr1}
	The set of $3\times 3$ positive semidefinite matrices with trace one, sometimes called the $3\times 3$ \emph{spectraplex}, is $2$-neighborly with respect to the (countably infinite) set of extreme points
\[ V = \left\{v_pv_p^\top\;:\;\textup{$p$ is an integer}\right\}\quad\textup{where}\quad
	v_p = \frac{1}{\sqrt{1+p^2+p^4}}\begin{bmatrix} 1 \\p \\p^2\end{bmatrix}.\]
To see this we choose any pair $\{i,j\}$ of integers and define the linear functional $\ell_{\{i,j\}}$ in the polar of the spectraplex by 
	\[ \ell_{\{i,j\}}(X) = \tr(X) - w_{ij}^\top Xw_{ij} \quad\textup{where}\quad w_{ij} = \begin{bmatrix} ij\\-(i+j)\\1\end{bmatrix}.\]
		Since $w_{ij}^\top v_p = (p-i)(p-j)/\sqrt{1+p^2+p^4}$, it follows that
$\ell_{\{i,j\}}(v_pv_p^\top) = 1 - \frac{1}{1+p^2+p^4}(p-i)^2(p-j)^2$
which clearly satisfies Definition~\ref{def:knV}.
By a similar argument, the $(k+1)\times (k+1)$ spectraplex is 
$k$-neighborly with respect to an
infinite set of extreme points.
\end{example}

Convex bodies with neighborliness properties arise very naturally when considering convex reformulations
of polynomial optimization problems from the point of view taken in Section~\ref{sec:introduction}. 
If $S\subseteq\RR^n$ is compact and full-dimensional, we have seen that optimization of 
polynomials of degree $2d$ over $S$ can be rephrased as linear optimization over the convex body
\[ C = \textup{cl}\,\textup{conv}\{(x^\alpha)_{|\alpha|\leq 2d}\;:\; x\in S\}.\]
The extreme points all have the form $(x^\alpha)_{|\alpha|\leq 2d}$ for $x\in S$. 
The polar, $C^\circ$, is exactly the set of coefficients of polynomials of degree $2d$ 
that take value at most one on $S$. 
Averkov~\cite{averkov2019optimal}, in a considerable generalization of 
Example~\ref{eg:sdptr1}, showed that for any positive integer $M \geq \left(\binom{n+d}{n}-1\right)$,
there is a subset $V_M$ of extreme
points of $C$ with $|V_M|=M$, and such that $C$ is $\left(\binom{n+d}{n}-1\right)$-neighborly 
with respect to $V_M$. In other words, $C$ is $\left(\binom{n+d}{n}-1\right)$-neighborly with 
respect to arbitrarily large finite subsets of extreme points. 


Being $k$-neighborly with respect to arbitrarily large finite sets of extreme points turns out to 
be an obstruction to $K$-lifts where $K$ is a Cartesian product of `low-complexity' convex cones. 
The first results in this direction 
considered the expressive power of second-order cone lifts~\cite{fawzi2018representing}. 
Recall that these are $K$-lifts in which $K$ is a finite 
Cartesian product of second-order cones, i.e., cones of the form
\[ \mathcal{L}_+^{\ell+1} = \{(x_0,x)\in \RR\times \RR^{\ell}\;:\; \|x\| \leq x_0\}.\]
\begin{theorem}[{\cite{fawzi2018representing}}]
	\label{thm:fawzi-socp}
 	Let $C$ be a convex body that is $2$-neighborly 
	with respect to arbitrarily large finite sets of extreme points. 
	Then $C$ does not have a second-order cone lift.
\end{theorem}
Since the $3\times 3$ spectraplex is $2$-neighborly with respect 
to an infinite set of extreme points, it does not have a second-order cone lift. 
It follows (by a homogenization argument) that $\mathcal{S}_+^3$ does not 
have a second-order cone lift. 

The second order cone $\mathcal{L}_+^{\ell+1}$ has a $(\cS_+^2)^{\ell}$-lift.
	As such, Theorem~\ref{thm:fawzi-socp}
	can also be thought of as giving an obstruction to $K$-lifts where 
	$K$ is a finite Cartesian product of $2\times 2$ positive semidefinite
	cones. Averkov~\cite{averkov2019optimal} extended Theorem~\ref{thm:fawzi-socp} to show that if $m$ is a positive
integer and $C$ is $k$-neighborly with respect to arbitrarily large finite sets of extreme points then 
$C$ does not have a $(\cS_+^k)^m$-lift. This allowed him to show that many convex bodies associated
with convex approaches to polynomial optimization do not have lifts using small positive semidefinite blocks. 

The property that $C$ is $k$-neighborly with respect to arbitrarily large
finite sets of extreme points is really a property of the face lattice of $C$.
As such, we might expect that such a $k$-neighborliness property is an
obstruction to having a $K$-lift for a class of cones $K$ defined purely by
properties of its faces.
\begin{theorem}[{\cite{saunderson2019limitations}}]
\label{thm:nbsau}
Let $m$ be a positive integer and let $K_1,\ldots,K_m$ be proper convex cones such that, for $i=1,2,\ldots,m$, 
the length of the longest chain of non-empty faces of $K_i$ is at most $k+1$. If $C$ is a convex body that 
is $k$-neighborly with respect to arbitrarily large finite sets of extreme points, 
then $C$ does not have a $K_1\times K_2\times \cdots \times K_m$-lift.
\end{theorem}
The length of the longest chain of non-empty faces of the $k\times k$ positive
semidefinite cone is $k+1$. As such, Theorem~\ref{thm:fawzi-socp}, and Averkov's extension, are special
cases of Theorem~\ref{thm:nbsau} where we take $K_1 = K_2 = \cdots = K_m =
\mathcal{S}_+^k$.  Other specializations of Theorem~\ref{thm:nbsau} tell us, for instance, that
\begin{itemize}
	\item the $3\times 3$ positive semidefinite cone has no $K$-lift where 
		$K$ is a finite Cartesian product of smooth convex cones, because any smooth cone only
has chains of non-empty faces of length at most three; and
\item the $4\times 4$ positive semidefinite cone has no lifted representation 
	using a finite Cartesian product of exponential cones, 
		$\mathcal{K}_{\exp} = \textup{cl}\{(x,y,t)\;:\; y > 0,\; y\exp(x/y) \leq t\}$, 
		since the longest chain of non-empty 
	faces in any three-dimensional cone is four.
\end{itemize}
This last example is of interest due to recent developments in 
tractable relaxations of polynomial optimization
problems based on geometric programming~\cite{chandrasekaran2016relative,dressler2017positivstellensatz}. The
connection arises because geometric programs can, themselves, be reformulated
as conic optimization problems over Cartesian products of exponential cones.

\subsection{Algebraic obstructions}
\label{sec:obs-alg}

We have so far discussed obstructions to the existence of lifts based on combinatorial properties of the face lattice of $C$. In this subsection we discuss \emph{algebraic obstructions} to the existence of lifts.



\subsubsection{Preliminaries}
\label{sec:obs-semi-prelim}

A \emph{semialgebraic} subset of $\RR^n$ is a set described using Boolean combinations (finite unions, intersections and complementations) of sets of the form
\begin{equation}
\label{eq:semialgebraic}
\{x \in \RR^n : f_i(x) = 0, g_j(x) \geq 0, \forall i \in I, j \in J\}
\end{equation}
where $f_i, g_j$ are polynomials with real coefficients, and $I$ and $J$ are finite index sets. Most subsets of $\RR^n$ we have seen in this article are semialgebraic subsets. For example any \emph{spectrahedron} $S = \{x \in \RR^n : A_0 + x_1 A_1 + \dots + x_n A_n \psd 0\}$ is semialgebraic, since the constraint that a matrix is positive semidefinite can be expressed using polynomial inequalities in the entries of the matrix. A celebrated result of Tarski shows that the projection of any semialgebraic set is semialgebraic.
\begin{theorem}[Tarski]
\label{thm:tarski}
Let $S \subset \RR^\ell$ be a semialgebraic set and $\pi:\RR^\ell \rightarrow \RR^n$ be a linear map. Then $\pi(S)$ is semialgebraic.
\end{theorem}

A consequence of Tarski's theorem is that  \emph{if $C \subset \RR^n$ has a spectrahedral lift, then $C$ must be semialgebraic.} This already gives us one obstruction to the existence of spectrahedral lift. For example it tells us that a set like $\{(x,y) \in \RR^2 : y \geq \exp(x)\}$ (which is not semialgebraic) cannot have a spectrahedral lift of finite size.



\subsubsection{Degree bounds}

Let $C$ be a convex semialgebraic set.  Define the (topological) boundary of $C$ as $\partial C = (\cl C) \setminus (\interior C)$. Assuming that our convex set $C$ is full-dimensional in $\RR^n$, it can be shown that the boundary is a semialgebraic set of dimension $n-1$ \cite{sinn}, i.e., it is a hypersurface. This means that there exists a nonzero polynomial $p \in \RR[x_1,\ldots,x_n]$ such that $\partial C \subset \Z(p)$ where $\Z(p) = \{x \in \RR^n : p(x) = 0\}$ is the zero set of $p$. The set $\Z(p)$ is known as the \emph{Zariski closure} of $\partial C$. The \emph{degree} of $\partial C$ is the smallest degree of a polynomial $p$ such that $\partial C \subset \Z(p)$; and the \emph{algebraic boundary} of $C$ is defined as $\Z(p)$. Below we give several examples of convex sets for which we describe the algebraic boundary and the corresponding degree.

\begin{example}[Disk]
Consider the convex set $C = \{(x,y) \in \RR^2 : x^2 + y^2 \leq 1\}$ which has the $2\times 2$ spectrahedral representation
\[
C = \left\{ (x,y) \in \RR^2 : \begin{bmatrix} 1-x & y\\ y & 1+x \end{bmatrix} \psd 0 \right\}.
\]
The topological boundary of this convex set is the circle $\{(x,y) \in \RR^2 : x^2 + y^2 = 1\}$ which is described by a polynomial of degree two. Thus the degree of $\partial C$ is equal to two.
\end{example}

\begin{example}[Oval]
Consider the two-dimensional convex set defined by
\begin{equation}
\label{eq:ovalC}
C = \left\{ (x,y) \in \RR^2 : \begin{bmatrix} x & 0 & y\\
0 & 1 & -x\\
y & -x & 1\end{bmatrix} \psd 0 \right\}
\end{equation}
which corresponds to the blue oval depicted in Figure \ref{fig:oval}. The smallest degree of a polynomial that vanishes on the boundary of $C$ is 3; an example of such a polynomial is $p(x,y) = x - x^3 - y^2 = 0$. The zero set of $p(x,y)$ is shown with a black thick line in Figure \ref{fig:oval}. We see that it has two components. In this example the topological boundary and algebraic boundary are distinct.

\begin{figure}[ht]
  \centering
  \includegraphics[scale=0.7]{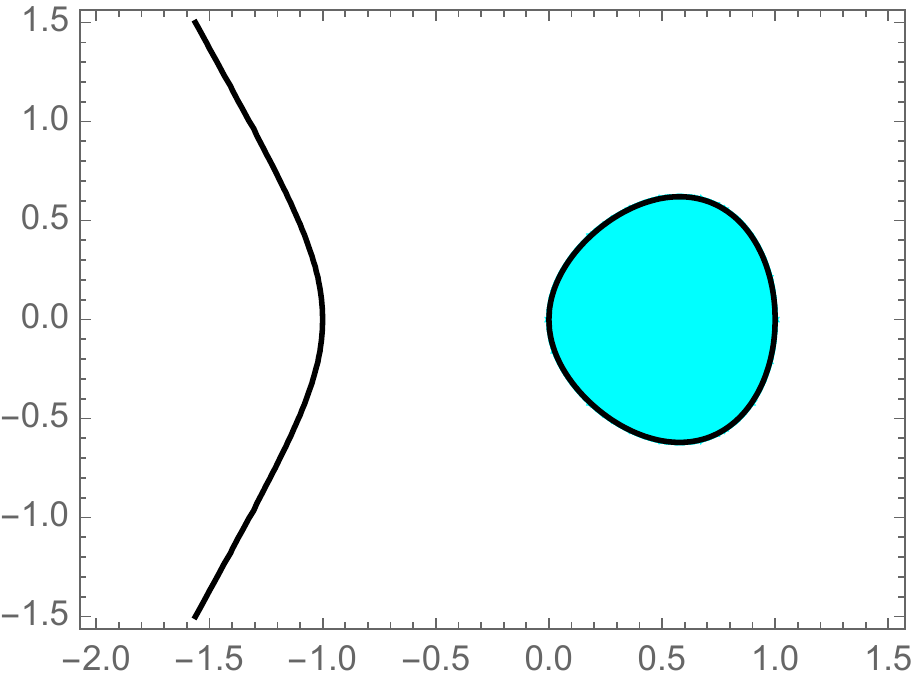}
  \caption{In blue is the convex set of Equation \eqref{eq:ovalC}. The black solid line represents the algebraic boundary of $C$. It has two components. The degree of the algebraic boundary of $C$ is 3.}
  \label{fig:oval}
\end{figure}

\end{example}

\begin{example}[$k$-ellipse]
Given points $p_1,\ldots,p_k \in \RR^2$ in general position, consider the convex set
\begin{equation}
\label{eq:kellipse}
C = \{x \in \RR^2 : \|x-p_1\|_2 + \dots + \|x-p_k\|_2 \leq 1\}.
\end{equation}
For $k=2$, the set $C$ is an ellipse with focal points $p_1$ and $p_2$ and the degree of the boundary of $C$ is equal to 2. For higher $k$, the degree of the boundary of $C$ was computed in \cite{kellipse}, and was shown to be $2^k$ if $k$ is odd, and $2^k - \binom{k}{k/2}$ if $k$ is even. Figure \ref{fig:kellipse} shows an example where $k=3$.
\begin{figure}[ht]
  \centering
  \includegraphics[scale=0.7]{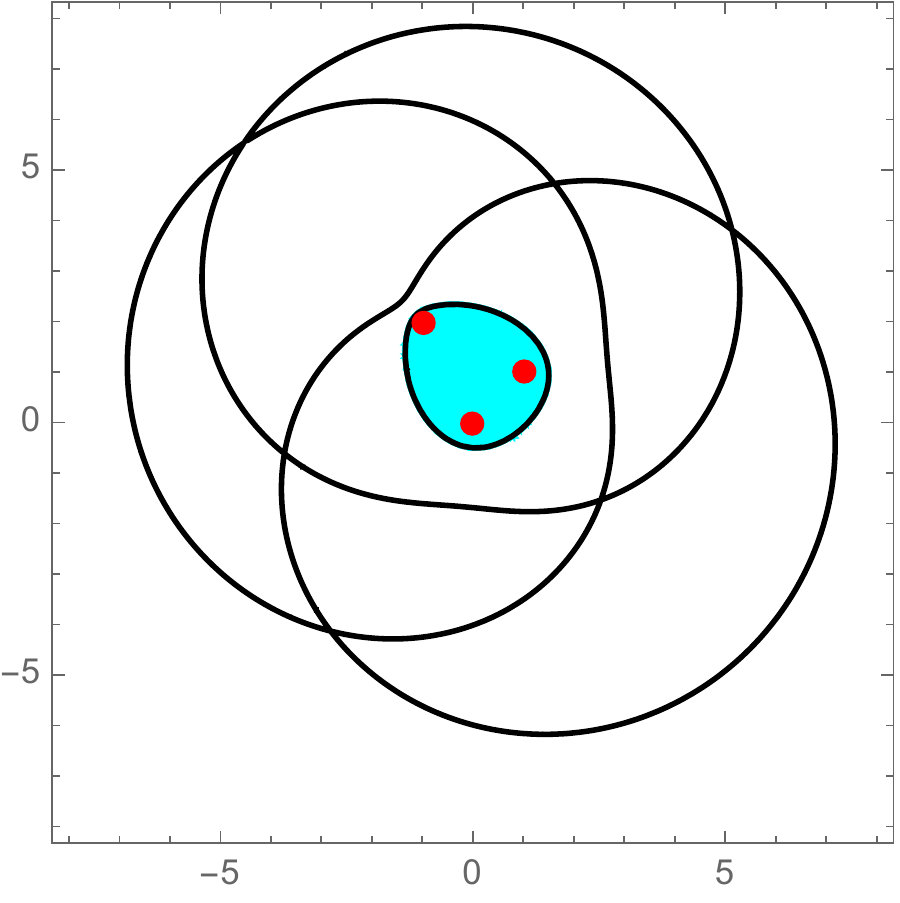}
  \caption{A 3-ellipse (see Equation \eqref{eq:kellipse}) with the focal points shown in red. The smallest degree of a polynomial that vanishes on the boundary of $C$ is 8. The figure shows the zero set of this degree-8 polynomial.}
  \label{fig:kellipse}
\end{figure}
\end{example}

We saw in Theorem \ref{thm:tarski} that the projection of any semialgebraic set is semialgebraic. More explicit formulations of Tarski's result give bounds on the degrees of the polynomials that define the semialgebraic set $\pi(S)$, in terms of the polynomials that define $S$, see, e.g., \cite{RenegarI}. These results can be used to show for example that if $S$ is a spectrahedron defined using matrices of size $m\times m$, then the degree of the boundary of $\pi(S)$ is at most $m^{O(m^2 n)}$ \cite{GPTlifts} where $n = \dim(\pi(S))$. It follows immediately that if $C$ is a convex body of dimension $n$ having a spectrahedral lift of size $m$ then necessarily $m^{O(m^2 n)} \geq d$, where $d$ is the degree of the boundary of $C$.
Inverting this inequality allows us to get a lower bound on $m$ in terms of $d$ and $n$ which reads 
\begin{equation}
\label{eq:mOmega123}
m \geq \Omega\left(\sqrt{\frac{\log d}{n \log \log (d/n)}}\right)
\end{equation}
where $\Omega$ hides a constant.

The previous lower bound on $m$ used very little about the structure of the spectrahedron $S$ living in higher dimensions and that projects onto $C$. A more refined analysis allows us to get the better lower bound \cite{FSED}:
\begin{equation}
\label{eq:degreebnd}
m \geq \sqrt{\log d}.
\end{equation}
This bounds relies on the specific structure of the lift and uses Karush-Kuhn-Tucker system of optimality conditions from convex optimization, combined with the B\'ezout bound to count zeroes of systems of polynomial equations.

Observe that the bound \eqref{eq:degreebnd} only depends on the degree $d$ of the boundary of $C$ and does not depend, say, on the dimension of $C$. One can easily show that no convex body can have a spectrahedral lift of size smaller than $\sqrt{\dim(C)}$. Indeed if $C = \pi(S)$ and $S$ is defined using a linear matrix inequality of size $m\times m$ then we must have 
\begin{equation}
\label{eq:dimbnd}
\dim(C) \leq \dim(S) \leq \dim \S^m = \binom{m+1}{2} \leq m^2,
\end{equation}
which gives $m \geq \sqrt{\dim(C)}$.

We note that the two bounds presented above, based on degree and dimension, are not in general comparable:
\begin{itemize}
\item Consider the regular $N$-gon in the plane. The degree of the boundary is $N$ and so the lower bound \eqref{eq:degreebnd} gives $m \geq \sqrt{\log N}$. On the other hand the bound \eqref{eq:dimbnd} gives $m \geq \sqrt{2}$.
\item Now consider the unit Euclidean ball in $\RR^n$. The degree of the boundary is equal to two, since the boundary is given by the degree-two polynomial equation $x_1^2 + \dots + x_n^2 - 1 = 0$. Thus the lower bound \eqref{eq:degreebnd} gives $m \geq \sqrt{2}$. On the other hand the dimension lower bound gives $m \geq \sqrt{n}$.
\end{itemize}
One interesting open question is to find a bound combining both degree $d$ and dimension $\dim(C)$ which improves on \eqref{eq:mOmega123}.

\subsubsection{Scheiderer's result}

In the previous section we saw, using Tarski's theorem, that in order for a convex set $C$ to have a spectrahedral lift, it must be semialgebraic. Is this condition sufficient, i.e., does every convex semialgebraic set have a spectrahedral lift? Nemirovski raised this question in his ICM survey \cite{nemirovskiICM} in 2006. In \cite{helton2009sufficient,helton2010semidefinite} Helton and Nie proved a series of results showing that every convex semialgebraic set whose boundary satisfies certain smoothness conditions admits a spectrahedral lift. They also conjectured that \emph{any} convex semialgebraic set has a spectrahedral lift. In a breakthrough result, Scheiderer \cite{scheidererSDR} disproved this conjecture, and showed the existence of convex semialgebraic sets with no spectrahedral lift. 

In this section we give a simple proof of Scheiderer's result based on the arguments in \cite{fawzi2019separable}. For concreteness we will focus on the following convex set:
\[
C_{n,2d} = \cl \conv \left\{ (x^{\alpha})_{|\alpha| \leq 2d} : x \in \RR^n \right\} \]
where $n$ and $d$ are two integers. Note that $C_{n,2d}$ is a convex semialgebraic set since it is the convex hull of a semialgebraic set. We will show:
\begin{theorem}
\label{thm:scheiderer}
Assume $n$ and $d$ are two integers such that there is a \emph{homogeneous} polynomial of degree $2d$ in $n$ variables that is not a sum of squares. Then $C_{n,2d}$ has no spectrahedral lift.
\end{theorem}
The theorem above shows in particular that $C_{3,6}$ and $C_{4,4}$ have no spectrahedral lift since it is known by a result of Hilbert that in the cases $(n,2d) = (3,6)$ and $(n,2d) = (4,4)$ there exist nonnegative polynomials that cannot be written as a sum of squares, see, e.g., \cite[Chapter 4]{BPTSIAMBook}.

We remark that the assumption in the theorem above can be used to prove that $C_{n-1,2d}$ has no spectrahedral lift (which readily implies that $C_{n,2d}$ has no spectrahedral lift, since $C_{n-1,2d}$ is a projection of $C_{n,2d}$), but the argument is more technical and uses additional properties about semialgebraic functions; we refer to \cite{fawzi2019separable} for more details.

\begin{proof}[Sketch of proof of Theorem \ref{thm:scheiderer}]
Let $n$ and $d$ as in the statement of the theorem, and assume for contradiction that $C_{n,2d}$ has a spectrahedral lift. We know from Section \ref{sec:constructions} (see Equation \eqref{eq:1-ellSOSV}) that there is a finite-dimensional subspace of functions that can be used to certify all valid linear inequalities on $C_{n,2d}$ using sums of squares. Linear inequalities valid on $C_{n,2d}$ are nothing but nonnegative polynomials of degree $2d$ in $n$ variables. Thus the existence of a spectrahedral lift for $C_{n,2d}$ means that there is a finite number of functions $f_1,\ldots,f_m : \RR^n \rightarrow \RR$ such that any nonnegative polynomial in $\RR[x_1,\ldots,x_n]_{\leq 2d}$ is a sum of squares from $\linspan(f_1,\ldots,f_m)$.

A key observation here is that the functions $f_i$ can be taken to be \emph{semialgebraic}. A semialgebraic function is one whose graph $\{(x,f(x)) : x \in \RR^n\} \subset \RR^{n+1}$ is a semialgebraic set. The observation about the $f_i$ can be shown by going back to the proof of the factorization theorem (Theorem \ref{thm:GPTliftstheorem}) and using some standard results from semialgebraic geometry, e.g., the cylindrical algebraic decomposition of semialgebraic sets, see \cite{fawzi2019separable} for details. Semialgebraic functions are tame and possess nice qualitative properties. The following fact will be crucial for our argument (see, e.g., \cite[Theorem 1.7]{son2016genericity} for a proof).

\begin{theorem}
A semialgebraic function $f:\RR^n \rightarrow \RR$ is smooth ($C^{\infty}$) almost everywhere.
\end{theorem}

We now have all the ingredients to prove our result. 
Since the $f_i$ ($i=1,\ldots,m$) are smooth almost everywhere there is at least one point $a \in \RR^n$ such that all the $f_i$ are smooth at the point $a$.


Let $p$ be a homogeneous polynomial in $n$ variables of degree $2d$ that is nonnegative but not a sum of squares.
Consider the shifted polynomial $p(x+a)$. Since $p(x+a)$ is in $\RR[x_1,\ldots,x_n]_{\leq 2d}$ and nonnegative it must be a sum of squares from $\linspan(f_1,\ldots,f_m)$. Shifting back to the origin, this means that $p(x)$ is a sum of squares from $\linspan(\tilde{f_1},\ldots,\tilde{f_m})$ where $\tilde{f_i}(x) = f_i(x-a)$. Note that all the functions in $\linspan(\tilde{f_1},\ldots,\tilde{f_m})$ are $C^{\infty}$ at the origin. We now use the following proposition:
\begin{proposition}
\label{prop:sostaylor}
Let $p$ be a \emph{homogeneous} polynomial of degree $2d$ in $n$ variables, and assume that $p = \sum_{j} h_j^2$ where the $h_j$ are arbitrary functions that are smooth ($C^{\infty}$) at the origin. Then $p$ is a sum of squares of polynomials.
\end{proposition}
\begin{proof}
We can write a Taylor expansion of each $h_j$ of order $d$: $h_j(x) = q_j(x) + o(\|x\|^{d})$ where $\deg q_j = d$. Let $\alpha_j = \mindeg q_j \leq d$ (where $\mindeg q_j$ is the smallest degree of a monomial in $q_j$) and note that $h_j(x)^2 = q_j(x)^2 + o(\|x\|^{d+\alpha_j})$. Then we get
\[
p(x) = \sum_{j} q_j(x)^2 + o(\|x\|^{d+\alpha_j}).
\]
Note that since $p$ is homogeneous we must have $\alpha_j = d$ for all $j$, since otherwise there is a term of degree $2\alpha_j$ in $q_j(x)^2$ that cannot be canceled by the other terms. This means that the $q_j$ are homogeneous of degree $d$. But then this means that we must have exact equality $p(x) = \sum_{j} q_j(x)^2$ and the remainder term $o(\|x\|^{2d})$ is identically zero. This shows that $p(x)$ is a sum of squares of polynomials.
\end{proof}
\textit{End of proof of Theorem \ref{thm:scheiderer}:} The proposition above implies that $p$ is a sum of squares of polynomials, which is a contradiction. We have thus shown that $C_{n,2d}$ has no spectrahedral lift.
\end{proof}

\section{Discussion}
\label{sec:discussion}


\newcommand{\parheading}[1]{\textbf{#1} \quad}

In this article we have focussed on understanding exact lifts of
convex sets. We saw several scenarios in which lifted representations
greatly simplify the description of a convex set with applications to
optimization and other fields. We also discussed in detail the key
tool of slack matrices and operators that both determine the existence
of lifts into specific cones and provide a method to construct
them. Various construction methods for lifts, as well as obstructions
to their existence, were also discussed. The main focus in these last
two sections was on spectrahedral lifts. The purpose of this article was 
to show the broad connections of lifts to mathematics at large with a focus on spectrahedral 
lifts in the second half. A short and focussed exposition on just spectrahedral lifts can be found in \cite{thomasicmarticle}.

There is much more that one can say and study in and around the theory of lifts, 
and in this final section we collect together several of these directions with pointers to the literature. 
We highlight the surveys and expository articles in each area. 
These aspects have not been discussed in this article but knit together a more complete view of this rich 
subject. They also showcase the wide ranging connections between the theory of lifts and other disciplines. 
 
\parheading{Integer programming/Combinatorial optimization.}~In this
article we have emphasized recent developments in spectrahedral lifts
of convex semialgebraic sets, at the expense of more classical topics
like polyhedral lifts of polytopes, usually called \emph{extended
formulations}. The study of such extended formulations has a long
history, motivated by their importance in combinatorial optimization.
A broad exposition of extended formulations of polytopes that arise in
combinatorial optimization can be found in
\cite{conforti-cornuejols-zambelli2013}. For a survey of mixed integer
linear programming reformulations, see \cite{vielma-sirev}.



\parheading{Lower bounds for polyhedral lifts.}~The main result of the seminal paper of Yannakakis \cite{yannakakis1991expressing} that was mentioned in Section~\ref{sec:slackoperator}, 
was that the {\em perfect matching polytope} of a complete graph does not admit a small {\em symmetric} polyhedral lift. 
Symmetric lifts were discussed in Section~\ref{sec:constructions}. As was mentioned there already,  
Yannakakis asked whether symmetry imposes restrictions on the linear 
extension complexity of a polytope. In \cite{kaibel-pashkovich-theis} Kaibel, Pashkovich and Theis showed that, indeed symmetry can force the size of a lift to be higher than necessary. Their examples are certain slices of the perfect matching polytope of a complete graph.  A friendly introduction to polyhedral lifts including the symmetry questions 
can be found in \cite{KaibelOptima}. 
A second breakthrough result on polyhedral lifts came in 
\cite{fiorini-et-al2015} where the authors showed that {\em cut polytopes} and {\em traveling salesman polytopes} do not admit small polyhedral lifts. 
The proof relied on showing that a submatrix of the slack matrix had exponentially high nonnegative rank. The underlying 
combinatorial problems in these examples, the {\em max cut problem} and the {\em traveling salesman problem}, are NP-hard, and so these results were expected although not known to be true explicitly. On the other hand, there is a polynomial time algorithm to find the maximum size matching in a graph which makes it plausible that the matching polytope has a small polyhedral lift. 
Rothvoss showed in \cite{rothvoss2017} that the matching polytope does not in fact have a small extended formulation. Several results on combinatorial lower bounds on the extension complexity of polytopes and the techniques for establishing such bounds can be found in 
\cite{fiorini-kaibel-pashkovich-theis2013}.

\parheading{Lower bounds for spectrahedral lifts.}~Lower bounds on the semidefinite extension complexity of polytopes have also received attention. In 
\cite{lee-raghavendra-steurer} the authors show that the cut and traveling salesman polytopes of graphs have super polynomial lower bounds 
on their semidefinite extension complexity. As before, such a result is expected but was not known explicitly. It remains 
open whether matching polytopes also have high semidefinite extension complexity. Also no family of polytopes is known 
with a significant gap between linear extension complexity and semidefinite extension complexity; see \cite{fawzi2016sparse} for an explicit example with a nontrivial gap, based on special cyclic polytopes. Semidefinite extension complexity can be studied for convex semialgebraic sets beyond polytopes. Not much is known in this direction beyond the results we have presented in Section~\ref{sec:obstructions and lower bounds} of this paper. 

\parheading{Other cones.}~Lifts of convex sets into cone families beyond nonnegative orthants, second order cones  and psd cones have also been studied. 
For instance, a striking result in this direction is that every polytope obtained as the convex hull of a collection of $0/1$ vectors in $\RR^n$ 
admits a lift into the {\em completely positive cone} of size $n+1$ \cite{burer}. A symmetric matrix $M$ is {\em copositive} if $x^\top M x \geq 0$ for all $x \geq 0$ and the {\em copositive cone} of size $k$ is the set of all $k \times k$ copositive matrices. Its dual cone is the {\em completely positive cone} and consists of all symmetric matrices of the form $BB^\top$ where $B$ is a nonnegative matrix. While the lifts mentioned above are small in terms of the size of matrices involved, linear programming over copositive and completely positive cones do not admit efficient algorithms unless $\textup{P}=\textup{NP}$ \cite{dickinson-gijben,murty-kabadi}. 
There are other families of convex cones for which the associated conic program enjoys efficient algorithms. Examples include
{\em hyperbolic programming} \cite{Guler1997}, {\em geometric programming} \cite{geometric-programming},  and {\em relative entropy programming} \cite{chandrasekaran-shah}. Studying lifted representations of convex sets using hyperbolicity cones or relative entropy cones gives a systematic approach to understanding the expressive power of these families of optimization problems.

\parheading{Approximate lifts and hierarchies.}~Since finding an exact lift of a convex set might be hard, it is useful to look for an approximate lift in the sense of a set 
in higher dimension whose projection contains the original convex
set. There is a vast body of literature on such approximate lifts, as
well as construction methods for producing a hierarchy of lifts in
increasing dimensions whose projections create a nested sequence of
convex approximations of the original set. The hierarchies most
closely related to the discussion in Section~\ref{sec:constructions}
of this paper are those of Lasserre \cite{Lasserre} based on moments
of probability measures, and its dual due to Parrilo \cite{Parrilo}
based on sums of squares of polynomials.  These methods are connected
to deep results in real algebraic geometry, the theory of moments, and
optimization, which give tools to address questions of convergence,
construction, geometry and more. 
See \cite{LaurentSurvey} for a survey
on the Lasserre and Parrilo hierarchies. For a discussion of some of the central hierarchies for convex
approximations of a set and their relationships,
see \cite{LaurentComparison}.  See \cite{BPTSIAMBook} for
broad connections to convex algebraic geometry and related topics.
The chapter \cite{gouveiathomaschapter} in the Handbook of Semidefinite, Conic and Polynomial Optimization is about 
{\em theta bodies} of algebraic varieties which is a specific hierarchy of spectrahedral lifts that successively approximate the 
convex hull of an algebraic variety. Theta bodies are also the subject of the chapter titled {\em Convex Hulls of Algebraic Sets} in \cite{BPTSIAMBook}. The tight connection between conic lifts of convex semialgebraic sets and cone factorizations was inspired by the work on theta bodies and the paper of Yannakakis \cite{yannakakis1991expressing} that connects polyhedral lifts to nonnegative factorizations.

\parheading{Approximation algorithms.}~Hierarchies of convex relaxations have been studied extensively in theoretical computer science from the point of view of approximation algorithms. The most celebrated result in this context is perhaps the rounding algorithm for the max cut problem 
due to Goemans and Williamson, that yields an $\alpha$-approximation with $\alpha \approx 0.878$ \cite{GoemansWilliamson}. This is based on a simple spectrahedral relaxation 
of the cut polytope. 
Recently, hierarchies have been used to produce lower bounds on the approximation power of arbitrary polynomial-size linear (semidefinite) programming relaxations of hard combinatorial problems; see e.g.\ \cite{CLRS,lee-raghavendra-steurer} and the references therein. 

\parheading{Computational complexity.}~A natural question is to ask: what is the complexity of deciding whether a convex set $C$ admits a $K$-lift?\footnote{We thank one of the reviewers for raising this question.} To make the question precise, one needs to specify how the input convex set $C$ is provided, e.g., as a set of extreme points, a set of valid inequalities, membership/separation oracle, etc., and similarly for the cone $K$.	Existing results on the hardness of computing the nonnegative and positive semidefinite rank of matrices \cite{vavasis,shitov2017complexity} can be used to show some NP-hardness results on the existence of lifts. For example, one can show that if $P \subset Q$ are two polytopes where $P$ is given by its set of vertices, and $Q$ by its set of facets, the question of deciding whether there is a polytope $S$ with an $\mathbb{R}^r_+$-lift such that $P \subset S \subset Q$ is NP-hard (here $r \in \mathbb{N}$ is part of the input). The same problem is also NP-hard if we replace $\RR^r_+$ by $\S^r_+$ lift, and is a consequence of the NP-hardness of computing the positive semidefinite rank of matrices \cite{shitov2017complexity}. These hardness results follow from the fact that any nonnegative matrix $M$ can be interpreted as the \emph{generalized slack matrix/operator} (see Prop. \ref{prop:nestedconvexsets})  of two such polytopes $P,Q$ by a simple rank factorization of $M$, see, e.g., \cite{gillis2012geometric}. As far as we know the same question where $P=Q$ is given by a $\mathcal{V}$ or  $\mathcal{H}$-representation is open, as not every nonnegative matrix is the slack matrix of a polytope \cite{gouveia2013nonnegative}.

\parheading{Other applications of factorizations.}~
Factorizations of nonnegative matrices
  have many uses beyond the theory of lifts. Nonnegative factorizations of matrices is an essential tool for dimensionality reduction. In this case one is often interested in approximate factorizations, rather than exact ones, see the chapter \cite{Gillis} for applications of nonnegative factorizations to data analysis. Nonnegative factorizations also have interpretations in probabilistic modeling, information theory and communication complexity, see \cite{cohen1993nonnegative} and \cite{fiorini-et-al2015}. Positive semidefinite factorizations of nonnegative matrices in turn have applications in quantum information theory and quantum communication complexity; these are explained in \cite{psdranksurvey} and \cite{fiorini-et-al2015}.  The survey \cite{psdranksurvey} is about psd factorizations and psd ranks of general nonnegative matrices. It does not focus on slack matrices which are the nonnegative matrices needed for lifts. 
  
Symmetric factorizations also have applications. Recall that a matrix $M$ is completely positive if there exists vectors 
$\{b_i\} \subset \RR^k_+$ such that $M_{ij} = \langle b_i, b_j \rangle$. The {\em completely positive rank} of $M$ is the smallest $k$ 
for which such symmetric factorization is possible. Analogously, 
a square matrix $M \in \RR^{p\times p}_+$ is said to be {\em psd-completely-positive} if there exist psd matrices $A_1,\ldots,A_p$ (of arbitrary size) such that $M_{ij} = \tr (A_i  A_j )$ for all $i,j=1,\ldots,p$. These factorizations have found many applications in quantum information, in the context of nonlocal games, see, e.g., \cite{gribling2018bounds}. 
Tools have recently been developed to bound both the completely positive and {\em psd-completely-positive-rank}
 of matrices \cite{gribling-delaat-laurent2017,gribling-delaat-laurent2019}.




\bibliographystyle{alpha}
\bibliography{sirev}

\end{document}